\documentclass[letterpaper,11pt,oneside,reqno]{article}

\usepackage[pdftex,backref=page,colorlinks=true,linkcolor=blue,citecolor=red]{hyperref}
\usepackage[alphabetic,nobysame]{amsrefs}

\usepackage{amsmath,amssymb,amsthm,amsfonts,mathtools}
\usepackage{graphicx,color}
\usepackage{upgreek}
\usepackage[mathscr]{euscript}

\allowdisplaybreaks
\numberwithin{equation}{section}

\usepackage{tikz}
\usetikzlibrary{shapes,arrows,positioning,decorations.markings}

\usepackage[centertableaux]{ytableau}
\usepackage{array}
\usepackage{adjustbox}
\usepackage{cleveref}
\usepackage{enumerate}
\usepackage{subcaption}

\usepackage[DIV=12]{typearea}

\newcommand{\ssp}{\hspace{1pt}}
\newcommand{\nwsymbol}{%
  \begin{tikzpicture}[baseline=(base)]
    \coordinate (base) at (0,0);
    \draw[line width=1pt] (0,0) |- (0.3em,0.4em);
  \end{tikzpicture}%
}
\newcommand{\sesymbol}{%
	\begin{tikzpicture}[baseline=(base)]
		\coordinate (base) at (0,0);
		\draw[line width=1pt] (0,0)--++(0.1em,0)--++(0,.2em)--++(.2em,0)--++(0,.2em);
	\end{tikzpicture}%
}
\newcommand{\SYT}{\mathrm{SYT}}
\newcommand{\SSYT}{\mathrm{SSYT}}

\newsavebox{\captionnw}
\sbox{\captionnw}{%
	\begin{tikzpicture}[baseline=(base)]
    \coordinate (base) at (0,0);
    \draw[line width=1pt] (0,0) |- (0.3em,0.4em);
  \end{tikzpicture}
}
\newsavebox{\captionse}
\sbox{\captionse}{%
	\begin{tikzpicture}[baseline=(base)]
		\coordinate (base) at (0,0);
		\draw[line width=1pt] (0,0)--++(0.1em,0)--++(0,.2em)--++(.2em,0)--++(0,.2em);
	\end{tikzpicture}
}

\newtheorem{proposition}{Proposition}[section]
\newtheorem{lemma}[proposition]{Lemma}
\newtheorem{corollary}[proposition]{Corollary}
\newtheorem{theorem}[proposition]{Theorem}
\theoremstyle{definition}
\newtheorem{definition}[proposition]{Definition}
\newtheorem{remark}[proposition]{Remark}

\newcommand{\la}{\lambda}
\newcommand{\al}{\alpha}
\DeclareMathOperator{\sgn}{sgn}

\newcommand{\defn}[1]{\emph{\textcolor{blue}{#1}}}

\newcommand{\loza}[2]{\draw (#1,#2/0.8660254)--++(1,0.57735027)--++(0,-1.15470054)--++(-1,-0.57735027)--cycle;}
\newcommand{\lozb}[2]{\draw (#1,#2/0.8660254)--++(-1,0.57735027)--++(0,-1.15470054)--++(1,-0.57735027)--cycle;}
\newcommand{\lozh}[2]{\draw[fill=cyan] (#1,#2/0.8660254)--++(1,0.57735027)--++(1,-0.57735027)--++(-1,-0.57735027)--cycle;}

\begin{document}
\title{Hook-length Formulas for Skew Shapes\\via Contour Integrals and Vertex Models}


\author{Greta Panova and Leonid Petrov}

\date{}

\maketitle

\begin{abstract}
	The number of standard Young tableaux of a skew shape  $\lambda/\mu$ can be computed as a sum over excited diagrams inside $\lambda$. Excited diagrams are in bijection with certain lozenge tilings,  with flagged semistandard tableaux and also nonintersecting lattice paths inside $\lambda$. We give two new proofs of a multivariate generalization of this formula, which allow us to extend the setup beyond standard Young tableaux and the underlying Schur symmetric polynomials. The first proof uses multiple contour integrals. The second one interprets excited diagrams as configurations of a six-vertex model at a free fermion point, and derives the formula for the number of standard Young tableaux of a skew shape from the Yang-Baxter equation.
\end{abstract}

%
%

\section{Introduction}\label{sec:intro}

The \defn{hook-length formula} of Frame-Robinson-Thrall~\cite{FRT1954hook} for the number of standard Young tableaux often goes with the adjective ``celebrated'': it is a remarkably rare phenomenon for a class of partially ordered sets to have a product formula for the number of their linear extensions. For a partition $\la$, the number $f^\la$ of standard Young tableaux (SYT) of shape $\la$ is given by 
\begin{align}\label{eq:syt_hook}
    f^\la = |\la|! \prod_{(i,j) \in \la} \frac{1}{h(i,j)}, \tag{HLF}
\end{align}
where $h(i,j) = \la_i+\la'_j-i-j+1$ is the hook length of the box $(i,j)$ in the Young diagram of $\la$. 
This formula has seen many different proofs, from combinatorial to probabilistic and algebraic, each bringing out different ideas and properties.  

\medskip
The immediate generalization of standard Young tableaux, the skew standard Young tableaux, do not have such nice product formulas. The number $f^{\lambda/\mu}$
of skew standard Young tableaux of shape $\lambda/\mu$ is usually represented via determinants or sums of weighted Littlewood-Richardson coefficients. Ten years ago Naruse~\cite{naruse2014schubert}, following work in~\cite{ikeda2009excited}, announced a remarkable formula, which directly generalizes~\eqref{eq:syt_hook}:
\begin{align}\label{eq:skeq_syt_hook}
    f^{\la/\mu} = |\la/\mu|! \sum_{D \in \mathcal{E}(\la/\mu)} \ssp
		\prod_{(i,j) \in \la \setminus D} \frac{1}{h(i,j)}, \tag{NHLF}
\end{align}
where $\mathcal{E}(\la/\mu)$ is the set of so called \defn{excited diagrams} of $\mu$ inside $\la$, and $h(i,j)$ is the hook length of the box $(i,j)$ within $\la$. The origins of this formula lay within equivariant Schubert calculus. 

\medskip
Formula \eqref{eq:skeq_syt_hook} attracted a lot of attention with its elegance and prompted a flurry of activity bringing in various proofs (including \cite{MPP1,MPP2,konvalinka2020bijective,Darij_Naruse2023}), generalizations (among them \cite{naruse2019skew,morales2020okounkov,morales2023minimal,suzuki2021hook,park2021naruse,kirillov2019hook,MPP4GrothExcited}),  wide-ranging applications (see e.g. \cite{HKYY2019reverse,jiradilok2023roots,morales2018asymptotics,felder2023hypergeometric,chan2021sorting, pak2021skew}) and other variations on the theme (e.g. \cite{konvalinka2020hook, morales2023minimal}). Its multivariate version appeared in the proofs and applications of \cite{MPP2, MPP3} in the context of lozenge tilings, but in its most explicit and general form it was stated and proved via elaborate but elementary combinatorial manipulations in \cite{Darij_Naruse2023}:
\begin{equation}\label{eq:general}
    \sum_{T \in \ssp\SYT(\la/\mu)}  \prod_{k=1}^{|\la/\mu|} \frac{1}{z(T^{-1}[\geq k])} = \sum_{D \in \mathcal{E}(\la/\mu)} \prod_{u \in \la\setminus D} \frac{1}{z( H(u) )}, 
\end{equation}
where $z(D) = \sum_{u \in D} z_{c(u)}$ for every excited diagram $D$, $H(u)$ is the hook of $u$ within $\lambda$, and $T^{-1}[\geq k]=\{ u \in \la/\mu, T(u) \geq k\}$ is the set of boxes in $T$ occupied by entries $\geq k$.  In the case of $\mu=\varnothing$ this formula is due to Pak-Postnikov (see~\cite{Darij_Naruse2023} for a detailed account). 
Setting all $z_i=1$ recovers~\eqref{eq:skeq_syt_hook}.

\medskip
In the present work, we give two completely self-contained
and short proofs of formula \eqref{eq:general}, which are
 different in nature from the approaches so far and are not combinatorial. Our
central identity in \Cref{thm:general} below is
equivalent\footnote{By setting $z_{-n+i}=t_{i+1} -t_{i}$,
where $n=\ell(\lambda)$, see \Cref{prop:equivalence}.} to~\eqref{eq:general}, and we refer to it as
the \defn{skew multivariate hook-length formula
(skew-MHLF)}.  Given a Young diagram $\la$ 
with $n=\ell(\la)$ nonzero rows
and formal
variables $t_1,t_2,\ldots $, we set $x_i\coloneqq
t_{\la_i+n-i+1}$ for $i=1,\ldots,n$, and set the remaining
$t$'s equal to the variables $y$:
$\{y_1,y_2,\ldots\}\coloneqq \{t_1,t_2,\ldots \} \setminus
\{x_1,\ldots,x_n \}$. For example, for $\lambda=(2,1)$, we have $n=2$,
$x_1=t_4$, $x_2=t_2$, and
$y_1=t_1$, $y_2=t_3$, $y_3=t_5$, and so on. 
\begin{theorem}\label{thm:general}
	 Let $\mu \subseteq \la$ be two Young diagrams, and
	 $t_1,t_2,\ldots$ be formal variables. Then
		\begin{align}\label{eq:general2}
			 \sum_{T \in \ssp\SYT(\la/\mu)} 
			 \prod_{k=1}^{|\lambda/\mu|}
			 \frac{1}{t(T^{-1}[< k])} 
			 =
			 \sum_{D \in \mathcal{E}(\la/\mu)}
			 \prod_{(i,j) \in \la \setminus D} 
			 \frac{1}{t_{\la_i+n-i+1} -t_{j+n-\la'_j}}, 
			 \tag{MHLF}
	 \end{align}
	 where for a skew Young diagram $T^{-1}[\geq k] = \la/\nu$
	 occupied by entries $\geq k$ in a \textnormal{SYT} $T$,
	we set $T^{-1}[<k]=\nu$ (by agreement, $T^{-1}[<1]=\mu$), 
	and denote $t(\nu) \coloneqq  \sum_i t_{\la_i+n-i+1}- t_{\nu_i+n-i+1} $. Here~$\la'$ is the transpose of $\la$ .
\end{theorem}

In \Cref{sec:background,sec:general_formalism},
we provide the necessary background and a general
formalism for obtaining sums over skew standard
Young tableaux from Pieri-type rules.

\medskip
Our first proof of \Cref{thm:general} given in 
\Cref{sec:contour_integrals}
evaluates a contour integral of a multivariate rational
function in two different ways. The first evaluation gives a
recursion (Pieri-type formula) which builds up standard
Young tableaux one box at a time, and produces the left-hand side of 
\eqref{eq:general2}. The second evaluation of that
integral gives a determinant of weighted lattice path counts, which via the
Gessel-Viennot formula is equivalent to a weighted enumeration of non-intersecting lattice paths inside $\lambda$, themselves equivalent to the excited
diagrams in the right-hand side of \eqref{eq:general2}. We also derive in~\Cref{prop:oof_multi} an analogous multivariate version of the Okounkov-Olshanski formula studied in~\cite{morales2020okounkov}.

\medskip
The second proof of 
\Cref{thm:general} given 
in \Cref{sec:YBE_proof_of_Naruse}
interprets the identity through integrable vertex models.
More precisely, we interpret the sum over
excited diagrams in the right-hand side of 
\eqref{eq:general2}
as a partition function in the 
six-vertex model at a free fermion point.
The vertex model lives inside the Young diagram $\lambda$,
and the boundary conditions depend on $\mu$.
Using the R-matrix and Yang-Baxter equation,
we show that this partition function
obeys a
recursive formula, building
up the SYTs in the left-hand side of \eqref{eq:general2}.

\medskip

These proofs clear some of the hanging mysteries around the
skew hook-length formula~\eqref{eq:skeq_syt_hook}.  Both
methods allow to generalize this formula to a sum over
semistandard Young tableaux (SSYTs) instead of SYTs, as well
as to other tableaux.  See \Cref{app:SSYT_variant} for one
possible generalization.  Connecting vertex models to
excited diagrams suggests a broad class of boundary
conditions for the six-vertex model. It would be interesting
to explore the corresponding partition functions beyond the
free fermion point.  Note also that both proofs suggest
explicit ways of generalizing formula~\eqref{eq:general2} to
the level of Hall-Littlewood and Macdonald
polynomials. As an illustration, in
\Cref{sub:Macdonald_hook_length} we produce an identity at
the Macdonald level.

\subsection*{Acknowledgments}

We are grateful to 
Darij Grinberg,
Alejandro Morales, Slava Naprienko, Igor Pak,
and Alexander Varchenko for helpful discussions.
A part of this research was performed in Spring 2024 while
the authors were
visiting the
program
``Geometry, Statistical Mechanics, and Integrability''
at the Institute for Pure and Applied Mathematics
(IPAM), which is supported by the NSF grant DMS-1925919. 
GP was partially supported by the NSF grant CCF-2302174. LP was partially supported by the NSF grant DMS-2153869 and the Simons Collaboration Grant for Mathematicians 709055.

\section{Background and definitions}
\label{sec:background}

\ytableausetup{boxsize=1ex}

%

\subsection{Partitions and Young tableaux}
\label{sub:partitions_young}

A \defn{partition} $\la$ of an integer $N$ is a sequence
$\la=(\la_1,\ldots,\la_k)$ of integers $\la_1\geq \la_2
\geq \cdots \geq \la_k \geq 0$, summing up to $N$, i.e.,
$|\lambda|\coloneqq \la_1+\cdots+\la_k=N$. We denote by $\ell(\la) =\max\{i:
\la_i>0\}$ the \defn{length} of the partition, and by $\la'$ its
conjugate transpose, i.e. $\la'_i = \max\{j: \la_j \geq i\}$.  
We represent partitions graphically as \defn{Young diagrams}, with top row having $\la_1$, and so on. For
example $\la =(4,2,1)$ has Young diagram $\ydiagram{4,2,1}$.
We use the same notation $\lambda$ for the partition and its Young diagram
(a set of boxes in $\mathbb{Z}_{\ge1}\times \mathbb{Z}_{\ge1}$),
and it would be clear from the context which one is meant.

A \defn{skew shape} (\defn{diagram}) $\la/\mu$ is the set of boxes in the
Young diagram of $\la$ but not in the diagram of $\mu$ when
both are drawn with top left corner coinciding. 
We view skew shapes $\la/\mu$ as sets of boxes in
$\mathbb{Z}_{\ge1}\times \mathbb{Z}_{\ge1}$.
When using the notation $\lambda/\mu$, 
we always assume that $\mu\subseteq \lambda$,
that is, $\mu_i\le \lambda_i$ for all $i$.
When $\mu=\varnothing$, we have $\la/\mu=\la$.
Denote by $|\lambda/\mu|$
the number of boxes in $\lambda/\mu$ (called \defn{size}).

\ytableausetup{boxsize=1em}

A \defn{standard Young tableaux} (SYT) of shape $\la/\mu$ is a bijection $T:\la/\mu \to \{1,\ldots,|\la/\mu|\}$, such that $T(a,b)\leq T(c,d)$ whenever $a\leq c, b\leq d$. For example, all SYTs of shape $(3,2)/(1)$ are 
$$\ytableaushort{\none 12,34} ,\quad \ytableaushort{\none 13,24}, \quad \ytableaushort{\none 14,23}, \quad \ytableaushort{\none 23,14}, \quad \ytableaushort{\none 24,13}\ .$$
A \defn{semistandard Young tableaux} (SSYT) of shape $\la/\mu$ and type $\alpha$ is a map $T:\la/\mu \to \{1,\ldots,\ell(\alpha)\}$, such that $T(a,b) < T(c,d)$ for $a<c$, $b\leq d$, $T(a,b) \leq T(a,d)$ for $b \leq d$ and $|T^{-1}(i)|=\alpha_i$. The last condition means that the number of boxes filled with $i$ equals $\alpha_i$. Here $\alpha$ is a composition (i.e., a partition without the ordering condition).
A \defn{flagged \textnormal{SSYT}}~\cite{wachs1985flagged} of shape $\la$ and flag $\mathsf{f}=(f_1,\ldots,f_{\ell(\la)})$ is an SSYT $T$, such that in addition, 
$T(i,j) \leq f_i$ for every $i=1,\ldots,\ell(\la)$.
We denote the sets of SYTs and SSYTs of shape $\la/\mu$ by $\SYT(\la/\mu)$ and $\SSYT(\la/\mu)$ respectively, and the set of
flagged SSYTs 
of shape $\mu$ with flag $\mathsf{f}$ by $\SSYT(\mu;\mathsf{f})$.
By convention, all tableaux and skew tableaux are filled with numbers 
starting from $1$.

\ytableausetup{boxsize=1.5ex}

\subsection{Excited diagrams, lozenge tilings and non-intersecting lattice paths}\label{ss:excited}

\defn{Excited diagrams} have appeared many times in the literature, including e.g. \cite{ikeda2009excited, knutson2009grobner, kreiman2005schubert, wachs1985flagged}. One definition uses the following recursive procedure.  Let $D$ be a set of boxes in $\lambda$. An \defn{excited move} on a box in $D$ shifts that box by one along its diagonal as long as none of its immediate neighbors below, to the right or down the diagonal are in $D$:
$$\ydiagram[*(cyan)]{1}*{2,2} \to \ydiagram{2,2}*[*(cyan)]{0,1+1}\ssp.$$
Then the set $\mathcal{E}(\la/\mu)$ is the set of all diagrams in $\la$ which can be obtained from $D=\mu$ after performing a set of the above moves. For example, all 
excited diagrams in $\mathcal{E}((3,3,2)/(2,1))$ are
$$\mathcal{E}((3,3,2)/(2,1)) = \left\{ \ydiagram{3,3,2}*[*(cyan)]{2,1}\ssp, \quad \ydiagram{3,3,2}*[*(cyan)]{1,2+1}*[*(cyan)]{0,1}\ssp, \quad \ydiagram{3,3,2}*[*(cyan)]{2}*[*(cyan)]{0,0,1+1}\ssp, \quad \ydiagram{3,3,2}*[*(cyan)]{1,2+1}*[*(cyan)]{0,0,1+1}\ssp, \quad \ydiagram{3,3,2}*[*(cyan)]{0,1+2,1+1} \right\}.$$
It was observed that excited diagrams
in $\mathcal{E}(\la/\mu)$
bijectively correspond to flagged SSYTs of shape $\mu$ with the flag condition
$f_i = \max\{j\colon  \la_j-j \geq \mu_i-i\}$. 
The correspondence is depicted on the left side of \Cref{fig:ssyt_paths}, and is given as follows.
For $D\in \mathcal{E}(\la/\mu)$, create an SSYT $T$ given by $T(i,j)=r$, 
where $r$ is the row index of the location of the initial box 
$(i,j)$ from $\mu$ in the excited diagram $D$.
This is pictured in the second subfigure of \Cref{fig:excited_diagrams}.
\ytableausetup{boxsize=1em}
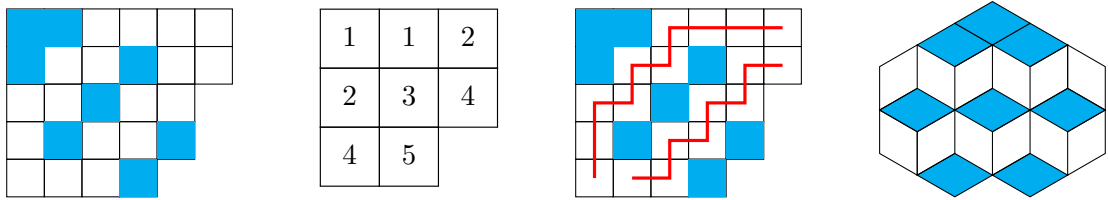
\begin{figure}[htpb]
	\centering
  \raisebox{-35pt}{\begin{tikzpicture}[scale=0.5]
               \foreach \x in {0,1,2,3,4,5} {
\draw (\x,0) rectangle ++(1,1);
\draw (\x,-1) rectangle ++(1,1);
}
\foreach \x in {0,1,2,3,4} {
\draw (\x,-2) rectangle ++(1,1);
\draw (\x,-3) rectangle ++(1,1);
}
\foreach \x in {0,1,2,3} {
\draw (\x,-4) rectangle ++(1,1);
}

\fill[cyan] (0,0) rectangle ++(1,1);
\fill[cyan] (1,0) rectangle ++(1,1);
\fill[cyan] (0,-1) rectangle ++(1,1);

\fill[cyan] (3,-1) rectangle ++(1,1);
\fill[cyan] (2,-2) rectangle ++(1,1);
\fill[cyan] (1,-3) rectangle ++(1,1);
\fill[cyan] (4,-3) rectangle ++(1,1);
\fill[cyan] (3,-4) rectangle ++(1,1);
 \end{tikzpicture}}   
$\qquad$   
\ytableausetup{boxsize=2em}
\ytableaushort{112,234,45}
\ytableausetup{boxsize=1em}
$\qquad$ 
 \raisebox{-35pt}{\begin{tikzpicture}[scale=0.5]
               \foreach \x in {0,1,2,3,4,5} {
\draw (\x,0) rectangle ++(1,1);
\draw (\x,-1) rectangle ++(1,1);
}
\foreach \x in {0,1,2,3,4} {
\draw (\x,-2) rectangle ++(1,1);
\draw (\x,-3) rectangle ++(1,1);
}
\foreach \x in {0,1,2,3} {
\draw (\x,-4) rectangle ++(1,1);
}

\fill[cyan] (0,0) rectangle ++(1,1);
\fill[cyan] (1,0) rectangle ++(1,1);
\fill[cyan] (0,-1) rectangle ++(1,1);

\fill[cyan] (3,-1) rectangle ++(1,1);
\fill[cyan] (2,-2) rectangle ++(1,1);
\fill[cyan] (1,-3) rectangle ++(1,1);
\fill[cyan] (4,-3) rectangle ++(1,1);
\fill[cyan] (3,-4) rectangle ++(1,1);

\draw[ very thick, red] (0.5,-3.5)--(0.5,-1.5)--(1.5,-1.5)--(1.5,-0.5)--(2.5,-0.5)--(2.5,0.5)--(5.5,0.5);

\draw[very thick, red] (1.5,-3.5)--(2.5,-3.5)--(2.5,-2.5)--(3.5,-2.5)--(3.5,-1.5)--(4.5,-1.5)--(4.5,-0.5)--(5.5,-0.5);

		 \end{tikzpicture}} $\qquad$
\raisebox{-35pt}{\begin{tikzpicture}[scale=0.5]
\lozh{2}{0}
\lozh{4}{0}
\lozh{1}{1.5}
\lozh{3}{1.5}
\lozh{5}{1.5}
\lozh{2}{3}
\lozh{4}{3}
\lozh{3}{3.5}
\loza{2}{1}
\loza{4}{1}
\loza{6}{1}
\loza{1}{2.5}
\loza{3}{2.5}
\loza{5}{2.5}
\lozb{2}{1}
\lozb{7}{2.5}
\end{tikzpicture}}
\caption{The many faces of excited diagrams. From left to right: An excited diagram in $\mathcal{E}(\la/\mu)$ for $\la = (6,6,5,5,4)$ and $\mu = (3,3,2)$; the corresponding flagged SSYT (with $f = (3,4,5)$); the nonintersecting lattice paths; and the lozenge tiling.}\label{fig:excited_diagrams}
\end{figure}

In~\cite{MPP2} and separately in~\cite{kreiman2005schubert}, it was observed that excited diagrams are also in bijection with non-intersecting lattice paths within $\la$ which start at the lower border and exit at the right border of $\la$. They are formed exactly by the squares in $\la\setminus D$, as illustrated in the third subfigure of Figure~\ref{fig:excited_diagrams}.

It was then observed in~\cite{MPP3} that excited diagrams are in a bijective correspondence with restricted lozenge tilings of a region with lower boundary given by $\mu$, which can be viewed in 3D as a stack of boxes in the corner of a room with base $\mu$ and height $d$, which depends on $\la$. To see this, let $T$ be the flagged SSYT corresponding to $D$, we then stack $d-T(i,j)+i$ many boxes on the square $(i,j)$ of $\mu$. The partition $\la$ determines how low each column can be, see the last subfigure in Figure~\ref{fig:excited_diagrams}.

\subsection{Symmetric functions} 

While the idea of the present paper is not to rely on any
symmetric functions formalism and identities, many of them
appear in our applications. For the background definitions
we refer to~\cite{Macdonald1995,Stanley1999}. 
The elementary and (complete) homogeneous symmetric polynomials
are
\begin{equation*}
    e_k(x_1,\ldots,x_n) \coloneqq  \sum_{1\leq i_1<\cdots<i_k \leq n} x_{i_1}x_{i_2} \cdots x_{i_k} ,
		\qquad 
		h_k(x_1,\ldots,x_n) \coloneqq  \sum_{1\leq i_1 \leq\cdots \leq i_k \leq n} x_{i_1}x_{i_2} \cdots x_{i_k}.
\end{equation*}
Their generating functions are, respectively, 
\begin{equation*}
	\sum_{r=0}^{n}z^r e_r(x_1,\ldots,x_n ) = \prod_{i=1}^n (1+zx_i), \qquad \sum_{r=0}^{\infty}z^r h_r(x_1,\ldots,x_n ) = \prod_{i=1}^n \frac{1}{1-zx_i}.
\end{equation*}
The \defn{factorial Schur polynomials} are defined as follows:
\begin{equation}\label{eq:weyl}
	s_\mu(x_1,\ldots,x_n\mid a) \coloneqq \frac{1}{\Delta(x)} \det[ (x_i-a_1)\cdots(x_i-a_{\mu_j+n-j})]_{i,j=1}^n,
\end{equation}
where $a_1,a_2,\ldots $ is an arbitrary sequence of shifts, and 
\begin{equation}
	\label{eq:Vandermonde}
	\Delta(x) \coloneqq \prod_{1\le i<j\le n} (x_i-x_j)
\end{equation}
is the Vandermonde determinant.
When all $a_i=0$, we obtain the classical Schur polynomials, $s_\mu(x) = s_\mu(x \mid  \mathbf{0} )$.
Factorial Schur polynomials admit many nice properties common with the Schur polynomials
\cite{biedenharn1989new},
\cite[6th variation]{macdonald1992schur_Theme},
\cite{MolevSagan1999,
molev2009comultiplication}.
In particular, there is the 
following combinatorial formula:
\begin{equation}\label{eq:combin}
	s_\mu(x_1,\ldots,x_n \mid a) = \sum_{T \in \ssp\SSYT(\mu)} \prod_{u \in \mu} (x_{T(u)} - a_{T(u)+c(u)}),
\end{equation}
where for a box $u =(i,j)$ in the Young diagram of $\mu$, the content is 
$c(u)\coloneqq j-i$, and $T(u)$ is the value of $T$ in that box.
The entries in $T$ must be $\leq n$.\footnote{It is possible to insert infinitely many variables into
$s_\mu(\cdots\mid a)$ 
within the formalism of symmetric functions and drop the condition $T(u)\le n$, but we do not need this here.}
We also employ interpolation Macdonald polynomials in \Cref{sub:Macdonald_hook_length}
which admit a combinatorial formula similar to \eqref{eq:combin},
but not a determinantal formula like \eqref{eq:weyl}.

\section{How to get sums over skew standard Young tableaux}
\label{sec:general_formalism}

\subsection{General formalism}
\label{sub:general_formalism_general}

Here we present a general formalism for obtaining summation formulas 
over skew standard Young tableaux (SYTs) 
which come from certain
vanishing properties and Pieri rules.
The main statement of this subsection, \Cref{prop:general_skew_SYT_sum},
appeared in the particular case of the factorial Schur functions 
in
\cite[Proposition 3.2]{MolevSagan1999}, with essentially the same proof.

Assume that 
$\mathsf{Z}_\mu(\lambda)$ is a function of two Young diagrams. 
It may be complex-valued, and can in addition depend on some parameters.
\begin{remark}
	In applications, we obtain $\mathsf{Z}_\mu(\lambda)$ as a specialization of 
	a symmetric polynomial $F_\mu(x_1,\ldots,x_n)$ (like the factorial
	Schur polynomial 
	$s_\mu(x_1,\ldots,x_n\mid a)$ \eqref{eq:weyl}) into the variables $x_i=x_i(\lambda)$,
	$1\le i\le n$, which depend on $\lambda$.
	See \Cref{sub:general_formalism_examples} for a list of examples.
	However, a connection to symmetric polynomials is not
	necessary for the results of the present
	\Cref{sub:general_formalism_general}.
\end{remark}

We assume that the quantities $\mathsf{Z}_\mu(\lambda)$ satisfy
the following conditions:
\begin{enumerate}[$\bullet$]
	\item 
		(\hspace{-1.5pt}\defn{Vanishing property})
		$\mathsf{Z}_\mu(\lambda) = 0$ if $\mu \not\subseteq \lambda$,
		and $\mathsf{Z}_\mu(\mu)\ne 0$ for all $\mu$.
		\item 
			(\hspace{-1pt}\defn{Pieri rule}) There exist quantities 
		$\mathsf{p}_\mu(\lambda)$ and constants $\mathsf{C}_{\nu/\mu}$ such
		that for all $\mu$ and $\lambda$ we have
		\begin{equation}
			\label{eq:pieri_rule_general}
			\mathsf{p}_\mu(\lambda)\ssp \mathsf{Z}_\mu(\lambda) =
			\sum_{\nu=\mu+\square} \mathsf{C}_{\nu/\mu} \ssp\mathsf{Z}_\nu(\lambda).
		\end{equation}
		Here the sum is over all $\nu$ with $|\nu|=|\mu|+1$ which are obtained from 
		$\mu$ by adding a box, and such that $\nu\subseteq \lambda$. 
		We also assume that 
		$\mathsf{p}_\mu(\lambda)\ne 0$ for all $\mu\subseteq \lambda$ with $\mu\ne \lambda$.
		Note that the vanishing property and the Pieri rule imply that 
		$\mathsf{p}_\lambda(\lambda)$ must be zero. 
\end{enumerate}

\begin{proposition}
	\label{prop:general_skew_SYT_sum}
	Under the vanishing property and the Pieri rule
	\eqref{eq:pieri_rule_general},
	we have for any pair of Young diagrams $\mu\subseteq\lambda$:
	\begin{equation}
		\label{eq:general_skew_SYT_sum}
		\mathsf{Z}_\lambda(\lambda)
		\sum_{T\in \ssp\SYT(\lambda/\mu)}
		\prod_{k=1}^{|\lambda/\mu|}
		\frac{\mathsf{C}_{T[=k]}}{\mathsf{p}_{T[< k]}(\lambda)}
		=
		\mathsf{Z}_\mu(\lambda).
	\end{equation}
	Here, for a skew Young diagram $T^{-1}[\geq k] = \la/\nu$
	occupied by entries $\geq k$ in a standard tableau $T\in \ssp\SYT(\lambda/\mu)$,
	we set $T^{-1}[<k]=\nu$ (by agreement, $T^{-1}[<1]=\mu$). 
\end{proposition}
\begin{remark}[More general Pieri rule]
	\label{rem:pieri_rule_more_general}
	The Pieri rule may be extended to add more than one box at a time.
	If $S^+(\mu)$ denotes the set of allowed Young diagrams $\nu\supset \mu$, then 
	\eqref{eq:pieri_rule_general} can be generalized to
	\begin{equation}
	\label{eq:pieri_rule_more_general}
		\mathsf{p}_\mu(\lambda)\ssp
		\mathsf{Z}_\mu(\lambda)
		= \sum_{\nu\in S^+(\mu)} \mathsf{C}_{\nu/\mu}\ssp
		\mathsf{Z}_\nu(\lambda).
	\end{equation}
	If we apply \eqref{eq:pieri_rule_more_general} as in 
	\Cref{prop:general_skew_SYT_sum},
	we obtain a sum over plane partitions $T$ whose equal entries occupy shapes  contained in $S^+(\al)/\al$. 
	This can produce sums over
	semistandard Young tableaux (SSYTs), 
	strict increasing tableaux (SIT) as in \cite{MPP4GrothExcited}, or
	other types of tableaux of skew shape $\lambda/\mu$.
	We present an example in \Cref{app:SSYT_variant}.
\end{remark}
\begin{proof}[Proof of \Cref{prop:general_skew_SYT_sum}]
	We have from \eqref{eq:pieri_rule_general}:
	\begin{equation*}
		\mathsf{Z}_\mu(\lambda)=
		\sum_{\nu=\mu+\square} \frac{\mathsf{C}_{\nu/\mu}}{\mathsf{p}_\mu(\lambda)}
		\ssp
		\mathsf{Z}_\nu(\lambda).
	\end{equation*}
	Continuing this process for each $\mathsf{Z}_\nu(\lambda)$, we add more boxes to the
	Young diagrams 
	until we reach the Young diagram $\lambda$.
	Then we cannot add any more boxes due to the vanishing property of the interpolation symmetric functions.
	As a result, we obtain the desired sum over the skew standard Young tableaux of shape $\lambda/\mu$.
	This completes the proof.
\end{proof}

Under very general assumptions,
\Cref{prop:general_skew_SYT_sum}
represents $\mathsf{Z}_\mu(\lambda)$
as a 
sum over SYTs as 
in the left-hand side of the multivariate hook-formula~\eqref{eq:general2}.
The right-hand side 
$\mathsf{Z}_\mu(\lambda)$
in formulas like~\eqref{eq:general2}
usually has a combinatorial interpretation.
Finding such an interpretation is a problem on its own.

\subsection{Interpolation symmetric polynomials}
\label{sub:general_formalism_examples}

Many examples of families $\{\mathsf{Z}_\mu(\lambda)\}$
satisfying vanishing and Pieri rule
are provided by \defn{interpolation polynomials}
$F_\mu(x_1,\ldots,x_n)$
appearing in the theory of symmetric functions.
Specializing the variables, we obtain
\begin{equation}
	\label{eq:Z_mu_as_specialization}
	\mathsf{Z}_\mu(\lambda)=F_\mu(x_1(\lambda),\ldots,x_n(\lambda)).
\end{equation}
Interpolation properties of $F_\mu(x_1,\ldots,x_n)$
lead to the vanishing, and the Pieri rule
is inherited from symmetric polynomials.
Examples based on interpolation symmetric polynomials
include:
\begin{enumerate}[$\bullet$]
	\item 
		\defn{Factorial Schur polynomials}
		$s_\mu(x_1,\ldots,x_n \mid a)$ 
		\eqref{eq:weyl}
		is the main example we consider in the present paper.
		Note that in both our approaches (via integrals and vertex models),
		we reprove the required 
		properties of factorial Schur polynomials from scratch, without
		using the theory of symmetric functions.
		From this point of view, the essence of
		the skew hook-length
		formula~\eqref{eq:general2}
		is the identification of 
		the specialization
		$\mathsf{Z}_\mu(\lambda)$ 
		\eqref{eq:Z_mu_as_specialization}, where $x_i(\lambda)=a_{\lambda_i+n-i+1}$, $1\le i\le n$
		(the $a$'s are the shifts in the factorial Schur polynomials),
		with a sum over 
		excited diagrams. The two proofs of this identification 
		we present here
		did not explicitly appear in the literature. 

	\item \defn{Interpolation Macdonald polynomials}
		$I_\mu(x_1,\ldots,x_n ; q,t)$
		and the corresponding symmetric functions 
		\cite{knop1997,
		knop1997symmetric,
		sahi1996interpolation,
		okounkov_newton_int,
		okounkov1998shifted},
		see also \cite{olshanski2019interpolation}.
		For interpolation Macdonald polynomials, 
		the quantities 
		\begin{equation}
			\label{eq:macdonald_interpolation_binomial_coefficients}
			\begin{bmatrix}
				\lambda\\ \mu
			\end{bmatrix}_{q,t}
			=
			\frac{
			I_\mu(x^{(q,t)}_1(\lambda),\ldots,x^{(q,t)}_n(\lambda);q,t)}
			{I_\mu(x^{(q,t)}_1(\mu),\ldots,x^{(q,t)}_n(\mu);q,t)}
		\end{equation}
		are multivariate $(q,t)$-analogues of the binomial coefficients
		\cite{okounkov1997binomial}. 
		Note that the normalization 
		in \eqref{eq:macdonald_interpolation_binomial_coefficients}
		differs from the one in our sum over SYTs \eqref{eq:general_skew_SYT_sum}.
		We discuss the Macdonald example in further 
		detail in
		\Cref{sub:Macdonald_hook_length}.
		In particular, the 
		specialization
		which ensures the interpolation 
		is defined 
		in \eqref{eq:a_lambda_Macdonald}.
		
	\item 
		\defn{Balanced elliptic interpolation functions}
		\cite{rains2006bcn}, see also
		\cite{coskun2006well}.

	\item 
		\defn{Factorial Hall-Littlewood polynomials} considered in
		\cite{nakagawa2023universal}.

	\item \defn{Factorial Grothendieck polynomials}
		\cite{mcnamara2006factorial}, see also~\cite{MPP4GrothExcited}.

	\item \defn{Inhomogeneous spin $q$-Whittaker polynomials}
		\cite{korotkikh2024representation}.
\end{enumerate}

\section{Proof by contour integrals}
\label{sec:contour_integrals}

\subsection{A family of integrals}

Let $f_j(u\mid a)$ be a family of polynomials in one variable $u$ depending on
parameters $a=(a_1,a_2,\ldots)$.
Define
the following $n$-fold contour integral indexed by a partition $\mu$ 
with $n\ge \ell(\mu)$:
\begin{equation}\label{eq:integral_def}
	F_\mu(x \mid a)=
	F_\mu(x_1,\ldots,x_n  \mid a)
	\coloneqq  (-1)^{\binom{n}{2}} \frac{1}{(2\pi \sqrt{-1} )^n }
	\oint_\gamma\ldots \oint_{\gamma} 
	\prod_{i=1}^n \frac{ f_{\mu_i+n-i}(u_i\mid a)}{\prod_{j=1}^n (u_i-x_j)} \ssp
	\Delta(u)\ssp du_1\cdots du_n,
\end{equation}
where $\gamma$ 
is a positively oriented contour which
contains all the poles $x_1,\ldots,x_n$,
and $\Delta(u)$ is the Vandermonde determinant \eqref{eq:Vandermonde}.

In this section, we evaluate the integral 
\eqref{eq:integral_def} in two ways, 
via the residues at the $x_j$'s, and 
via the residues at $u_i=\infty$, $1\le i\le n$.
Choosing appropriate polynomials $f_j$ 
produces a proof of the skew hook-length formula~\eqref{eq:general2}
(\Cref{thm:general}). Namely, 
throughout the rest of this section, we set
\begin{equation}\label{eq:f_factorial_schur}
	f_j(u) \coloneqq  f_j(u \mid a)=  \prod_{i=1}^{j}(u - a_i).
\end{equation} 
Taking other polynomials $f_j$ in \eqref{eq:integral_def}
produces generalizations of the skew hook-length formula~\eqref{eq:general2}.
We discuss one such generalization in \Cref{app:SSYT_variant}.

\subsection{Determinant in disguise?}

The integral $F_\mu$ defined via \eqref{eq:integral_def}--\eqref{eq:f_factorial_schur}
can be identified with the factorial Schur polynomial $s_\mu$ \eqref{eq:weyl}.
This fact is not needed for our proof of \eqref{eq:general2},
but we include it for completeness.
\begin{theorem}\label{thm:integral}
	With $f_j(u)$ defined in~\eqref{eq:f_factorial_schur} we have that
	$F_\mu(x_1,\ldots,x_n \mid a) = s_\mu(x_1,\ldots,x_n \mid a)$,
	where $s_\mu$ is the factorial Schur polynomial 
	given by the determinantal formula
	\eqref{eq:weyl}.
\end{theorem}
\begin{proof}
	We evaluate the integral
	\eqref{eq:integral_def}
	by the residue formula at poles $u_i= x_{\sigma(i)}$ for all possible assignments of 
	the poles to the variables, which are encoded by 
	$\sigma \in \left\{ 1,\ldots,n  \right\}^n$. 
	Note that if $\sigma(k)=\sigma(l)$ for some $k\neq l$, then
	$$Res_{u=x_\sigma} = \frac{ f_{\mu_i+n-i}(x_{\sigma(i)})}{\prod_{j\neq \sigma(i)} (x_{\sigma(i)} -x_j)} \prod_{i<j} (x_{\sigma(i)} -x_{\sigma(j)} ) =0, $$
		as the Vandermonde factor vanishes. Thus, nonzero
		residues appear only when $\sigma$ is a permutation.
We then observe that 
\begin{equation*}
\begin{split}
    Res_{u=x_\sigma} &= \prod_i \frac{ f_{\mu_i+n-i}(x_{\sigma(i)})}{\prod_{j\neq \sigma(i)} (x_{\sigma(i)} -x_j)} \; \prod_{i<j} (x_{\sigma(i)} -x_{\sigma(j)} ) \\
    &= \prod_i f_{\mu_i+n-i}(x_{\sigma(i)}) \frac{\sgn(\sigma) \prod_{k<l} (x_k-x_l)}{ \prod_k \prod_{l\neq k} (x_k - x_l) }  \\
		&= \prod_i f_{\mu_i+n-i}(x_{\sigma(i)}) 
		\frac{\sgn(\sigma)}{ \prod_{k>l} (x_k - x_l) },
	\end{split}
\end{equation*}
and summing over all permutations $\sigma$ gives $\frac{1}{\Delta(x)} \det[f_{\mu_i+n-i}(x_j)]_{i,j=1}^n$. Identifying that determinant with~\eqref{eq:weyl} and $s_\mu(x\mid a)$
completes the proof.
\end{proof}

The rest of this section does not rely on 
\Cref{thm:integral}.

\subsection{Pieri rule}
\label{sub:pieri_rule_by_contour}

Here we show that the integrals $F_\mu$
\eqref{eq:integral_def}--\eqref{eq:f_factorial_schur}
satisfy a Pieri rule.
Let $\epsilon_i = (0^{i-1},1,0^{n-i})$ be the $i$-th elementary vector.
\begin{proposition}\label{cor:pieri}
    We have
		\begin{equation}
			\label{eq:pieri_integral_formula_identity}
    \sum_{i=1}^n F_{\mu+\epsilon_i}(x\mid a) = \left( \sum_{i=1}^n x_i - \sum_{i=1}^n a_{\mu_i+n-i+1}\right) F_\mu(x\mid a),
		\end{equation}
		where $F_{\mu+\epsilon_i}(x\mid a) =0$ if $\mu+\epsilon_i$ is not a partition (i.e., does not weakly decrease).
\end{proposition}

The integral $F_\mu$ 
\eqref{eq:integral_def}--\eqref{eq:f_factorial_schur}
is defined for any sequence $\mu$ 
which not necessarily a partition.
Moreover, note that if $\mu_i +1=\mu_{i+1}$ for some $i$, 
then $\mu_i+n-i = \mu_{i+1}+n-(i+1)$, and so
the product of the $f_{\mu_i+n-i}$'s
contains two identical terms. Therefore, the 
integrand becomes antisymmetric in $u_i,u_{i+1}$
thanks to the factor $u_i-u_{i+1}$ coming from the Vandermonde.
Since all integration contours are the same, 
the integral vanishes when $\mu_i+1=\mu_{i+1}$.
Thus, in
\eqref{eq:pieri_integral_formula_identity},
only
the terms for which $\mu+\epsilon_i$ is a partition survive.

\begin{proof}[Proof of \Cref{cor:pieri}]
  Identity~\eqref{eq:pieri_integral_formula_identity} follows from 
	the computation:\footnote{Notation $\mathbf{1}_A$
	means the indicator function of the condition $A$.}
	\begin{equation*}
		\begin{split}
			&\sum_{i=1}^n F_{\mu+\epsilon_i}(x\mid a) 
			=
			\sum_{i=1}^n (-1)^{\binom{n}{2}} \frac{1}{(2\pi
			\sqrt{-1})^n} \oint_\gamma \cdots \oint_\gamma
			\prod_{j=1}^n \frac{f_{\mu_j + \mathbf{1}_{i=j} + n -
			j}(u_j \mid a)}{\prod_{k=1}^n (u_j - x_k)} \Delta(u) \ssp
			du_1 \cdots du_n
			\\ 
			&\hspace{10pt}=
			(-1)^{\binom{n}{2}} \frac{1}{(2\pi
			\sqrt{-1})^n} \oint_\gamma \cdots \oint_\gamma
			\prod_{j=1}^n \frac{f_{\mu_j + n -
			j}(u_j \mid a)}{\prod_{k=1}^n (u_j - x_k)} 
			\Bigl( \sum_{i=1}^{n}u_j-a_{\mu_i+n-i+1} \Bigr)
			\Delta(u) \ssp
			du_1 \cdots du_n
			\\ 
			&\hspace{10pt}=
			\Bigl(
				\sum_{i=1}^{n} x_i-
				\sum_{i=1}^{n} a_{\mu_i+n-i+1}
			\Bigr)F_\mu(x\mid a).
		\end{split}
	\end{equation*}
	In the first line, we used the fact that 
	\begin{equation*}
		 f_{\mu_j + \mathbf{1}_{i=j} + n - j} (u_j \mid a) = f_{\mu_j +
		n - j} (u_j \mid a) \times \left\{ \begin{array}{ll} (u_i -
		a_{\mu_i + n - i + 1}), & \text{if } j = i, \\ 1, & \text{if
		} j \ne i, \end{array} \right.
	\end{equation*}
	which implies that 
	\begin{equation}
		\label{eq:factorial_schur_pieri_proof_factoring}
		\sum_{i=1}^n \prod_{j=1}^n 
		f_{\mu_j + \mathbf{1}_{i=j} + n - j} (u_j \mid a)
		=
		\Bigl( \sum_{i=1}^{n}u_i-a_{\mu_i+n-i+1} \Bigr)
		\prod_{j=1}^n f_{\mu_j + n - j} (u_j \mid a).
	\end{equation}
	In the second line, 
	we used the fact that 
	nonzero residues appear only 
	at permutations:
	$u_i=x_{\sigma(i)}$, $1\le i\le n$.
	The latter implies that 
	$\sum_{i=1}^n u_i = \sum_{i=1}^{n} x_i$.
	This completes the proof.
\end{proof}
The observation \eqref{eq:factorial_schur_pieri_proof_factoring}
is the crux of the proof of \Cref{cor:pieri}.
We use a similar idea to get a generalization
of the Pieri rule (and the skew hook-length formula)
in \Cref{app:SSYT_variant}.

\subsection{Lattice paths and SSYTs}
\label{ss:ssyt}

Let us now evaluate the same integral
$F_\mu(x\mid a)$ \eqref{eq:integral_def}--\eqref{eq:f_factorial_schur}
using the residues at $u_i=\infty$.
We then
interpret the result in terms of weighted
non-intersecting lattice paths. 
We consider a slightly more general setup. 
Let $b=(b_1,b_2,\ldots)$ be a family of parameters,
the polynomials
$f_j(x\mid b)$
be defined by \eqref{eq:f_factorial_schur}, as before, and
$\mathsf{m}=(m_1,m_2,\ldots)$ be a sequence of nonnegative integers. We define
\begin{align}\label{eq:integral_general}
	F_{\mu,\mathsf{m}} (x \mid b)\coloneqq
	(-1)^{\binom{n}{2}} \frac{1}{(2\pi \sqrt{-1} )^n
	}
	\oint_\gamma
	\ldots 
	\oint_\gamma
	\prod_{i=1}^n \frac{ f_{\mu_i+m_i-i}(u_i\mid
	b)}{\prod_{j=1}^{m_i} (u_i-x_j)} \ssp\Delta(u)\ssp 
	du_1\cdots du_n,
\end{align}
where the contour $\gamma$ encompasses all the poles $x_1,\ldots,x_n$.
We recover the original definition 
\eqref{eq:integral_def} by setting $b_i=a_i$ and $m_j=n$ for all $i,j$.

\begin{proposition}\label{prop:general_jacoby_trudy}
    We have 
		\begin{equation}
			\label{eq:general_jacoby_trudy_identity}
			F_{\mu,\mathsf{m}} (x \mid b) 
			= 
			\det[ P_{i,j}^{\mu, \mathsf{m}}(x \mid b) ]_{i,j=1}^n,
		\end{equation}
    where 
		\begin{equation}
			\label{eq:general_jacoby_trudy}
			P_{i,j}^{\mu,\mathsf{m}}(x \mid b)\coloneqq  
			\sum_{r=0}^{\mu_i+j-i} (-1)^r 
			e_r(b_1,\ldots,b_{\mu_i+m_i-i})
			\ssp
			h_{\mu_i+j-i-r}(x_1,\ldots,x_{m_i}).
		\end{equation}
\end{proposition}
\begin{proof}
	The integral 
	\eqref{eq:integral_general}
	becomes, after changing the variables $u_i=1/v_i$:
	\begin{equation*}
		 \frac{1}{(2\pi \sqrt{-1} )^n
		}
		\oint_{\gamma'}
		\ldots 
		\oint_{\gamma'}
		(-1)^{\binom{n}{2}}
		\prod_{i=1}^n \frac{ f_{\mu_i+m_i-i}(1/v_i\mid b)}
		{(-v_i^2)\prod_{j=1}^{m_i} (1/v_i-x_j)} \ssp\Delta(1/v)\ssp 
		dv_1\cdots dv_n,
	\end{equation*}
	where the integration contour $\gamma'$ goes around 
	$0$ in the negative direction, and leaves the 
	poles $x_1^{-1},\ldots,x_n^{-1}$ outside.
	The integrand becomes
	\begin{equation}
		\label{eq:general_jacoby_trudy_proof_integrand}
		\begin{split}
			&
			(-1)^{\binom{n}{2}+n}
			\Delta(1/v)
			\prod_{i=1}^n \frac{(1-v_i b_1)\ldots (1-v_i b_{\mu_i+m_i-i})v_i^{-\mu_i-m_i+i}}
			{v_i^{2-m_i} (1- v_ix_1)\ldots(1-v_ix_{m_i}) }
			\\&\hspace{120pt}=
			(-1)^n
			\Delta(v)
			\prod_{i=1}^n 
			\frac{1}{v_i^{\mu_i+n-i+1}}\cdot
			\frac{(1-v_i b_1)\ldots (1-v_i b_{\mu_i+m_i-i})}
			{(1-v_ix_1)\ldots(1-v_ix_{m_i}) }
			.
		\end{split}
	\end{equation}
	Expanding the Vandermonde determinant as
	$\Delta(v)=\sum_{\sigma\in S_n}\sgn(\sigma) v_1^{n-\sigma_1}\ldots v_n^{n-\sigma_n}$,
	and further using the generating functions for the elementary and complete symmetric functions, we can continue \eqref{eq:general_jacoby_trudy_proof_integrand} as
	\begin{equation*}
		\begin{split}
			=(-1)^n\sum_{\sigma\in S_n}\sgn(\sigma)
			\prod_{i=1}^n\frac{1}{v_i^{\mu_i+\sigma_i-i+1}}
			\sum_{r,k=0}^{\infty} u_i^{r+k}(-1)^r e_r(b_1,\ldots,b_{\mu_i+m_i-i}) h_k(x_1,\ldots,x_{m_i}).
		\end{split}
	\end{equation*}
	Taking the residue at all $v_i=0$ (note that $(-1)^n$ is absorbed
	by changing the orientation of the contours), we arrive at the condition
	$k=\mu_i+\sigma_i-i-r$. Replacing $\sigma_i$ by $j$ leads to a determinant
	of the $P_{i,j}^{\mu,\mathsf{m}}$'s,
	which coincides with the desired expression.
\end{proof}

\begin{remark}
	If $b=(0,0,\ldots)$ and $\mathsf{m}=(n,n,\ldots)$, we have $F_{\mu,\mathsf{m}}(x \mid b) =s_\mu(x_1,\ldots,x_n)$.
	Since in this case $P_{i,j}^{\mu,\mathsf{m}} = h_{\mu_i+j-i}(x_1,\ldots,x_n)$, 
	we recover the classical Jacobi-Trudy identity
	\cite[(II.3.4)]{Macdonald1995}.
\end{remark}

Let us interpret \eqref{eq:general_jacoby_trudy}
as a partition function of weighted lattice paths. 



\begin{lemma}\label{lem:lattice_paths}
	Let $x_1,x_2,\ldots$ be indeterminates and $\ldots,b_{-1},b_0,b_1,\ldots$ be parameters. Consider directed lattice paths $L$ starting at $(-s+1,1)$ and ending at $(k+1,n)$, which make up and right steps 
	(several such paths are in \Cref{fig:ssyt_paths}, right).
	To each horizontal (right) step $(r,t)\to (r+1,t)$ we assign the weight
	$w(r,t): = x_t - b_{r+t}$. Define the weight
	of a path by $w(L)\coloneqq \prod_{u \in L} w(u)$, where 
	the product is over all horizontal steps of $L$.
	Then
	\begin{equation}
		\label{eq:weighted_path_sum_as_P}
		\sum_{L: (-s+1,1) \to (k+1,n)}  w(L) 
		=
		\sum_{r=0}^{k+i} 
		(-1)^r 
		e_r(b_{-i+2},\ldots,b_0,b_1,\ldots,b_{k+n})
		\ssp 
		h_{k+i-r}(x_1,\ldots,x_n).
	\end{equation}
\end{lemma}
\begin{proof}
	The right-hand side of \eqref{eq:weighted_path_sum_as_P}
	can be rewritten as 
	\begin{align*}
			\sum_r (-1)^r \sum_{i_1<\cdots<i_r,j_1\leq j_2 \leq \cdots}b_{i_1}\cdots b_{i_r} x_{j_1} x_{j_2}\cdots
	\end{align*}
	The indices of the $b$'s define $r$ diagonal strips
	$D_\ell=\{ (u,v): i_k-1 \leq u+v \leq i_\ell\}$. 
	Set $D =
	\cup_\ell D_\ell$. We create a path $L$ by greedily
	picking the horizontal steps $j_1,j_2,\ldots$ from the
	vertical line at $(-s+1,1)$ to the right as follows. If
	$(-s+1,j_1)\to (-s+2,j_1) \not \in D$, then we add this
	step to $L$ with weight $x_{j_1}$ and continue. If it
	is in $D$, then we find the largest index $\ell<i$, such
	that $(-s+1,j_1)\to (-\ell+1,j_1) \in D$, but
	$(-\ell+1,j_1)\to(-\ell+2,j_1) \not \in D$, then we add
	$(-s+1,j_1)\to(-\ell+2,j_1)$ to $L$ with weight
	$b_{j_1-i+1}\cdots b_{j_1-\ell+1} x_{j_1}$ and note that
	we must have $i_1=j_1-i+1$. We then continue with $j_2$
	starting from $(-\ell+2,j_1)$ and build up $L$ via
	its horizontal steps. We see that the weight we picked up
	this way is obtained by selecting the corresponding terms
	from each of the brackets $(x_j - b_{i+j})$ along the
	path, picking up a weight $b$ if the horizontal step is in
	$D$ and a weight $x$ otherwise. 

	This is a bijection with the monomials in the left-hand side
	of \eqref{eq:weighted_path_sum_as_P}. Indeed, to see
	the inverse map, taking a path $L$ we select a term from
	each bracket. If we select the term $b_{\ell}$ at the
	$t$'th horizontal step (so $t$th bracket), we set $D_p = \{
	(u,v): \ell-1 \leq u+v \leq \ell\} $, where $p$ is the
	number of $b$ terms selected so far. We also set
	$i_p\coloneqq \ell$. The brackets where $x$'s were
	selected then produce the $j$ indices. 
\end{proof}

Applying \Cref{lem:lattice_paths} with $k=\mu_i-i$, $s=j$ and $n=m_j$, we obtain the following interpretation:

\begin{corollary}\label{cor:P_paths}
	With the notation 
	from \Cref{lem:lattice_paths} and
	\eqref{eq:general_jacoby_trudy},
	setting $b_i=0$ for $i\leq 0$, we have 
    \begin{equation}
			\label{eq:weighted_path}
			P_{i,j}^{\mu,\mathsf{m}}(x \mid b) = 
				\sum_{L\colon (-j+1,1) \to (\mu_i-i+1,m_i)}  w(L). 
		\end{equation}
\end{corollary}

\begin{theorem}\label{thm:F_ssyt}
    We have the following combinatorial formula
		for the functions $F_{\mu,\mathsf{m}}$ \eqref{eq:integral_general}:
		\begin{equation}
			\label{eq:combin_F_ssyt}
			F_{\mu,\mathsf{m}}(x \mid b) 
			=
			\sum_{T \in \ssp\SSYT(\mu;\mathsf{m})} 
			\prod_{v \in \mu} (x_{T(v)} -b_{T(v)+c(v)})
			,
		\end{equation}
		where $\SSYT(\mu;\mathsf{m})$ is the set of all flagged 
		semistandard Young tableaux
		(see
		\Cref{sub:partitions_young})
		of shape~$\mu$ and flag $\mathsf{m}=(m_1,m_2,\ldots)$. 
\end{theorem}
\begin{proof}
	The statement follows by combining \Cref{prop:general_jacoby_trudy}
	and \Cref{cor:P_paths}. Starting from
	\eqref{eq:general_jacoby_trudy_identity},
	let us rewrite this determinantal formula
	as a sum over lattice paths.
	This is possible thanks to the 
	Lindstr\"om-Gessel-Viennot
	lemma 
	\cite{lindstrom1973vector,gessel1985binomial}.
	The lattice paths corresponding to the determinant
	have weights $P_{i,j}^{\mu,\mathsf{m}}(x \mid b)$,
	start at $(-i+1,1)$, and end at $(\mu_j-j+1,m_j)$.
	These lattice paths are in a well-known bijection
	with SSYTs:
	the path starting at $(-i+1,1)$ is the $i$th row
	of the SSYT $T$, and the entries are the levels of the
	horizontal steps. The non-intersecting condition ensures
	that the columns are strictly increasing. 
	If $(i,j)$ is
	a cell in an SSYT $T$, then $T(i,j)$ is equal to the height of
	the path starting at $(-i+1,1)$ at its $j$'th step. Thus, 
	$T(i,j)=t$ if the step is $(-i+j,t)-(i+j+1,t)$, and the
	weight of this step is $w(-i+j,t)=x_t - b_{-i+j+t}$.
	See \Cref{fig:ssyt_paths} for an illustration.	
	Since $c(i,j)=j-i$, this
	weight matches the tableau weight in the right-hand side
	of \eqref{eq:combin_F_ssyt}.
	
	Note that
	nonintersecting conditions force the lattice path
	starting at $(-i+1,1)$ to initially take $i-1$ vertical
	steps and continue through $(-i+1,i)$, which ensures
	that all indices of $b$ appearing in \eqref{eq:combin_F_ssyt} are 
	at least $1$.  
\end{proof}

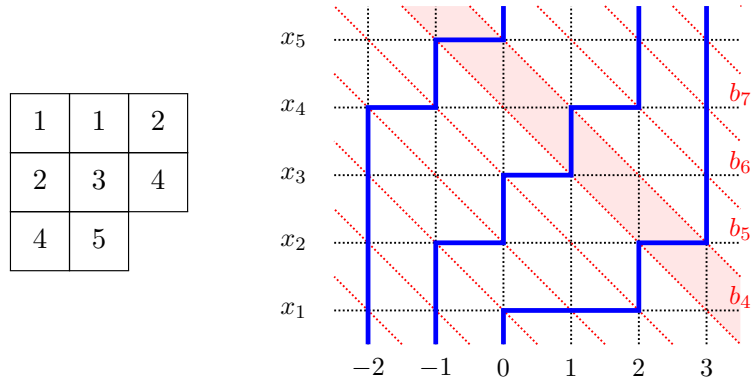
\begin{figure}[htpb]
	\centering
 \begin{subfigure}{0.2\textwidth}
 \ytableausetup{boxsize=2em}
 \raisebox{60pt}{\ytableaushort{112,234,45}}
  

\vspace{0.2in}

 \end{subfigure}
	\begin{subfigure}{0.5\textwidth}
	\scalebox{.9}{
	\begin{tikzpicture}
	[scale=1,thick]
  \fill[red!10,opacity=90] (6.5,0.5)--++(0,1)--++(-4,4)--++(-1,0)--++(5,-5);
	\foreach \jj in {1,2,3,4,5}{
	\draw[densely dotted] (.5-0,\jj)--++(6,0);
	\node at (-.1,\jj) {$x_\jj$};
	}
	\foreach \ii in {-2,...,3}{
	\draw[densely dotted] (\ii+3,.5)--++(0,5);
	\node at (\ii+3,.15) {$\ii$};
	}
    \foreach \aa in {2,...,6}{
    \draw[densely dotted,red] (\aa-.5,0.5)--(1-.5,\aa-.5);
    }
    \foreach \aa in {4,...,9}{
    \draw[densely dotted,red] (\aa+3-5.5,5.5)--(6.5,\aa+3-6.5);
    }
    \foreach \aa in {4,...,7}{ 
    \node[red] at (6.5,\aa-3+.2) {$b_\aa$};
    }

	\draw[line width=2, blue] (3,.5)--++(0,.5)--++(2,0)--++(0,1)--++(1,0)--++(0,3.5);
	\draw[line width=2, blue] (2,.5)--++(0,1.5)--++(1,0)--++(0,1)--++(1,0)--++(0,1)--++(1,0)--++(0,1.5);
	\draw[line width=2, blue] (1,.5)--++(0,3.5)--++(1,0)--++(0,1)--++(1,0)--++(0,0.5);




	\end{tikzpicture}
	}
	\end{subfigure}
	\caption{A semistandard tableaux of total weight $wt(T) = (x_1-b_1)(x_1-b_2)(x_2-b_4)(x_2-b_1)(x_3-b_3)(x_4-b_5)(x_4-b_2)(x_5-b_4)$ and its corresponding non-intersecting lattice path configuration.  
 }
	\label{fig:ssyt_paths}
\end{figure}

As a hint to our final step, and for completeness,
let us obtain as a corollary 
the
combinatorial formula for factorial Schur polynomials~\eqref{eq:combin}:
\begin{corollary}\label{cor:factorial_combin}
	For $\mathsf{m}=(n,n,n,\ldots)$ and $\mu$ a partition,
	formula \eqref{eq:combin_F_ssyt}
	reduces to the one for the factorial Schur polynomials
	\eqref{eq:combin}:
	$$s_\mu(x \mid a) = F_{\mu,\mathsf{m}}(x \mid a) = \sum_{T \in \ssp\SSYT(\mu)} \prod_{v \in \mu} (x_{T(v)} -a_{T(v)+c(v)}).$$
\end{corollary}
\begin{proof}
    \Cref{thm:F_ssyt} applied to parameters $a$ with $m_i=n$ gives the desired right side, as the flag condition becomes trivial and we are summing over all SSYT of shape $\mu$. 
		To see the left side, we invoke \Cref{thm:integral} with $m_i=n$, and observe that integral matches
		\eqref{eq:integral_general}. 
\end{proof}

\subsection{Partition functions of excited diagrams}
\label{ss:excited_formula}

Here we specialize the parameters $a$
in the integral $F_\mu(x\mid a)$
\eqref{eq:integral_def} into a 
sequence containing $x$'s and $y$'s,
where the order of the variables is determined by another Young diagram
$\la$. Namely, let the boundary of $\lambda$ 
be a 
lattice path $L$ from $(0,1)$ to $(\la_1,n)$, encoded as a sequence of U(p) and H(orizontal) steps, so $L_{\la_i +n-i+1}=U$ are the vertical (up) steps for $i=1,\ldots,n$. 
We write a variable $y_j$ for a horizontal step at column $j$ and a variable $x_i$ for the vertical step at height $i$. 
See \Cref{fig:a_lambda} for an illustration.
In detail, 
reading along $L$, we record a sequence of
$x$'s and $y$'s as the entries for $a^\la$:
\begin{equation}
	\label{eq:a_lambda_def}
	a^\la_{\la_i + n-i+1}\coloneqq x_{i},
	\quad 
	\textnormal{and}
	\quad 
	a^\la_r=y_{j_r} \text{ for } \la_{i+1}+n-i+1 \leq r \leq
	\la_i+n-i, \text{ setting }j_r\coloneqq  r-i.
\end{equation}

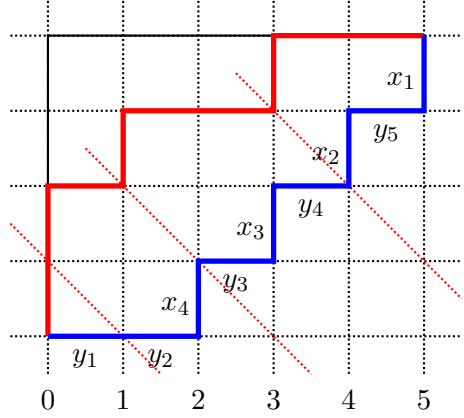
\begin{figure}
    \centering
\begin{tikzpicture}
    [scale=1,thick]
  \draw[thick] (1,1)--(1,5)--(6,5);
  
	\foreach \jj in {1,2,3,4,5}{
	\draw[densely dotted] (.5-0,\jj)--++(6,0);
	}
	\foreach \ii in {0,...,5}{
	\draw[densely dotted] (\ii+1,.5)--++(0,5);
	\node at (\ii+1,.15) {$\ii$};
	}
    \draw[densely dotted,red] (1-.5,2.5)--++(2,-2);
        \draw[densely dotted,red] (2-.5,3.5)--++(3,-3);
        \draw[densely dotted,red] (4-.5,4.5)--++(3,-3);

 \draw[line width=2, red] (1,1)--++(0,2)--++(1,0)--++(0,1)--++(2,0)--++(0,1)--++(2,0);
	\draw[line width=2, blue] (1,1)--++(2,0)--++(0,1)--++(1,0)--++(0,1)--++(1,0)--++(0,1)--++(1,0)--++(0,1);

 \node at (1.5,0.7) {$y_1$};
  \node at (2.5,0.7) {$y_2$};
 \node at (3.5,1.7) {$y_3$};
 \node at (4.5,2.7) {$y_4$};
 \node at (5.5,3.7) {$y_5$};
 \node at (2.7,1.4) {$x_4$};
 \node at (3.7,2.4) {$x_3$};
 \node at (4.7,3.4) {$x_2$};
 \node at (5.7,4.4) {$x_1$};

\end{tikzpicture}    \caption{The parameter sequence $a^{\la}$. Here $\la=(5,4,3,2)$ is the blue path, with labels on the steps giving $a^\la=(y_1,y_2,x_4,y_3,x_3,y_4,x_2,y_5,x_1)$. The shape $\mu=(3,1)$ is drawn in red, and the diagonal lines that start from the end of the rows of $\mu$ are shown meeting $\la$ in rows $m_1=2$, $m_2=3$, $m_3=4$.}
    \label{fig:a_lambda}
\end{figure}



We now consider the combinatorial interpretation of $F_\mu(x \mid a^\la)$ as excited diagrams.
\begin{proposition}\label{prop:excited}
    Let $\mathcal{E}(\la/\mu)$ be the set of excited diagrams of $\mu$ inside $\la$. Then
    $$F_\mu(x \mid a^\la)= \sum_{D \in \mathcal{E}(\la/\mu)} \prod_{(i,j) \in \la\setminus D} (x_i-y_j).$$
\end{proposition}

\begin{proof}
	Substituting $a^\lambda$ into 
	the initial integral formula
	\eqref{eq:integral_def}, we obtain
	many cancellations.
	Denote 
	$m_i \coloneqq \min \{ \ell : \la_\ell < \mu_i+\ell-i\}-1$,
	so $m_i$ is the row index where the diagonal from
	$(i,\mu_i)$ meets the outer boundary
	of the Young diagram $\la$ (see
	\Cref{fig:a_lambda}). 
	We then observe that
	\begin{equation*}
		(a^\la_1,a^\la_2,\ldots, a^\la_{\mu_i+n-i}) 
		=
		(y_1,\ldots,x_n,\ldots,x_{m_i+1},\ldots, y_{\mu_i+m_i-i}),
	\end{equation*}
	that is, the last $x$ variable appearing is $x_{m_i+1}$. 
	We can cancel some of the terms in the integrand of formula \eqref{eq:integral_def} of $F_\mu$ as
 \begin{align*}
     \frac{
			 (u_i-y_1)\cdots(u_i-x_n)\cdots (u_i-x_{m_i+1}) \cdots
			 (u_i -y_{\mu_i+m_i-i})}{ (u_i-x_n)\cdots(u_i-x_1)} 
    = \frac{ (u_i-y_1)\cdots  (u_i
				 -y_{\mu_i+m_i-i})}{ (u_i-x_{m_i})\cdots(u_i-x_1)}.
 \end{align*}
	Then the integral formula becomes
	\begin{equation*}
			F_\mu(x \mid a^\la) =
				 \frac{ (-1)^{\binom{n}{2}} }{(2\pi\sqrt{-1})^n}
				 \oint_\gamma\ldots\oint_\gamma 
				 \prod_i \frac{ (u_i-y_1)\cdots  (u_i
				 -y_{\mu_i+m_i-i})}{ (u_i-x_{m_i})\cdots(u_i-x_1)}\ssp
				 \Delta(u)\ssp du_1 \ldots d u_n 
				 = F_{\mu,\mathsf{m}}(x \mid y),
	\end{equation*}
	where $F_{\mu,\mathsf{m}}(x \mid y)$ is the generalized integral 
	\eqref{eq:integral_general},
	with parameters $b$ replaced by $y$.

	We can now apply \Cref{thm:F_ssyt} to interpret $F_{\mu,\mathsf{m}}$
	as a sum over flagged SSYT. Our final step is to identify
	these flagged SSYT with excited diagrams with the
	corresponding weight. The map from an excited
	diagram $D \in \mathcal{E}(\la/\mu)$ to a flagged SSYT $T$ of
	shape $\mu$ and flag $f_i = \max\{ j: \la_j-j \geq
	\mu_i-i\}$ was discussed in \Cref{ss:excited}. We observe
	that $f_i=m_i$, so we have the same set of SSYTs. Finally,
	if for a box $v=(i,j) \in \mu$ we have $T(i,j)=t$, then
	$T(i,j)+c(i,j)=t+j-i$ is the column index of the
	corresponding excited box and $x_{T(v)}-y_{T(v)+c(v)}=x_t -
	y_{t+j-i}$ and $(t,t+j-i) \in \la \setminus D$ is the
	corresponding box. This completes the proof. 
\end{proof}

The interpolation
property 
of the factorial
Schur polynomials $s_{\mu}(x\mid a)=F_\mu(x\mid a)$
(see \Cref{thm:integral}) can be derived 
directly from  \Cref{prop:excited}.
This property is originally due to
\cite{okounkov_newton_int}
(see also 
\cite{Okounkov1996quantumImm,
OkounkovOlshanski1996ShiftSchur}),
and can be
alternatively shown using the
double alternant formula~\eqref{eq:weyl}.
\begin{corollary}
    	\label{cor:vanishing}
			Let $\la$ be a partition and $a^\la$ be defined in \eqref{eq:a_lambda_def}.
			Then
   $$ F_\mu(x \mid a^\la) =0 \text{ if }\mu \not \subseteq \la, \quad \text{ and } \quad F_\la(x \mid a^\la)=\prod_{(i,j) \in \la}(x_i-y_j).$$
\end{corollary}
\begin{proof}
	For $\mu\not \subseteq \la$, there are 
	no allowable flagged SSYT/excited diagrams, and thus $F_\mu(x \mid a^\la)=0$.
	For $\lambda=\mu$, the only possible flagged tableau
	is the one with $T(i,j)=i$, whose weight for the box $(i,j)$
	is $x_i-y_j$. This completes the proof.
\end{proof}

\subsection{Proof of the generalized hook-length formula and \Cref{thm:general}}
\label{sub:proof_general}

The generalized 
(multivariate) skew
hook-length formulas~\eqref{eq:general},
\eqref{eq:general2} follow from evaluating 
$F_\mu(x \mid a^\la)$ in two different ways. 
One is
recursively by the Pieri rule adding boxes to $\mu$ until it
reaches $\la$, and the other is the combinatorial interpretation
for $s_\mu$ given in \Cref{prop:excited}. 
This approach closely follows the general formalism of
\Cref{sub:general_formalism_general}.
First, let us establish the equivalence of the two formulas:
\begin{proposition}
	\label{prop:equivalence}
	Formulas~\eqref{eq:general} and~\eqref{eq:general2} are equivalent. 
\end{proposition}
\begin{proof}
	Throughout the proof, 
	we use the notation from the Introduction (\Cref{sec:intro}),
	more precisely, the definitions 
	after formula~\eqref{eq:general} and in
	\Cref{thm:general}.

	We start from the right-hand sides.
	Given shapes $\mu \subset \la$, set $n=\ell(\la)$ and 
	$N=|\la/\mu|$.
	Given $z_{-n+1},\ldots,z_{\la_1-1}$ as
	in~\eqref{eq:general}, set $t_0=0$ and
	$t_i=z_{-n}+\cdots+z_{-n+i-1}$ for $i\geq 1$, so that
	$z_{-n+i}=t_{i+1} -t_{i}$. Next, rename
	$t_{\la_{i}+n+1-i}=x_{i}$, and denote the rest of $t_1,t_2\ldots$
	by $y_1,y_2,\ldots$.
	That is, $t=a^\lambda$
	\eqref{eq:a_lambda_def}.
	Observe that with this notation, if
	$u=(i,j)$ is a box in $\la$, then 
	we have telescoping along the hook:
	\begin{equation*}
		\begin{split}
			z(H(u)) &=
			z_{j-\la'_j}+z_{j-\la'_j+1} +\cdots +z_{\la_i-i} \\&=
			t_{\la_i-i+1} -
			t_{\la_i+n-i}+t_{\la_i+n-i}-t_{\la_i+n-i-1}-\cdots -
			t_{j+n-\la'_j} \\&= x_{i}-y_j.
		\end{split}
	\end{equation*}
	In the last equality, we noted that 
	$t_{j+n-\la'_j}=y_j$, as this is the $j$th horizontal step 
	of~$L$, the outer boundary of the Young diagram $\lambda$. 

	For the left-hand sides, pick $T\in \ssp\SYT(\lambda/\mu)$, and 
	let $\nu=T^{-1}[< k]$
	be an intermediate shape occupied
	by the entries $< k$ in $T$. 
	We have
	\begin{equation*}
		\begin{split}
			t(\nu)& = \sum_{i=1}^n ( x_i - t_{\nu_i+n-i+1})
			=
			\sum_{i=1}^n (x_{n+1-i} -t_{\nu_i+n-i+1}) 
			\\&=
			\sum_{i=1}^{n}
			(t_{\la_i +n-i+1} - t_{\nu_i+n-i+1}) 
			=\sum_{i=1}^{n} (z_{\nu_i-i+1}+\cdots+z_{\la_i-i}) 
			= z(\la/\nu),
		\end{split}
	\end{equation*}
	which completes the proof.
\end{proof}

Applying 
\Cref{prop:general_skew_SYT_sum}
with the Pieri rule and the vanishing property
for $F_\mu$ given in 
\Cref{cor:pieri} and \Cref{cor:vanishing}, respectively, we obtain
(with $t=a^\lambda$):
\begin{equation}\label{eq:last_step} 
	\sum_{\mu \subset \mu^1 \subset \mu^2 \cdots \subset \mu^N=\la} \frac{1}{t(\mu) t(\mu^1) \cdots t(\mu^{N-1})} \ssp F_\la(x \mid t)=
	F_\mu(x\mid a^\lambda),
\end{equation}
where $\mu^i/\mu^{i-1}=(1)$ for all $i$, and so each such sequence corresponds to a
standard Young tableau of shape $\lambda/\mu$.
Dividing both sides by 
$F_\lambda(x\mid a^\lambda)$ (given by \Cref{cor:vanishing}),
we identify the 
ratio $F_\mu(x\mid a^\lambda)/F_\lambda(x\mid a^\lambda)$
as a sum over excited diagrams $\mathcal{E}(\la/\mu)$,
thanks to \Cref{prop:excited}.
This completes the proof of \Cref{thm:general}.

\subsection{The generalized Okounkov-Olshanski formula}
In~\cite{morales2020okounkov}, Morales and Zhu obtained a
variant (via reverse excited diagrams in a shifted shape, or certain SSYTs)
of~\eqref{eq:skeq_syt_hook} which they coined as the
Okounkov-Olshanski formula (OOF). The derivations above can
be used to give a multivariate version of formula 
\eqref{eq:last_step},
too.
This derivation was suggested by Alejandro Morales. 

We start with equation~\eqref{eq:last_step} and the same
notation. 
Let us apply the combinatorial formula (\Cref{cor:factorial_combin})
to $F_\mu(x \mid a) = s_\mu(x \mid a) = s_\mu(x_n,\ldots,x_1
\mid a)$, where the variables can be reversed because these
polynomials are symmetric in the $x$'s. Then, let us 
substitute $a=a^\la$. We obtain:
\begin{equation*}
	\sum_{\mu \subset \mu^1 \subset \mu^2 \cdots \subset \mu^N=\la} \frac{1}{t(\mu) t(\mu^1) \cdots t(\mu^{N-1})} \ssp F_\la(x \mid a^\la) = \sum_{T \in \SSYT(\mu)} \prod_{u \in \mu}\bigl(x_{n+1-T(u)} - a^\la_{T(u)+c(u)}\bigr).
\end{equation*}
The terms in the product on the RHS can be written 
in terms of the parameters $t_j$:
\begin{equation*}
	\sum_{\mu \subset \mu^1 \subset \mu^2 \cdots \subset \mu^N=\la} \frac{1}{t(\mu) t(\mu^1) \cdots t(\mu^{N-1})} \ssp F_\la(x \mid a^\la) = \sum_{T \in \SSYT(\mu)} \prod_{u \in \mu}\bigl(t_{\la_{n+1-T(u)}+T(u)} - t_{T(u)+c(u)}\bigr).
\end{equation*}
Substituting the value for $F_\la(x \mid a^\la)$ from \Cref{cor:vanishing} we arrive at the following formula.
\begin{proposition}\label{prop:oof_multi}
   Let $\mu \subset \la$ be two Young diagrams and $t_1,t_2,\ldots$ formal variables. Then
   		\begin{equation}
				\label{eq:general_oof}
				 \tag{MOOF} 
				\begin{split}
				 &
				 \sum_{T \in \ssp\SYT(\la/\mu)} 
				 \prod_{k=1}^{|\lambda/\mu|}
				 \frac{1}{t(T^{-1}[< k])} 
				 =
				 \prod_{(i,j) \in \la} \frac{1}{ t_{\la_i+n-i+1} -t_{j+n-\la'_j} } 
				 \\ 
				 &\hspace{160pt}\times \Bigl(\sum_{T \in \SSYT(\mu)}
				 \prod_{u\in \mu} \bigl(t_{\la_{n+1-T(u)}+T(u)} - t_{T(u)+c(u)}\bigr)\Bigr), 
				\end{split}
		 \end{equation}
  where the notation follows \Cref{thm:general}.
\end{proposition}
The tableaux appearing in the Okounkov-Olshanski formula are not all $\SSYT(\mu)$, as the terms involved vanish for some of them. Morales and Zhu have found several different characterizations of these tableaux and it remains to be understood whether any of these interpretation have nice meanings for the indices $\la_{n+1-T(u)}+T(u)$ and $T(u)+c(u)$. When we substitute $t_i=i$, then $t_{\la_{n+1-T(u)}+T(u)} - t_{T(u)+c(u)}=\la_{n+1-T(u)}-c(u)$ are the arm lengths of certain cells in the reversed shifted excited diagrams of~\cite{morales2020okounkov}. The multivariate formula above reduces to the original Okounkov-Olshanski formula for $f^{\la/\mu}$ in~\cite{okounkov1998shifted}.

\section{Proof via free fermion five-vertex model}
\label{sec:YBE_proof_of_Naruse}

In this section, we present the proof of 
\Cref{thm:general} using the free fermion five-vertex model.
This proof is completely independent
of the one by contour integrals (\Cref{sec:contour_integrals}),
but we comment on the identification 
of certain quantities arising in both approaches.

\subsection{Vertex weights and the Yang--Baxter equation}
\label{sub:vertex_weights_and_YBE}

We begin by recalling the five-vertex weights
which are related to the factorial Schur polynomials
\cite{lascoux20076, mcnamara2009factorial, bump2011factorial}.
See also \cite[Section~4.1]{ABPW2021free} and \cite{Naprienko2023} for generalizations
connecting free fermion six-vertex model
to most known Schur-type functions.

Consider the following vertex weights 
$w_{x}$:
\begin{equation}
	\label{eq:5v_usual_weights}
	\begin{split}	
	w_x
	\bigl( 
		\begin{tikzpicture}[baseline=-3,scale=.7,very thick]
        \draw[fill] (0,0) circle [radius=0.025];
        \draw [dotted] (0.5,0) -- (0.05,0);
        \draw [dotted] (-0.5,0) -- (-0.05,0);
        \draw [dotted] (0,0.05) -- (0, 0.5);
        \draw [dotted] (0,-0.05) -- (0,-0.5);
  \end{tikzpicture}
	\bigr)
	=
	w_x(0,0;0,0)
	=
	x,\quad 
	w_x
	\bigl( 
		\begin{tikzpicture}[baseline=-3,scale=.7,very thick]
			\draw[fill] (0,0) circle [radius=0.025];
        \draw [red] (0.5,0) -- (0.05,0);
        \draw [red] (-0.5,0) -- (-0.05,0);
        \draw [red] (0,0.05) -- (0, 0.5);
        \draw [red] (0,-0.05) -- (0,-0.5);
  \end{tikzpicture}
	\bigr)
	&=
	w_x(1,1;1,1)
	=
	0
	,\quad
	w_x
	\bigl( 
		\begin{tikzpicture}[baseline=-3,scale=.7,very thick]
			\draw[fill] (0,0) circle [radius=0.025];
        \draw [dotted] (0.5,0) -- (0.05,0);
        \draw [dotted] (-0.5,0) -- (-0.05,0);
        \draw [red] (0,-0.05) -- (0,-0.5);
        \draw[red](0,0.05) -- (0,0.5);
  \end{tikzpicture}
\bigr)
=
w_x(1,0;1,0)
=
1,\\ 
	w_x
	\bigl( 
		\begin{tikzpicture}[baseline=-3,scale=.7,very thick]
 \draw[fill] (0,0) circle [radius=0.025];
        \draw [red] (0.5,0) -- (0.05,0);
        \draw [red] (-0.5,0) -- (-0.05,0);
        \draw [dotted] (0,0.05) -- (0, 0.5);
        \draw [dotted] (0,-0.05) -- (0,-0.5);
  \end{tikzpicture}
	\bigr)=
	w_x(0,1;0,1)=
	1
	,\quad
	w_x
	\bigl( 
		\begin{tikzpicture}[baseline=-3,scale=.7,very thick]
			\draw[fill] (0,0) circle [radius=0.025];
				\draw [red] (0,0.5) -- (0,.05);
				\draw [dotted] (0.05,0) -- (0.5,0);
				\draw [dotted] (0,-0.5) -- (0,-0.05);
				\draw [red] (-0.05,0) -- (-0.5,0);
  \end{tikzpicture}
\bigr)&=w_x(0,1;1,0)
	=1
	,\quad 
	w_x
	\bigl( 
		\begin{tikzpicture}[baseline=-3,scale=.7,very thick]
\draw[fill] (0,0) circle [radius=0.025];
				\draw [red] (0,-0.5) -- (0,-.05);
				\draw [dotted] (-0.05,0) -- (-0.5,0);
				\draw [dotted] (0,0.5) -- (0,0.05);
				\draw [red] (0.05,0) -- (0.5,0);
  \end{tikzpicture}
	\bigr)
	=w_x(1,0;0,1)=1.
	\end{split}
\end{equation}
We also need the following dual weights $\check{w}_{y}$:
\begin{equation}
	\label{eq:5v_dual_weights}
	\begin{split}	
	\check{w}_y
	\bigl( 
		\begin{tikzpicture}[baseline=-3,scale=.7,very thick]
        \draw[fill] (0,0) circle [radius=0.025];
        \draw [dotted] (0.5,0) -- (0.05,0);
        \draw [dotted] (-0.5,0) -- (-0.05,0);
        \draw [dotted] (0,0.05) -- (0, 0.5);
        \draw [dotted] (0,-0.05) -- (0,-0.5);
  \end{tikzpicture}
	\bigr)
	&=
	\check{w}_y(0,0;0,0)
	=
	1,\quad 
	\check{w}_y
	\bigl( 
		\begin{tikzpicture}[baseline=-3,scale=.7,very thick]
			\draw[fill] (0,0) circle [radius=0.025];
        \draw [red] (0.5,0) -- (0.05,0);
        \draw [red] (-0.5,0) -- (-0.05,0);
        \draw [red] (0,0.05) -- (0, 0.5);
        \draw [red] (0,-0.05) -- (0,-0.5);
  \end{tikzpicture}
	\bigr)
	=
	\check{w}_y(1,1;1,1)
	=
	0
	,\quad
	\check{w}_y
	\bigl( 
		\begin{tikzpicture}[baseline=-3,scale=.7,very thick]
			\draw[fill] (0,0) circle [radius=0.025];
        \draw [dotted] (0.5,0) -- (0.05,0);
        \draw [dotted] (-0.5,0) -- (-0.05,0);
        \draw [red] (0,-0.05) -- (0,-0.5);
        \draw[red](0,0.05) -- (0,0.5);
  \end{tikzpicture}
\bigr)
=
\check{w}_y(1,0;1,0)
=
\tfrac{1}{y},\\
	\check{w}_y
	\bigl( 
		\begin{tikzpicture}[baseline=-3,scale=.7,very thick]
 \draw[fill] (0,0) circle [radius=0.025];
        \draw [red] (0.5,0) -- (0.05,0);
        \draw [red] (-0.5,0) -- (-0.05,0);
        \draw [dotted] (0,0.05) -- (0, 0.5);
        \draw [dotted] (0,-0.05) -- (0,-0.5);
  \end{tikzpicture}
	\bigr)&=
	\check{w}_y(0,1;0,1)=
	\tfrac{1}{y}
	,\quad
	\check{w}_y
	\bigl( 
		\begin{tikzpicture}[baseline=-3,scale=.7,very thick]
			\draw[fill] (0,0) circle [radius=0.025];
				\draw [red] (0,0.5) -- (0,.05);
				\draw [red] (-0.05,0) -- (-0.5,0);
				\draw [dotted] (0,-0.5) -- (0,-0.05);
				\draw [dotted] (0.05,0) -- (0.5,0);
  \end{tikzpicture}
\bigr)=\check{w}_y(0,1;1,0)
=\tfrac{1}{y}
	,\quad 
	\check{w}_y
	\bigl( 
		\begin{tikzpicture}[baseline=-3,scale=.7,very thick]
\draw[fill] (0,0) circle [radius=0.025];
				\draw [red] (0,-0.5) -- (0,-.05);
				\draw [red] (0.05,0) -- (0.5,0);
				\draw [dotted] (0,0.5) -- (0,0.05);
				\draw [dotted] (-0.05,0) -- (-0.5,0);
  \end{tikzpicture}
	\bigr)
	=\check{w}_y(1,0;0,1)=\tfrac{1}{y}.
	\end{split}
\end{equation}
We set
$w_x(i_1,j_1;i_2,j_2)=\check{w}_y(i_1,j_1;i_2,j_2)=0$
for all choices of $i_1,j_1,i_2,j_2\in\{0,1\}$ not listed in 
\eqref{eq:5v_usual_weights} and \eqref{eq:5v_dual_weights}.
Clearly, $\check{w}_y(i_1,j_1;i_2,j_2)=y^{-1} w_y(i_1,j_1;i_2,j_2)$, 
but these weights play two very different roles, so we will keep this separate notation.

The weights 
\eqref{eq:5v_usual_weights}--\eqref{eq:5v_dual_weights} satisfy the Yang--Baxter equation
with the following weights $r=r_{z}$:
\begin{equation}
	\label{eq:r_weights}
	\begin{split}	
	r_z
	\bigl( 
		\begin{tikzpicture}[baseline=-3,scale=.7,very thick]
        \draw[fill] (0,0) circle [radius=0.025];
        \draw [dotted] (0.5,0) -- (0.05,0);
        \draw [dotted] (-0.5,0) -- (-0.05,0);
        \draw [dotted] (0,0.05) -- (0, 0.5);
        \draw [dotted] (0,-0.05) -- (0,-0.5);
  \end{tikzpicture}
	\bigr)
	=
	r_z(0,0;0,0)
	=
	1,\quad 
	r_z
	\bigl( 
		\begin{tikzpicture}[baseline=-3,scale=.7,very thick]
			\draw[fill] (0,0) circle [radius=0.025];
        \draw [red] (0.5,0) -- (0.05,0);
        \draw [red] (-0.5,0) -- (-0.05,0);
        \draw [red] (0,0.05) -- (0, 0.5);
        \draw [red] (0,-0.05) -- (0,-0.5);
  \end{tikzpicture}
	\bigr)
	&=
	r_z(1,1;1,1)
	=
	1
	,\quad
	r_z
	\bigl( 
		\begin{tikzpicture}[baseline=-3,scale=.7,very thick]
			\draw[fill] (0,0) circle [radius=0.025];
        \draw [dotted] (0.5,0) -- (0.05,0);
        \draw [dotted] (-0.5,0) -- (-0.05,0);
        \draw [red] (0,-0.05) -- (0,-0.5);
        \draw[red](0,0.05) -- (0,0.5);
  \end{tikzpicture}
\bigr)
=
r_z(1,0;1,0)
=
z,\\ 
	r_z
	\bigl( 
		\begin{tikzpicture}[baseline=-3,scale=.7,very thick]
 \draw[fill] (0,0) circle [radius=0.025];
        \draw [red] (0.5,0) -- (0.05,0);
        \draw [red] (-0.5,0) -- (-0.05,0);
        \draw [dotted] (0,0.05) -- (0, 0.5);
        \draw [dotted] (0,-0.05) -- (0,-0.5);
  \end{tikzpicture}
	\bigr)=
	r_z(0,1;0,1)=
	0
	,\quad
	r_z
	\bigl( 
		\begin{tikzpicture}[baseline=-3,scale=.7,very thick]
			\draw[fill] (0,0) circle [radius=0.025];
				\draw [red] (0,0.5) -- (0,.05);
				\draw [dotted] (0.05,0) -- (0.5,0);
				\draw [dotted] (0,-0.5) -- (0,-0.05);
				\draw [red] (-0.05,0) -- (-0.5,0);
  \end{tikzpicture}
\bigr)&=r_z(0,1;1,0)
	=1
	,\quad 
	r_z
	\bigl( 
		\begin{tikzpicture}[baseline=-3,scale=.7,very thick]
\draw[fill] (0,0) circle [radius=0.025];
				\draw [red] (0,-0.5) -- (0,-.05);
				\draw [dotted] (-0.05,0) -- (-0.5,0);
				\draw [dotted] (0,0.5) -- (0,0.05);
				\draw [red] (0.05,0) -- (0.5,0);
  \end{tikzpicture}
	\bigr)
	=r_z(1,0;0,1)=1.
	\end{split}
\end{equation}
Like for \eqref{eq:5v_usual_weights}--\eqref{eq:5v_dual_weights},
the weights $r_z$ \eqref{eq:r_weights} are nonzero on \defn{five} out of six 
configurations which conserve the total number of incoming and outgoing paths at a vertex
(that is, $i_1+j_1=i_2+j_2$).
However, under $r_z$, the paths are allowed to meet at a vertex.

\begin{remark}
	\label{rmk:free_fermion}
	Each of the weights $w_x,\check{w}_y$, and $r_z$ satisfies the free fermion condition
	\begin{equation*}
		w(0,0;0,0)\ssp w(1,1;1,1)+w(1,0;1,0)\ssp w(0,1;0,1)=w(0,1;1,0)\ssp w(1,0;0,1),
	\end{equation*}
	which allows to write many partition functions 
	(i.e., sums of products of vertex weights over all configurations of paths in a region
	with fixed boundary conditions) as determinants.
	See \cite{Naprienko2023} for the most general case of free fermion six-vertex model.
	Note that partition functions for the general six-vertex model
	also take determinantal form for special boundary conditions.
	The most well-known example of this phenomenon 
	is the Izergin--Korepin determinant \cite{korepin1982calculation,Izergin1987}.
\end{remark}

The \defn{spectral parameters} $x,y,z$ in \eqref{eq:5v_usual_weights}--\eqref{eq:r_weights}
may be thought of as generic complex numbers, and the Yang--Baxter equation holds under 
the condition that $z=y-x$:

\begin{proposition}[Yang--Baxter equation]
	\label{prop:YBE}
	For any $i_1,i_2,i_3,j_1,j_2,j_3\in\left\{ 0,1 \right\}$
	and all $x,y,t$ with $x\ne t$,
	we have
	\begin{equation}
		\label{eq:5V_YBE_R}
		\begin{split}
			&\sum\nolimits_{k_1,k_2,k_3}
			w_{x-y}(i_3,k_1;k_3,j_1)
			\ssp
			\check{w}_{x-t}(i_2,i_1;k_2,k_1)
			\ssp
			r_{y-t}(k_3,k_2;j_3,j_2)
			\\&\hspace{40pt}=
			\sum\nolimits_{k_1',k_2',k_3'}
			w_{x-y}(k_3',i_1;j_3,k_1')
			\ssp
			\check{w}_{x-t}(k_2',k_1';j_2,j_1)
			\ssp
			r_{y-t}(i_3,i_2;k_3',k_2').
		\end{split}
	\end{equation}
	where all sums are over $k_1,k_2,k_3\in\left\{ 0,1 \right\}$
	or $k_1',k_2',k_3'\in\left\{ 0,1 \right\}$.
	See \Cref{fig:5V_YBE} for illustration.
\end{proposition}
\begin{proof}
	For each 
	$i_1,i_2,i_3,j_1,j_2,j_3\in\left\{ 0,1 \right\}$,
	equation \eqref{eq:5V_YBE_R} is an identity of rational functions in $x,y,t$
	which is directly checked.
\end{proof}

\begin{figure}[htpb]
	\centering
	\begin{tikzpicture}
		[scale=1,thick]
		
		\draw[line width=3pt, gray!30] (0,0) --++ (3,0);
		\draw[line width=3pt, gray!30] (2,-1) --++ (0,3);
		\draw[line width=3pt, blue!30] (1,-1) --++ (0,1) to[out=90,in=180] (2,1) --++ (1,0);
		\draw[fill] (1,0) circle [radius=1.5pt] node (wc1) {};
				\node(wc1l) at (0,1) {$\scriptsize\check{w}_{x-t}$};
				\draw[->] (wc1l) -- (wc1);
		\draw[fill] (2,1) circle [radius=1.5pt] node (r1) {};
				\node(r1l) at (3,2) {$\scriptsize r_{y-t}$};
				\draw[->] (r1l) -- (r1);
		\draw[fill] (2,0) circle [radius=1.5pt] node (w1) {};
				\node(w1l) at (3,-1) {$\scriptsize w_{x-y}$};
				\draw[->] (w1l) -- (w1);

				\node[left] at (0,0) {\scriptsize$i_1$};
				\node[left] at (1,-1) {\scriptsize$i_2$};
				\node[left] at (2,-1) {\scriptsize$i_3$};
				\node[above] at (3,0) {\scriptsize$j_1$};
				\node[above] at (3,1) {\scriptsize$j_2$};
				\node[left] at (2,2) {\scriptsize$j_3$};
				\node[above, blue] at (1.5,0) {\scriptsize$k_1$};
				\node[above, blue] at (1.25,.75) {\scriptsize$k_2$};
				\node[right, blue] at (2,.5) {\scriptsize$k_3$};

		\node at (4,.5) {\huge$=$};

		\begin{scope}[shift={(3,0)}]
			\draw[line width=3pt, gray!30] (2,1) --++ (3,0);
			\draw[line width=3pt, gray!30] (3,-1) --++ (0,3);
			\draw[line width=3pt, blue!30] (2,0) --++ (1,0) to[out=0,in=-90] (4,1) --++ (0,1);
			\draw[fill] (3,1) circle [radius=1.5pt] node (w2) {};
			\draw[fill] (3,0) circle [radius=1.5pt] node (r2) {};
			\draw[fill] (4,1) circle [radius=1.5pt] node (wc2) {};

			\node(wc2l) at (5,0) {$\scriptsize\check{w}_{x-t}$};
			\draw[->] (wc2l) -- (wc2);
			\node(r2l) at (4,-1) {$\scriptsize r_{y-t}$};
			\draw[->] (r2l) -- (r2);
			\node(w2l) at (2,2) {$\scriptsize w_{x-y}$};
			\draw[->] (w2l) -- (w2);

				\node[left] at (2,1) {\scriptsize$i_1$};
				\node[left] at (2,0) {\scriptsize$i_2$};
				\node[left] at (3,-1) {\scriptsize$i_3$};
				\node[above] at (5,1) {\scriptsize$j_1$};
				\node[right] at (4,2) {\scriptsize$j_2$};
				\node[right] at (3,2) {\scriptsize$j_3$};
				\node[above, blue] at (3.5,1) {\scriptsize$k'_1$};
				\node[left, blue] at (3,.5) {\scriptsize$k'_3$};
				\node[below, blue] at (3.75,.25) {\scriptsize$k'_2$};
		\end{scope}
\end{tikzpicture}
	\caption{Graphical representation of the Yang--Baxter equation \eqref{eq:5V_YBE_R}
	which states that for any fixed boundary conditions
	$i_1,i_2,i_3,j_1,j_2,j_3\in\left\{ 0,1 \right\}$,
	the partition functions on the left and on the right
	are equal. The Yang--Baxter equation is nontrivial only if $i_1+i_2+i_3=j_1+j_2+j_3$.}
	\label{fig:5V_YBE}
\end{figure}
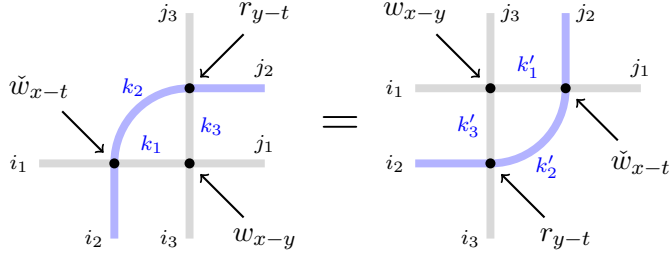

\subsection{Excited diagrams as configurations of the five-vertex model}
\label{sub:5V_excited_diagrams}

Fix two Young diagrams $\mu$ and $\la$ such that $\mu\subset\la$.
Recall the set of excited diagrams $\mathcal{E}(\la/\mu)$
described in \Cref{ss:excited}.
Let $x_1,x_2,\ldots,y_1,y_2,\ldots$ be generic complex numbers
such that $x_i\ne y_j$ for all $i,j$.
Define the following sum over excited diagrams:
\begin{equation}
	\label{eq:def_excited_sum}
	\mathsf{Z}_\mu(\lambda)\coloneqq
	\sum_{D\in \mathcal{E}(\lambda/\mu)}
	\prod_{(i,j)\in D}(x_i-y_j).
\end{equation}
When $\mu\not \subseteq \la$, we set $\mathsf{Z}_\mu(\lambda)=0$.

Clearly, 
\begin{equation}
	\label{eq:F_mu_is_excited_RHS_of_Naruse}
	\mathsf{Z}_\mu(\lambda) = 
	\prod_{(i,j)\in \lambda}(x_i-y_j)
	\sum_{D \in \mathcal{E}(\la/\mu)} \prod_{(i,j) \in \la \setminus D} \frac{1}{x_i-y_j},
\end{equation}
and the sum in the right-hand side is the same as the right-hand
side of the multivariate hook-length formula \eqref{eq:general2}.
Our goal in the present \Cref{sec:YBE_proof_of_Naruse} is to find a representation of 
$\mathsf{Z}_\mu(\lambda)$ as a sum over skew standard Young tableaux.
For this, we will verify the vanishing property and a Pieri rule for 
$\mathsf{Z}_\mu(\lambda)$, following the general strategy
outlined in \Cref{sub:general_formalism_general}. 

\begin{remark}[Connection to factorial Schur polynomials]
	\label{rmk:connection_to_S4}
	The argument in this section 
	is independent from the rest of the paper, and
	does not rely on properties of factorial Schur polynomials.
	More precisely, $\mathsf{Z}_\mu(\lambda)$ is a 
	\defn{specialized} factorial Schur polynomial
	$s_\mu(x \mid a^\lambda)=F_\mu(x \mid a^\lambda)$.
	Here we use only these specialized quantities, and not 
	the general parameters $a$.
\end{remark}

\begin{remark}
	\label{rmk:excited_diagrams_vanishing}
	The vanishing property 
	$\mathsf{Z}_\mu(\lambda)=0$ for $\mu \not \subseteq \la$
	is a part of the definition~\eqref{eq:def_excited_sum}, and 
	we also immediately have
	\begin{equation}
		\label{eq:F_lambda_lambda_product}
		\mathsf{Z}_\lambda(\lambda)=\prod_{(i,j)\in \lambda}(x_i-y_j),
	\end{equation}
	which is nonzero by our assumptions.
\end{remark}

In the present \Cref{sub:5V_excited_diagrams}, we identify the sum 
$\mathsf{Z}_\mu(\lambda)$
in \eqref{eq:def_excited_sum} as a partition function of the 
five-vertex model with the weights \eqref{eq:5v_usual_weights}.
This is done in several steps.

\paragraph{Domain.}
Let $\Omega=\{(i,j)\colon i,j\ge1\}$ be the set of all boxes
in the bottom right quadrant. Here $i$ and $j$ are the
row and column coordinates, with $i$ increasing down, and $j$ increasing to the right.
We will represent boxes by vertices in the five-vertex model, and in this way
the domain $\Omega$ becomes the quadrant $\mathbb{Z}_{\ge1}^{2}$ in the square grid.
Let 
\begin{equation*}
	\Omega_\lambda\coloneqq\{(i,j)\colon 1\le j\le \lambda_i \text{ for all } i\ge1\}
	\subseteq \Omega
\end{equation*}
be the set of all boxes in the Young diagram $\lambda$,
identified with a subset of the square grid.
See \Cref{fig:lambda_5V_conf}, right, for an illustration of $\Omega_\lambda$
for $\lambda=(6,6,5,5,4)$. 

\paragraph{Weights.}
Assign spectral parameters $x_1,x_2,\ldots $ to the rows~$i$,
and spectral parameters $y_1,y_2,\ldots $ to the columns~$j$. Let the
weight at each vertex $(i,j)\in \Omega_\lambda$ be $w_{x_i-y_j}$.

Consider a configuration of paths of the five-vertex model
in $\Omega_\lambda$, such that the paths travel in the
up-right direction and are allowed to enter and exit
$\Omega_{\lambda}$ only through the southeast broken line
border of the Young diagram $\lambda$ (i.e., not through
the west and north straight boundaries of the quadrant
$\Omega$ in which $\lambda$ is placed). Two paths are
nonintersecting; that is, they are
not allowed to pass through the same vertex because
$w_{x_i-y_j}(1,1;1,1)=0$.
Each configuration of paths 
is identified with an excited diagram $D$
whose boxes are precisely the empty vertices $(0,0;0,0)$.
Recall that $w_{x_i-y_j}(0,0;0,0)=x_i-y_j$,
and the 
weights of all other vertices are $1$.
Thus, the weight of a five-vertex path configuration
is equal to $\prod_{(i,j)\in D}(x_i-y_j)$.
See \Cref{fig:lambda_5V_conf} for an illustration.

\ytableausetup{boxsize=3.5ex}
\begin{figure}[htpb]
	\centering
	\begin{tikzpicture}
		\def\offsetx{0.145}
		\def\offsety{-1.1875}
		\def\sqsize{0.595}
		\def\halfsqsize{0.2975}
		\node[anchor=west] at (0,0) {\ydiagram[*(cyan)]{2,1}*[*(cyan)]{1+0,3+1}*[*(cyan)]{1+0,1+0,2+1,1+1,3+1}*[*(cyan)]{1+0,1+0,1+0,4+1}*{6,6,5,5,4}};
		\draw[ultra thick, red,->] (\offsetx+\halfsqsize,\offsety-\sqsize) --++ (0,\halfsqsize) ;
		\draw[ultra thick, red,->] (\offsetx+\halfsqsize+\sqsize,\offsety-\sqsize) --++ (0,\halfsqsize) ;
		\draw[ultra thick, red] 
		(\offsetx+\halfsqsize,\offsety-\sqsize) --++ 
		(0,\sqsize*3)--++(\sqsize,0)--++(0,\sqsize)--++(\sqsize,0)--++(0,\sqsize)--++(4*\sqsize,0);
		\draw[ultra thick, red,->] (\offsetx+6.04*\sqsize,\offsety+4*\sqsize) --++ (\halfsqsize,0) ;
		\draw[ultra thick, red] 
		(\offsetx+\sqsize+\halfsqsize,\offsety-\sqsize) --++ (0,\sqsize)--++(\sqsize,0)
		 --++ (0,\sqsize)--++(\sqsize,0) --++ (0,\sqsize)--++(\sqsize,0) --++ (0,\sqsize)--++(2*\sqsize,0);
		\draw[ultra thick, red,->] (\offsetx+6.04*\sqsize,\offsety+3*\sqsize) --++ (\halfsqsize,0) ;
	\end{tikzpicture}
	\qquad \qquad \qquad 
	\begin{tikzpicture}
		\def\offsetx{0.145}
		\def\offsety{-1.1875}
		\def\sqsize{0.595}
		\def\halfsqsize{0.2975}
		\draw[line width=4pt, gray!30]  (\offsetx+\halfsqsize,\offsety-\sqsize+\halfsqsize)--++(0,5*\sqsize);
		\draw[line width=4pt, gray!30]  (\offsetx+\halfsqsize+\sqsize,\offsety-\sqsize+\halfsqsize)--++(0,5*\sqsize);
		\draw[line width=4pt, gray!30]  (\offsetx+\halfsqsize+2*\sqsize,\offsety-\sqsize+\halfsqsize)--++(0,5*\sqsize);
		\draw[line width=4pt, gray!30]  (\offsetx+\halfsqsize+3*\sqsize,\offsety-\sqsize+\halfsqsize)--++(0,5*\sqsize);
		\draw[line width=4pt, gray!30]  (\offsetx+\halfsqsize+4*\sqsize,\offsety+\halfsqsize)--++(0,4*\sqsize);
		\draw[line width=4pt, gray!30]  (\offsetx+\halfsqsize+5*\sqsize,\offsety+2*\sqsize+\halfsqsize)--++(0,2*\sqsize);
		\draw[line width=4pt, gray!30]  (\offsetx,\offsety)--++(4*\sqsize,0);
		\draw[line width=4pt, gray!30]  (\offsetx,\offsety+\sqsize)--++(5*\sqsize,0);
		\draw[line width=4pt, gray!30]  (\offsetx,\offsety+2*\sqsize)--++(5*\sqsize,0);
		\draw[line width=4pt, gray!30]  (\offsetx,\offsety+3*\sqsize)--++(6*\sqsize,0);
		\draw[line width=4pt, gray!30]  (\offsetx,\offsety+4*\sqsize)--++(6*\sqsize,0);
		\draw[ultra thick, red,->] (\offsetx+\halfsqsize,\offsety-\sqsize) --++ (0,\halfsqsize) ;
		\draw[ultra thick, red,->] (\offsetx+\halfsqsize+\sqsize,\offsety-\sqsize) --++ (0,\halfsqsize) ;
		\draw[ultra thick, red] 
		(\offsetx+\halfsqsize,\offsety-\sqsize) --++ 
		(0,\sqsize*3)--++(\sqsize,0)--++(0,\sqsize)--++(\sqsize,0)--++(0,\sqsize)--++(4*\sqsize,0);
		\draw[ultra thick, red,->] (\offsetx+6.04*\sqsize,\offsety+4*\sqsize) --++ (\halfsqsize,0) ;
		\draw[ultra thick, red] 
		(\offsetx+\sqsize+\halfsqsize,\offsety-\sqsize) --++ (0,\sqsize)--++(\sqsize,0)
		 --++ (0,\sqsize)--++(\sqsize,0) --++ (0,\sqsize)--++(\sqsize,0) --++ (0,\sqsize)--++(2*\sqsize,0);
		\draw[ultra thick, red,->] (\offsetx+6.04*\sqsize,\offsety+3*\sqsize) --++ (\halfsqsize,0) ;
		\node[left] at (\offsetx,\offsety) {$x_5$};
		\node[left] at (\offsetx,\offsety+\sqsize) {$x_4$};
		\node[left] at (\offsetx,\offsety+2*\sqsize) {$x_3$};
		\node[left] at (\offsetx,\offsety+3*\sqsize) {$x_2$};
		\node[left] at (\offsetx,\offsety+4*\sqsize) {$x_1$};
		\node[above] at (\offsetx+\halfsqsize,\offsety+4.5*\sqsize) {$y_1$};
		\node[above] at (\offsetx+\sqsize+\halfsqsize,\offsety+4.5*\sqsize) {$y_2$};
		\node[above] at (\offsetx+2*\sqsize+\halfsqsize,\offsety+4.5*\sqsize) {$y_3$};
		\node[above] at (\offsetx+3*\sqsize+\halfsqsize,\offsety+4.5*\sqsize) {$y_4$};
		\node[above] at (\offsetx+4*\sqsize+\halfsqsize,\offsety+4.5*\sqsize) {$y_5$};
		\node[above] at (\offsetx+5*\sqsize+\halfsqsize,\offsety+4.5*\sqsize) {$y_6$};
	\end{tikzpicture}
	\caption{A configuration of the five-vertex model
		in the domain $\Omega_\lambda$ for $\lambda=(6,6,5,5,4)$
		superimposed on an excited diagram (left), and 
		the same path configuration in the domain 
		$\Omega_\lambda$ (right). In the right picture, we also indicated the spectral parameters $x_i,y_j$
		along the lines. The weight of this configuration is 
		$(x_1-y_1)(x_1-y_2)(x_2-y_1)(x_2-y_4)(x_3-y_3)(x_4-y_2)(x_4-y_5)(x_5-y_4)$.
		The left picture is essentially the same as the third
		picture in \Cref{fig:excited_diagrams}.}
	\label{fig:lambda_5V_conf}
\end{figure}
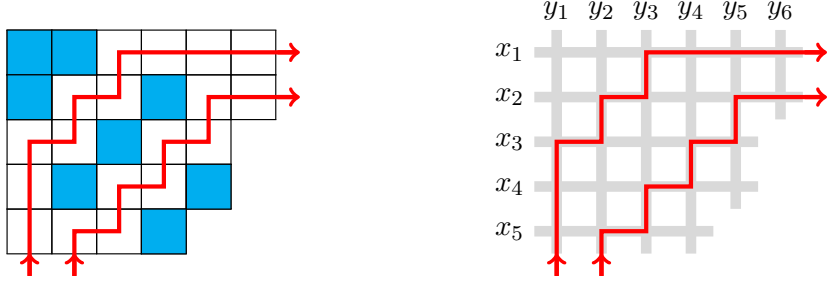

\paragraph{Boundary conditions.}
An elementary diagonal move 
(\Cref{ss:excited})
of a box in an excited diagram is the same as 
the flip $(\mathrm{right}, \mathrm{up})\to (\mathrm{up}, \mathrm{right})$
of a path in the vertex model path configuration. 
Therefore, the set of 
all excited diagrams $D\in \mathcal{E}(\lambda/\mu)$
(for some $\mu\subseteq \lambda$)
is in bijection with the set of all path configurations in $\Omega_\lambda$
with a fixed boundary condition. 
Here by a boundary condition we mean a binary string
along the southeast border of $\lambda$,
where $1$ encodes a entering/exiting path, and $0$ means no path.
For example, the boundary condition for the five-vertex model in
\Cref{fig:lambda_5V_conf} is $11000000011$.

For an arbitrary Young diagram $\mu$, let us define a rim-hook decomposition
of $\Omega\setminus \Omega_\mu$
into parallel translations of 
the first (infinite)
outer rim-hook
\begin{equation*}
	R_\mu^{(1)}\coloneqq
	\bigcup_{i=1}^{\infty}\left\{ (i,j)\colon \mu_i+1\le j\le \mu_{i-1}+1 \right\},
\end{equation*}
where, by agreement, $\mu_0=+\infty$.
Define by $R_\mu^{(k)}$, $k\ge2$, the parallel translation of $R_\mu^{(1)}$
by the vector $(i,j)=(k-1,k-1)$ (that is, by $k-1$ in the southeast direction).
We refer to the $R_\mu^{(k)}$'s as the \defn{$\mu$-rim-hooks}.
We have
\begin{equation*}
	\Omega\setminus \Omega_\mu=\bigcup_{k=1}^{\infty}R_\mu^{(k)}.
\end{equation*}
See \Cref{fig:rim_hook_decomposition} for an illustration.

\begin{figure}[htpb]
	\centering
	\begin{tikzpicture}
		\def\sqsize{0.6}
		\def\halfsqsize{0.3}
		\foreach \ii in {0,...,6}
		{
			\draw[line width=4pt, gray!30]  (0,0 +\ii*\sqsize)--++(9*\sqsize,0);
		}
		\foreach \jj in {0,...,8}
		{
			\draw[line width=4pt, gray!30]  (0+\halfsqsize+\jj*\sqsize,0-\halfsqsize)--++(0,7*\sqsize);
		}
		\node[left] at (0,0+6*\sqsize) {$x_1$};
		\node[left] at (0,0+5*\sqsize) {$x_2$};
		\node[left] at (0,0+4*\sqsize) {$x_3$};
		\node[left] at (0,0+3*\sqsize) {$x_4$};
		\node[left] at (0,0+2*\sqsize) {$x_5$};
		\node[left] at (0,0+1*\sqsize) {$x_6$};
		\node[left] at (0,0) {$\vdots$};
		\node[above] at (0+\halfsqsize,0+6.5*\sqsize) {$y_1$};
		\node[above] at (0+1.5*\sqsize,0+6.5*\sqsize) {$y_2$};
		\node[above] at (0+2.5*\sqsize,0+6.5*\sqsize) {$y_3$};
		\node[above] at (0+3.5*\sqsize,0+6.5*\sqsize) {$y_4$};
		\node[above] at (0+4.5*\sqsize,0+6.5*\sqsize) {$y_5$};
		\node[above] at (0+5.5*\sqsize,0+6.5*\sqsize) {$y_6$};
		\node[above] at (0+6.5*\sqsize,0+6.5*\sqsize) {$y_7$};
		\node[above] at (0+7.5*\sqsize,0+6.5*\sqsize) {$y_8$};
		\node[above] at (0+8.5*\sqsize,0+6.5*\sqsize) {$\ldots$};
		\draw[blue, ultra thick] (0,6.5*\sqsize) --++ (5*\sqsize,0)--++(0,-\sqsize)--++(-\sqsize,0)
		--++(0,-\sqsize)--++(-3*\sqsize,0)--++(0,-\sqsize)--++(-\sqsize,0)--cycle;
		\draw[red!50, line width=6pt] (\halfsqsize,-.5*\sqsize)--++(0,3.5*\sqsize) --++ (1*\sqsize,0)--++(0,1*\sqsize)--++(3*\sqsize,0)--++(0,1*\sqsize)--++(1*\sqsize,0)--++(0,1*\sqsize)--++(3.5*\sqsize,0);
		\draw[green!70!black, line width=6pt] (\halfsqsize+\sqsize,-.5*\sqsize)--++(0,2.5*\sqsize) --++ (1*\sqsize,0)--++(0,1*\sqsize)--++(3*\sqsize,0)--++(0,1*\sqsize)--++(1*\sqsize,0)--++(0,1*\sqsize)--++(2.5*\sqsize,0);
		\draw[yellow!60!black, line width=6pt] (\halfsqsize+2*\sqsize,-.5*\sqsize)--++(0,1.5*\sqsize) --++ (1*\sqsize,0)--++(0,1*\sqsize)--++(3*\sqsize,0)--++(0,1*\sqsize)--++(1*\sqsize,0)--++(0,1*\sqsize)--++(1.5*\sqsize,0);
		\draw[red!70!black, line width=6pt] (\halfsqsize+3*\sqsize,-.5*\sqsize)--++(0,.5*\sqsize) --++ (1*\sqsize,0)--++(0,1*\sqsize)--++(3*\sqsize,0)--++(0,1*\sqsize)--++(1*\sqsize,0)--++(0,1*\sqsize)--++(.5*\sqsize,0);
		\draw[blue!60, line width=6pt] (\halfsqsize+5*\sqsize,-.5*\sqsize)--++(0,.5*\sqsize)--++(3*\sqsize,0)--++(0,1*\sqsize)--++(.5*\sqsize,0);
		\node[right] at (9*\sqsize,6*\sqsize) {\scriptsize$R_\mu^{(1)}$};
		\node[right] at (9*\sqsize,5*\sqsize) {\scriptsize$R_\mu^{(2)}$};
		\node[right] at (9*\sqsize,4*\sqsize) {\scriptsize$R_\mu^{(3)}$};
		\node[right] at (9*\sqsize,3*\sqsize) {\scriptsize$R_\mu^{(4)}$};
		\node[right] at (9*\sqsize,1*\sqsize) {\scriptsize$R_\mu^{(5)}$};

		\draw[densely dotted, line width=2.1pt] (0,1.5*\sqsize) --++ (4*\sqsize,0)
		--++(0,\sqsize)--++(\sqsize,0)--++(0,2*\sqsize)--++(\sqsize,0)--++(0,2*\sqsize);
	\end{tikzpicture}
	\caption{Rim-hook decomposition of the part of the square lattice $\Omega\setminus \Omega_\mu$
		into the union of $R_{\mu}^{(k)}$, $k\ge1$. Here 
		Here $\mu=(5,4,1)$. The dotted line indicates the southeast border of another 
		Young diagram, $\lambda=(6,6,5,5,4)$.}
	\label{fig:rim_hook_decomposition}
\end{figure}
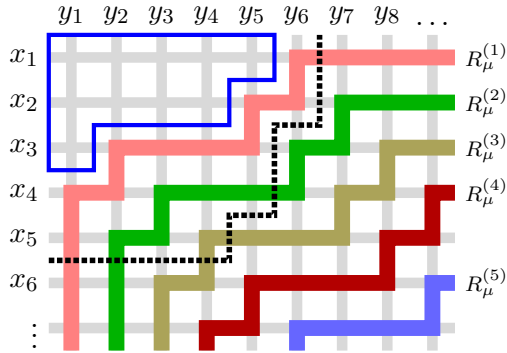

\begin{definition}
	\label{def:mu_boundary_condition}
	For two Young diagrams $\mu\subseteq\lambda$, 
	the \defn{$\mu$-boundary condition} on $\Omega_\lambda$ is a binary string
	$B(\lambda/\mu)$
	of length
	$\lambda_1+\ell(\lambda)$
	which records intersections of of the southeast border 
	of $\lambda$
	with the $\mu$-rim-hooks $R_\mu^{(k)}$, $k\ge1$.
	Namely, when a length $1$ segment of the boundary of $\lambda$
	intersects any $\mu$-rim-hook $R_\mu^{(k)}$,
	we put $1$ in the position of the string $B(\lambda/\mu)$ corresponding to this boundary segment.
	When a $\lambda$-boundary segment does not intersect a $\mu$-rim-hook, we put $0$ in $B(\lambda/\mu)$.
\end{definition}

For example, 
in \Cref{fig:rim_hook_decomposition} the $\mu$-boundary condition is $B(66554/541)=11011010001$.
Note that the diagram $\mu=(5,4,1)$ in \Cref{fig:rim_hook_decomposition} differs
from the inner diagram in \Cref{fig:lambda_5V_conf},
which results in a different binary string $B(66554/332)=11000000011$.

Clearly, the boundary of $\lambda$ intersects each 
rim-hook $R_\mu^{(k)}$ an even number of times.
However, not every binary string of length $\lambda_1+\ell(\lambda)$ 
with an even number of $1$'s
is a valid $\mu$-boundary condition
(see \Cref{prop:delta_symmetric_difference} below for a precise description).

\begin{proposition}
	\label{prop:excited_as_5V}
	For any $\mu\subseteq\lambda$,
	the sum over excited diagrams $\mathsf{Z}_\mu(\lambda)$ 
	\eqref{eq:def_excited_sum} is equal to the partition function of the five-vertex model
	in $\Omega_\lambda$ with the weight $w_{x_i-y_j}$ \eqref{eq:5v_usual_weights}
	at each vertex $(i,j)\in\Omega_\lambda$, boundary conditions $B(\lambda/\mu)$ along the
	southeast border of $\Omega_\lambda$, and 
	empty boundary conditions along its west and north boundaries.
\end{proposition}
\begin{proof}
	This statement follows from the discussion above in the
	present \Cref{sub:5V_excited_diagrams}. Indeed,
	observe that the configuration of rim-hooks $R_\mu^{(k)}$
	inside $\Omega_\lambda$ is the same as a distinguished 
	five-vertex model paths configuration. This distinguished
	configurations is minimal in the sense that all empty
	vertices are pushed in the northwest direction. 
	The
	minimal configuration is identified with an initial
	excited diagram $D=\mu\in \mathcal{E}(\lambda/\mu)$.
	A move of a box in an
	excited diagram (\Cref{ss:excited}) corresponds to a
	flip of a path in the five-vertex model. All five-vertex
	model path configurations are obtained from the minimal
	one by a sequence of flips. Thus, the five-vertex model
	partition function is equal to the sum over all $D\in
	\mathcal{E}(\lambda/\mu)$. This completes the proof.
\end{proof}

\begin{remark}
	\label{rmk:connection_to_S4_vertex_models}
	The five-vertex model configurations 
	in 
	\Cref{prop:excited_as_5V} are the same as 
	the nonintersecting lattice 
	path configurations discussed in \Cref{ss:ssyt}
	(and which we enumerated by 
	a determinantal
	formula).
	Note that in the present \Cref{sec:YBE_proof_of_Naruse}, 
	a key role in the analysis of the five-vertex model
	is played by the boundary
	conditions along the southeast border of $\Omega_\lambda$.
	In particular, the dependence of these boundary
	conditions on $\mu$ is crucial for the Pieri rule.
\end{remark}

\subsection{Yang--Baxter moves sweeping a Young diagram}
\label{sub:YBE_on_Young_diagram}

In this subsection, we apply the Yang--Baxter equation
(\Cref{prop:YBE}) to express $\mathsf{Z}_\mu(\lambda)$ as a partition
function in a larger domain with an additional strand of
vertices along its southeast border. 
Throughout this subsection, $t$ is an
auxiliary spectral parameter assumed to be a generic complex number.
First, let us add a
strand to the northwest boundary of $\Omega_\lambda$.

\begin{definition}[Domain $\Omega^{\nwsymbol}_\lambda$]
	\label{def:nw_domain}
	Fix $\mu\subseteq\lambda$, and consider a 
	larger domain 
	$\Omega^{\nwsymbol}_\lambda$ obtained
	by adding a single \defn{new strand} of $\lambda_1+\ell(\lambda)$ vertices
	along the northwest boundary of $\Omega_\lambda$.
	Let the additional vertices $(i,0)$, $1\le i\le \ell(\lambda)$,
	have weights $\check{w}_{x_i-t}$, and $(0,j)$, $1\le j\le \lambda_1$,
	have weights $r_{y_j-t}$. The northwest boundary
	of $\Omega^{\nwsymbol}_\lambda$ 
	and the boundary conditions on the new strand are empty, while the
	southeast border carries the binary string $B(\lambda/\mu)$.
	Inside $\Omega_\lambda$, the weights are $w_{x_i-y_j}$, as before.
	See \Cref{fig:5V_Pieri_additional_strand}, left, for an illustration.
\end{definition}

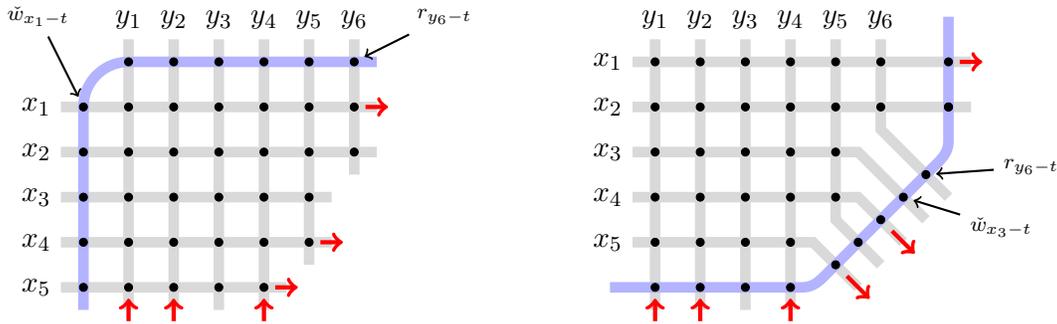
\begin{figure}[htpb]
	\centering
	\begin{tikzpicture}
		\def\sqsize{0.6}
		\def\halfsqsize{0.3}
		\draw[line width=4pt, gray!30]  (0+\halfsqsize,0-\sqsize+\halfsqsize)--++(0,6*\sqsize);
		\draw[line width=4pt, gray!30]  (0+\halfsqsize+\sqsize,0-\sqsize+\halfsqsize)--++(0,6*\sqsize);
		\draw[line width=4pt, gray!30]  (0+\halfsqsize+2*\sqsize,0-\sqsize+\halfsqsize)--++(0,6*\sqsize);
		\draw[line width=4pt, gray!30]  (0+\halfsqsize+3*\sqsize,0-\sqsize+\halfsqsize)--++(0,6*\sqsize);
		\draw[line width=4pt, gray!30]  (0+\halfsqsize+4*\sqsize,0+\halfsqsize)--++(0,5*\sqsize);
		\draw[line width=4pt, gray!30]  (0+\halfsqsize+5*\sqsize,0+2*\sqsize+\halfsqsize)--++(0,3*\sqsize);
		\draw[line width=4pt, gray!30]  (-\sqsize,0)--++(5*\sqsize,0);
		\draw[line width=4pt, gray!30]  (0-\sqsize,0+\sqsize)--++(6*\sqsize,0);
		\draw[line width=4pt, gray!30]  (0-\sqsize,0+2*\sqsize)--++(6*\sqsize,0);
		\draw[line width=4pt, gray!30]  (0-\sqsize,0+3*\sqsize)--++(7*\sqsize,0);
		\draw[line width=4pt, gray!30]  (0-\sqsize,0+4*\sqsize)--++(7*\sqsize,0);
		\draw[ultra thick, red,->] (0+\halfsqsize,0-\sqsize+\halfsqsize/2) --++ (0,\halfsqsize) ;
		\draw[ultra thick, red,->] (0+\halfsqsize+\sqsize,0-\sqsize+\halfsqsize/2) --++ (0,\halfsqsize) ;
		\draw[ultra thick, red,->] (0+\halfsqsize+3*\sqsize,0-\sqsize+\halfsqsize/2) --++ (0,\halfsqsize) ;
		\draw[ultra thick, red,->] (0+4*\sqsize-\halfsqsize/2 ,0+0*\sqsize) --++ (\halfsqsize,0) ;
		\draw[ultra thick, red,->] (0+6*\sqsize-\halfsqsize/2 ,0+4*\sqsize) --++ (\halfsqsize,0) ;
		\draw[ultra thick, red,->] (0+5*\sqsize - \halfsqsize/2,0+1*\sqsize) --++ (\halfsqsize,0) ;
		\node[left] at (-\sqsize,0) {$x_5$};
		\node[left] at (-\sqsize,0+\sqsize) {$x_4$};
		\node[left] at (-\sqsize,0+2*\sqsize) {$x_3$};
		\node[left] at (-\sqsize,0+3*\sqsize) {$x_2$};
		\node[left] at (-\sqsize,0+4*\sqsize) {$x_1$};
		\node[above] at (0+\halfsqsize,0+5.5*\sqsize) {$y_1$};
		\node[above] at (0+\sqsize+\halfsqsize,0+5.5*\sqsize) {$y_2$};
		\node[above] at (0+2*\sqsize+\halfsqsize,0+5.5*\sqsize) {$y_3$};
		\node[above] at (0+3*\sqsize+\halfsqsize,0+5.5*\sqsize) {$y_4$};
		\node[above] at (0+4*\sqsize+\halfsqsize,0+5.5*\sqsize) {$y_5$};
		\node[above] at (0+5*\sqsize+\halfsqsize,0+5.5*\sqsize) {$y_6$};
		\draw[line width=4pt, blue!30] (- \halfsqsize,- \halfsqsize) --++ (0,4.5*\sqsize) 
		to[out=90,in=180] ++(\sqsize,\sqsize) --++ (5.5*\sqsize,0);
		\draw[fill] (0*\sqsize-\halfsqsize, 0*\sqsize) circle [radius=1.5pt] {};
		\draw[fill] (1*\sqsize-\halfsqsize, 0*\sqsize) circle [radius=1.5pt] {};
		\draw[fill] (2*\sqsize-\halfsqsize, 0*\sqsize) circle [radius=1.5pt] {};
		\draw[fill] (3*\sqsize-\halfsqsize, 0*\sqsize) circle [radius=1.5pt] {};
		\draw[fill] (4*\sqsize-\halfsqsize, 0*\sqsize) circle [radius=1.5pt] {};
		\draw[fill] (0*\sqsize-\halfsqsize, 1*\sqsize) circle [radius=1.5pt] {};
		\draw[fill] (1*\sqsize-\halfsqsize, 1*\sqsize) circle [radius=1.5pt] {};
		\draw[fill] (2*\sqsize-\halfsqsize, 1*\sqsize) circle [radius=1.5pt] {};
		\draw[fill] (3*\sqsize-\halfsqsize, 1*\sqsize) circle [radius=1.5pt] {};
		\draw[fill] (4*\sqsize-\halfsqsize, 1*\sqsize) circle [radius=1.5pt] {};
		\draw[fill] (5*\sqsize-\halfsqsize, 1*\sqsize) circle [radius=1.5pt] {};
		\draw[fill] (0*\sqsize-\halfsqsize, 2*\sqsize) circle [radius=1.5pt] {};
		\draw[fill] (1*\sqsize-\halfsqsize, 2*\sqsize) circle [radius=1.5pt] {};
		\draw[fill] (2*\sqsize-\halfsqsize, 2*\sqsize) circle [radius=1.5pt] {};
		\draw[fill] (3*\sqsize-\halfsqsize, 2*\sqsize) circle [radius=1.5pt] {};
		\draw[fill] (4*\sqsize-\halfsqsize, 2*\sqsize) circle [radius=1.5pt] {};
		\draw[fill] (5*\sqsize-\halfsqsize, 2*\sqsize) circle [radius=1.5pt] {};
		\draw[fill] (0*\sqsize-\halfsqsize, 3*\sqsize) circle [radius=1.5pt] {};
		\draw[fill] (1*\sqsize-\halfsqsize, 3*\sqsize) circle [radius=1.5pt] {};
		\draw[fill] (2*\sqsize-\halfsqsize, 3*\sqsize) circle [radius=1.5pt] {};
		\draw[fill] (3*\sqsize-\halfsqsize, 3*\sqsize) circle [radius=1.5pt] {};
		\draw[fill] (4*\sqsize-\halfsqsize, 3*\sqsize) circle [radius=1.5pt] {};
		\draw[fill] (5*\sqsize-\halfsqsize, 3*\sqsize) circle [radius=1.5pt] {};
		\draw[fill] (6*\sqsize-\halfsqsize, 3*\sqsize) circle [radius=1.5pt] {};
		\draw[fill] (0*\sqsize-\halfsqsize, 4*\sqsize) circle [radius=1.5pt] node (whceck) {};
		\node (whceckl) at (-1*\sqsize-\halfsqsize, 6*\sqsize) {\scriptsize$\check{w}_{x_1-t}$};
		\draw[->,thick] (whceckl) -- (whceck);
		\draw[fill] (1*\sqsize-\halfsqsize, 4*\sqsize) circle [radius=1.5pt] {};
		\draw[fill] (2*\sqsize-\halfsqsize, 4*\sqsize) circle [radius=1.5pt] {};
		\draw[fill] (3*\sqsize-\halfsqsize, 4*\sqsize) circle [radius=1.5pt] {};
		\draw[fill] (4*\sqsize-\halfsqsize, 4*\sqsize) circle [radius=1.5pt] {};
		\draw[fill] (5*\sqsize-\halfsqsize, 4*\sqsize) circle [radius=1.5pt] {};
		\draw[fill] (6*\sqsize-\halfsqsize, 4*\sqsize) circle [radius=1.5pt] {};
		\draw[fill] (1*\sqsize-\halfsqsize, 5*\sqsize) circle [radius=1.5pt] {};
		\draw[fill] (2*\sqsize-\halfsqsize, 5*\sqsize) circle [radius=1.5pt] {};
		\draw[fill] (3*\sqsize-\halfsqsize, 5*\sqsize) circle [radius=1.5pt] {};
		\draw[fill] (4*\sqsize-\halfsqsize, 5*\sqsize) circle [radius=1.5pt] {};
		\draw[fill] (5*\sqsize-\halfsqsize, 5*\sqsize) circle [radius=1.5pt] {};
		\draw[fill] (6*\sqsize-\halfsqsize, 5*\sqsize) circle [radius=1.5pt] {} node (r) {};
		\node (rl) at (8*\sqsize-\halfsqsize, 6*\sqsize) {\scriptsize$r_{y_6-t}$};
		\draw[->,thick] (rl) -- (r);

		\begin{scope}[shift={(7,\sqsize)}]
			\draw[line width=4pt, gray!30]  (0+\halfsqsize,0-2*\sqsize+\halfsqsize)--++(0,6*\sqsize);
			\draw[line width=4pt, gray!30]  (0+\halfsqsize+\sqsize,0-2*\sqsize+\halfsqsize)--++(0,6*\sqsize);
			\draw[line width=4pt, gray!30]  (0+\halfsqsize+2*\sqsize,0-2*\sqsize+\halfsqsize)--++(0,6*\sqsize);
			\draw[line width=4pt, gray!30]  (0+\halfsqsize+3*\sqsize,0-2*\sqsize+\halfsqsize)--++(0,6*\sqsize);
			\draw[line width=4pt, gray!30]  (0+\halfsqsize+5*\sqsize,-\sqsize+\halfsqsize)--++(-\sqsize,\sqsize)--++(0,4*\sqsize);
			\draw[line width=4pt, gray!30]  (0+\halfsqsize+6.5*\sqsize,0+.5*\sqsize+\halfsqsize)--++(-\sqsize*1.5,\sqsize*1.5)--++(0,2*\sqsize);
			\draw[line width=4pt, gray!30]  (0,0)--++(4*\sqsize,0)--++(\sqsize,-\sqsize);
			\draw[line width=4pt, gray!30]  (0,0+\sqsize)--++(5*\sqsize,0)--++(\sqsize,-\sqsize);
			\draw[line width=4pt, gray!30]  (0,0+2*\sqsize)--++(5*\sqsize,0)--++(\sqsize*1.5,-\sqsize*1.5);
			\draw[line width=4pt, gray!30]  (0,0+3*\sqsize)--++(7.5*\sqsize,0);
			\draw[line width=4pt, gray!30]  (0,0+4*\sqsize)--++(7.5*\sqsize,0);
			\draw[line width=4pt, blue!30]
			(-\halfsqsize, -\halfsqsize-\sqsize/2)
			-- ++(\sqsize/2, 0)
			-- ++(3.75*\sqsize, 0)
			to[out=0, in=-135] ++ (\sqsize/2,\sqsize/4) --++ (2.5*\sqsize, 2.5*\sqsize)
			to[out=45, in=-90] ++ (\sqsize/4,\sqsize/2) -- (7*\sqsize,4.5*\sqsize)
			--++ (0,\sqsize/2);
			\node[left] at (-0,0) {$x_5$};
			\node[left] at (-0,0+\sqsize) {$x_4$};
			\node[left] at (-0,0+2*\sqsize) {$x_3$};
			\node[left] at (-0,0+3*\sqsize) {$x_2$};
			\node[left] at (-0,0+4*\sqsize) {$x_1$};
			\node[above] at (0+\halfsqsize,0+4.5*\sqsize) {$y_1$};
			\node[above] at (0+\sqsize+\halfsqsize,0+4.5*\sqsize) {$y_2$};
			\node[above] at (0+2*\sqsize+\halfsqsize,0+4.5*\sqsize) {$y_3$};
			\node[above] at (0+3*\sqsize+\halfsqsize,0+4.5*\sqsize) {$y_4$};
			\node[above] at (0+4*\sqsize+\halfsqsize,0+4.5*\sqsize) {$y_5$};
			\node[above] at (0+5*\sqsize+\halfsqsize,0+4.5*\sqsize) {$y_6$};
			\draw[ultra thick, red,->] (0+\halfsqsize,0-2*\sqsize+\halfsqsize/2) --++ (0,\halfsqsize);
			\draw[ultra thick, red,->] (0+\halfsqsize+\sqsize,0-2*\sqsize+\halfsqsize/2) --++ (0,\halfsqsize);
			\draw[ultra thick, red,->] (0+\halfsqsize+3*\sqsize,0-2*\sqsize+\halfsqsize/2) --++ (0,\halfsqsize);
			\draw[ultra thick, red,->] (0+7.5*\sqsize - \halfsqsize/2,0+4*\sqsize) --++ (\halfsqsize,0) ;
			\draw[ultra thick, red,->] (0+5*\sqsize - \halfsqsize/2,0+-.75*\sqsize) --++ (\halfsqsize,-\halfsqsize);
			\draw[ultra thick, red,->] (0+6*\sqsize - \halfsqsize/2,0+.25*\sqsize) --++ (\halfsqsize,-\halfsqsize);
			\draw[fill] (1*\sqsize-\halfsqsize, 0*\sqsize) circle [radius=1.5pt] {};
			\draw[fill] (2*\sqsize-\halfsqsize, 0*\sqsize) circle [radius=1.5pt] {};
			\draw[fill] (3*\sqsize-\halfsqsize, 0*\sqsize) circle [radius=1.5pt] {};
			\draw[fill] (4*\sqsize-\halfsqsize, 0*\sqsize) circle [radius=1.5pt] {};
			\draw[fill] (1*\sqsize-\halfsqsize, 1*\sqsize) circle [radius=1.5pt] {};
			\draw[fill] (2*\sqsize-\halfsqsize, 1*\sqsize) circle [radius=1.5pt] {};
			\draw[fill] (3*\sqsize-\halfsqsize, 1*\sqsize) circle [radius=1.5pt] {};
			\draw[fill] (4*\sqsize-\halfsqsize, 1*\sqsize) circle [radius=1.5pt] {};
			\draw[fill] (5*\sqsize-\halfsqsize, 1*\sqsize) circle [radius=1.5pt] {};
			\draw[fill] (1*\sqsize-\halfsqsize, 2*\sqsize) circle [radius=1.5pt] {};
			\draw[fill] (2*\sqsize-\halfsqsize, 2*\sqsize) circle [radius=1.5pt] {};
			\draw[fill] (3*\sqsize-\halfsqsize, 2*\sqsize) circle [radius=1.5pt] {};
			\draw[fill] (4*\sqsize-\halfsqsize, 2*\sqsize) circle [radius=1.5pt] {};
			\draw[fill] (5*\sqsize-\halfsqsize, 2*\sqsize) circle [radius=1.5pt] {};
			\draw[fill] (1*\sqsize-\halfsqsize, 3*\sqsize) circle [radius=1.5pt] {};
			\draw[fill] (2*\sqsize-\halfsqsize, 3*\sqsize) circle [radius=1.5pt] {};
			\draw[fill] (3*\sqsize-\halfsqsize, 3*\sqsize) circle [radius=1.5pt] {};
			\draw[fill] (4*\sqsize-\halfsqsize, 3*\sqsize) circle [radius=1.5pt] {};
			\draw[fill] (5*\sqsize-\halfsqsize, 3*\sqsize) circle [radius=1.5pt] {};
			\draw[fill] (6*\sqsize-\halfsqsize, 3*\sqsize) circle [radius=1.5pt] {};
			\draw[fill] (1*\sqsize-\halfsqsize, 4*\sqsize) circle [radius=1.5pt] {};
			\draw[fill] (2*\sqsize-\halfsqsize, 4*\sqsize) circle [radius=1.5pt] {};
			\draw[fill] (3*\sqsize-\halfsqsize, 4*\sqsize) circle [radius=1.5pt] {};
			\draw[fill] (4*\sqsize-\halfsqsize, 4*\sqsize) circle [radius=1.5pt] {};
			\draw[fill] (5*\sqsize-\halfsqsize, 4*\sqsize) circle [radius=1.5pt] {};
			\draw[fill] (6*\sqsize-\halfsqsize, 4*\sqsize) circle [radius=1.5pt] {};
			\draw[fill] (1*\sqsize-\halfsqsize, -1*\sqsize) circle [radius=1.5pt] {};
			\draw[fill] (2*\sqsize-\halfsqsize, -1*\sqsize) circle [radius=1.5pt] {};
			\draw[fill] (3*\sqsize-\halfsqsize, -1*\sqsize) circle [radius=1.5pt] {};
			\draw[fill] (4*\sqsize-\halfsqsize, -1*\sqsize) circle [radius=1.5pt] {};
			\draw[fill] (7.5*\sqsize-\halfsqsize, 4*\sqsize) circle [radius=1.5pt] {};
			\draw[fill] (7.5*\sqsize-\halfsqsize, 3*\sqsize) circle [radius=1.5pt] {};
			\draw[fill] (7.5*\sqsize-\halfsqsize, 3*\sqsize) circle [radius=1.5pt] {};
			\draw[fill] (5*\sqsize-\halfsqsize, -.5*\sqsize) circle [radius=1.5pt] {};
			\draw[fill] (5.5*\sqsize-\halfsqsize, 0*\sqsize) circle [radius=1.5pt] {};
			\draw[fill] (6*\sqsize-\halfsqsize, .5*\sqsize) circle [radius=1.5pt] {};
			\draw[fill] (6.5*\sqsize-\halfsqsize, 1*\sqsize) circle [radius=1.5pt] node (w2c) {};
			\draw[fill] (7*\sqsize-\halfsqsize, 1.5*\sqsize) circle [radius=1.5pt] node (r2) {};
				\node (r2l) at (5.3,1) {\scriptsize$r_{y_6-t}$};
				\draw[->,thick] (r2l) -- (r2);
				\node (w2cl) at (4.9,.2) {\scriptsize$\check{w}_{x_3-t}$};
				\draw[->,thick] (w2cl) -- (w2c);
		\end{scope}
	\end{tikzpicture}
	\caption{The domains
	$\Omega_\lambda^{\usebox{\captionnw}}$ (left) and
	$\Omega_\lambda^{\usebox{\captionse}}$ (right), see
	\Cref{def:nw_domain,def:se_domain}. The partition functions
	in these domains depend on $x_i,y_j$, and $t$. They are
	equal to each other by the Yang--Baxter equation. Here
	$\lambda=(6,6,5,5,4)$, $\mu=(5,4,1)$, and the boundary
	binary string is $B(\lambda/\mu)=11011010001$. For each $1$
	in the binary string, we draw an incoming or an outgoing
	arrow for, respectively, a vertical or a horizontal edge. In
	$\Omega_\lambda^{\usebox{\captionse}}$, we 
	modified the way to draw the
	southeast border (while preserving the same intersections)
	for better visibility.}
	\label{fig:5V_Pieri_additional_strand}
\end{figure}

\begin{lemma}
	\label{lemma:add_YB_vertices}
	The partition function of the vertex model in 
	$\Omega^{\nwsymbol}_\lambda$ 
	is equal to $\mathsf{Z}_\mu(\lambda)$.
\end{lemma}
\begin{proof}
	Due to the arrow preservation at each vertex $(i,0)$, $1\le i\le \ell(\lambda)$,
	the empty boundary conditions along the west boundary of $\Omega^{\nwsymbol}_\lambda$
	lead to the empty boundary conditions entering $\Omega_\lambda$. 
	Note that $\check{w}_{x_i-t}(0,0;0,0)=1$,
	so the extra vertices at $(i,0)$ contribute a factor of $1$ to the partition function.
	Similarly, the arrow preservation 
	and the fact that $r_{y_j-t}(0,0;0,0)=1$
	imply that the north boundary of $\Omega_\lambda$
	gets an empty boundary condition, and the extra vertices at $(0,j)$ also contribute a factor of~$1$.
	Thus, the partition function in $\Omega^{\nwsymbol}_\lambda$ 
	reduces to the one in~$\Omega_\lambda$,
	which is equal to $\mathsf{Z}_\mu(\lambda)$.
\end{proof}

The lattice configuration in the extended domain $\Omega^{\nwsymbol}_\lambda$
now allows to apply the Yang--Baxter equation (\Cref{prop:YBE}).
That is, we start 
in $\Omega^{\nwsymbol}_\lambda$ 
at the triangle formed by the 
vertices $(1,1),(0,1)$, and $(1,0)$, and apply the Yang--Baxter equation to move the 
new strand one step in the southeast direction.
Continuing in this way, the strand sweeps the Young diagram $\lambda$,
and in the end it is located below
the southeast border of $\Omega_\lambda$. This results in a new domain
for the vertex model:

\begin{definition}[Domain $\Omega^{\sesymbol}_\lambda$]
	\label{def:se_domain}
	Let $\Omega^{\sesymbol}_\lambda$ 
	be obtained from the domain $\Omega_\lambda$
	by adding one more vertex to each horizontal and vertical edge
	along the southeast border of $\lambda$.
	Let these new vertices be connected by a single \defn{new strand}.
	When the new strand intersects a horizontal edge carrying a spectral parameter $x_i$
	or a vertical edge carrying a spectral parameter $y_j$,
	we assign the weight $\check{w}_{x_i-t}$ or $r_{y_j-t}$, respectively,
	to the new vertex on this edge.
	The southeast border of the new domain $\Omega^{\sesymbol}_\lambda$
	carries the binary string $B(\lambda/\mu)$, while the northwest boundary and the 
	boundary conditions on the new strand are empty.
	The weights inside $\Omega_\lambda$ are $w_{x_i-y_j}$, as before.
	See \Cref{fig:5V_Pieri_additional_strand}, right, for an illustration.
\end{definition}

Combining \Cref{lemma:add_YB_vertices} with the Yang--Baxter equation, we immediately obtain:

\begin{proposition}
	\label{prop:add_YB_vertices_omega_se}
	The partition function of the vertex model in $\Omega^{\sesymbol}_\lambda$
	is equal to $\mathsf{Z}_\mu(\lambda)$.
\end{proposition}

\begin{remark}
	In $\Omega^{\nwsymbol}_\lambda$, the 
	new strand may be thought of as the boundary
	of an empty Young diagram $\kappa$.
	Each application of the Yang--Baxter equation 
	when passing from 
	$\Omega^{\nwsymbol}_\lambda$ to $\Omega^{\sesymbol}_\lambda$
	may be thought of as adding a box to $\kappa$.
	When $\kappa$ becomes~$\lambda$, the new strand is 
	located below the southeast border of $\Omega_\lambda$.
	Since the Yang--Baxter equation is a local transformation, the order of adding boxes to $\kappa$
	in this growing process is irrelevant.
\end{remark}

\subsection{Boundary binary strings via Maya diagrams}
\label{sub:Maya_diagrams}

For a Young diagram $\lambda$, denote
\begin{equation}
	\label{eq:def_I_lambda}
	I(\lambda)\coloneqq 
	\{-\ell(\lambda),-\ell(\lambda)+1,\ldots,\lambda_1-2,\lambda_1-1\}\subset \mathbb{Z},
	\qquad 
	|I(\lambda)|=\lambda_1+\ell(\lambda).
\end{equation}
Encode the southeast border of $\lambda$ 
via its (zero-charge) \defn{Maya diagram}
\begin{equation}
	\label{eq:def_X_lambda}
	X(\lambda)\coloneqq \bigl\{\lambda_i-i
		\colon 1\le i\le \ell(\lambda)\bigr\}
		\subset I(\lambda).
\end{equation}
The vertical and horizontal edges 
along the southeast border of $\lambda$
correspond, respectively, to the elements of 
$X^c(\lambda)\coloneqq I(\lambda)\setminus X(\lambda)$ 
and 
$X(\lambda)$.
It is well-known that
\begin{equation}
	\label{eq:Xc_lambda}
	X^c(\lambda)=\bigl\{ -\lambda_j'+j-1\colon 1\le j\le \lambda_1 \bigr\},
\end{equation}
where $\lambda'$ is the transposed Young diagram of $\lambda$.

We have $I_{\varnothing}=X(\varnothing)=\varnothing$.
For our 
running
example $\lambda=(6,6,5,5,4)$,
we have 
\begin{equation*}
	I(\lambda)=\{-5,-4,\ldots,4,5 \}
	,\qquad 
	X(\lambda)=\{5,4,2,1,-1\},
	\qquad
	X^c(\lambda)=\{-5,-4,-3,-2,0,3\}.
\end{equation*}

Maya diagrams help demystify the boundary binary string $B(\lambda/\mu)$
from \Cref{def:mu_boundary_condition}:
\begin{proposition}
	\label{prop:delta_symmetric_difference}
	For any $\mu\subseteq\lambda$, we have
	\begin{equation}
		\label{eq:delta_symmetric_difference}
		B(\lambda/\mu)=X(\lambda)\operatorname{\Delta} X(\mu)\subseteq I(\lambda),
	\end{equation}
	where $\operatorname{\Delta}$ denotes the symmetric difference of sets,
	and we interpret the binary string as a subset of $I(\lambda)$.
\end{proposition}
\begin{remark}
	\label{rmk:X_mu_subset_I_lambda}
	In \eqref{eq:delta_symmetric_difference}
	and throughout the rest of \Cref{sec:YBE_proof_of_Naruse},
	we slightly abuse the notation 
	by
	appending $\mu\subseteq\lambda$ by zeros, if necessary, 
	such that
	the set
	$X(\mu)=\{\mu_i-i\colon i=1,2,\ldots \}$ is treated a subset of
	$I(\lambda)$.
	Note that $I(\mu)$ may be strictly inside
	$I(\lambda)$, but we never deal with the set $I(\mu)$ of the inner Young diagram $\mu$.

	The number of elements in $X(\lambda)$ is equal to 
	$\ell(\lambda)$. One can check that 
	for any $\mu\subseteq\lambda$,
	the number of elements of $X(\mu)$ 
	(viewed as a subset of $I(\lambda)$)
	is also equal to $\ell(\lambda)$.
\end{remark}
Continuing with our example $\lambda=(6,6,5,5,4)$, $\mu=(5,4,1)$, we have
\begin{equation*}
	X(\mu)=\{4,2,-2,-4,-5\}\subseteq I(\lambda),
	\qquad
	X(\lambda)\operatorname{\Delta} X(\mu)=\{-5, -4, -2, -1, 1, 5\},
\end{equation*}
which agrees with 
$B(\lambda/\mu)=11011010001$.
\begin{proof}[Proof of \Cref{prop:delta_symmetric_difference}]
	Throughout the proof, we treat all equalities between subsets of $\mathbb{Z}$
	as valid only when intersecting with $I(\lambda)=\{-\ell(\lambda),\ldots,\lambda_1-1,\lambda_1\}$
	(but 
	do not explicitly include this intersection in the notation).

	We argue by induction, by adding one box to $\mu$. The base case is $\mu=\varnothing$.
	The binary string $B(\lambda/\varnothing)$ arises from the usual hook decomposition of $\lambda$
	(cf. \Cref{fig:rim_hook_decomposition}).
	One readily sees that 
	\begin{equation*}
		B(\lambda/\varnothing)=\{\lambda_i-i\colon \lambda_i-i\ge 0\}\cup 
		\left\{ i-1-\lambda_i'\colon \lambda_i'-i+1\ge1 \right\}
		=X(\lambda)\operatorname{\Delta} \mathbb{Z}_{\le0},
	\end{equation*}
	as desired.

	\begin{figure}[htpb]
		\centering
		\begin{tikzpicture}
			[scale=1,very thick]

			\begin{scope}[shift={(0,0)}]
				\node at (-.5,1) {(a)};
				\node[anchor=north] at (1,  1.5) {$\mu$};
				\draw[line width=4pt, gray!30]  (.5,-.5)--++(0,1)--++(1,0)--++(0,-1)--cycle;
				\draw[yellow!60!black, line width=7pt] (.5,-.5)--++(0,1)--++(1,0);
				\draw[red!70!black, line width=7pt] (1.5,-.9)--++(0,.4)--++(.4,0);
				\draw[densely dotted, line width=2.1pt] (0,0) --++ (2,0);
				\draw[->,line width=2pt] (2.25,0)--++(.5,0);
				\node at (2.5,-1.5) {$10\to01$};
				\begin{scope}[shift={(3,0)}]
					\node[anchor=north] at (1,1.5) {$\nu$};
					\draw[line width=4pt, gray!30]  (.5,-.5)--++(0,1)--++(1,0)--++(0,-1)--cycle;
					\draw[yellow!60!black, line width=7pt] (.5,-.5)--++(1,0)--++(0,1);
					\draw[green!70!black, line width=7pt] (.1,.5) --++(.4,0)--++(0,.4);
					\draw[densely dotted, line width=2.1pt] (0,0) --++ (2,0);
				\end{scope}
			\end{scope}

			\begin{scope}[shift={(8,0)}]
				\node at (-.5,1) {(b)};
				\node[anchor=north] at (1,  1.5) {$\mu$};
				\draw[line width=4pt, gray!30]  (.5,-.5)--++(0,1)--++(1,0)--++(0,-1)--cycle;
				\draw[yellow!60!black, line width=7pt] (.5,-.5)--++(0,1)--++(1,0);
				\draw[red!70!black, line width=7pt] (1.5,-.9)--++(0,.4)--++(.4,0);
				\draw[densely dotted, line width=2.1pt] (0,0) --++ (1,0)--++(0,1);
	
				\draw[->,line width=2pt] (2.25,0)--++(.5,0);
				\node at (2.5,-1.5) {$11\to00$};
				
				\begin{scope}[shift={(3,0)}]
					\node[anchor=north] at (1,1.5) {$\nu$};
					\draw[line width=4pt, gray!30]  (.5,-.5)--++(0,1)--++(1,0)--++(0,-1)--cycle;
					\draw[yellow!60!black, line width=7pt] (.5,-.5)--++(1,0)--++(0,1);
					\draw[green!70!black, line width=7pt] (.1,.5) --++(.4,0)--++(0,.4);
					\draw[densely dotted, line width=2.1pt] (0,0) --++ (1,0)--++(0,1);
				\end{scope}
			\end{scope}
			
			\begin{scope}[shift={(0,-3.5)}]
				\node at (-.5,1) {(c)};
				\node[anchor=north] at (1,  1.5) {$\mu$};
				\draw[line width=4pt, gray!30]  (.5,-.5)--++(0,1)--++(1,0)--++(0,-1)--cycle;
				\draw[yellow!60!black, line width=7pt] (.5,-.5)--++(0,1)--++(1,0);
				\draw[red!70!black, line width=7pt] (1.5,-.9)--++(0,.4)--++(.4,0);
				\draw[densely dotted, line width=2.1pt] (1,-1) --++ (0,1)--++(1,0);
				
				\draw[->,line width=2pt] (2.25,0)--++(.5,0);
				\node at (2.5,-1.5) {$00\to11$};
				
				\begin{scope}[shift={(3,0)}]
					\node[anchor=north] at (1,1.5) {$\nu$};
					\draw[line width=4pt, gray!30]  (.5,-.5)--++(0,1)--++(1,0)--++(0,-1)--cycle;
					\draw[yellow!60!black, line width=7pt] (.5,-.5)--++(1,0)--++(0,1);
					\draw[green!70!black, line width=7pt] (.1,.5) --++(.4,0)--++(0,.4);
					\draw[densely dotted, line width=2.1pt] (1,-1) --++ (0,1)--++(1,0);
				\end{scope}
			\end{scope}

			\begin{scope}[shift={(8,-3.5)}]
				\node at (-.5,1) {(d)};
				\node[anchor=north] at (1,  1.5) {$\mu$};
				\draw[line width=4pt, gray!30]  (.5,-.5)--++(0,1)--++(1,0)--++(0,-1)--cycle;
				\draw[yellow!60!black, line width=7pt] (.5,-.5)--++(0,1)--++(1,0);
				\draw[red!70!black, line width=7pt] (1.5,-.9)--++(0,.4)--++(.4,0);
				\draw[densely dotted, line width=2.1pt] (1,-1) --++ (0,2);
				
				\draw[->,line width=2pt] (2.25,0)--++(.5,0);
				\node at (2.5,-1.5) {$01\to10$};				
				
				\begin{scope}[shift={(3,0)}]
					\node[anchor=north] at (1,1.5) {$\nu$};
					\draw[line width=4pt, gray!30]  (.5,-.5)--++(0,1)--++(1,0)--++(0,-1)--cycle;
					\draw[yellow!60!black, line width=7pt] (.5,-.5)--++(1,0)--++(0,1);
					\draw[green!70!black, line width=7pt] (.1,.5) --++(.4,0)--++(0,.4);
					\draw[densely dotted, line width=2.1pt] (1,-1) --++ (0,2);
				\end{scope}
			\end{scope}
		\end{tikzpicture}
		\caption{Four cases of adding a box $\nu=\mu+\square$
			in the proof of \Cref{prop:delta_symmetric_difference}.
		The dashed line is the southeast boundary of $\lambda$, and thick lines are $\mu$- or $\nu$-rim-hooks.
		Below each case, we indicate the local change in the boundary binary
		string, $B(\lambda/\mu)\to B(\lambda/\nu)$.}
		\label{fig:delta_symmetric_difference_proof}
	\end{figure}
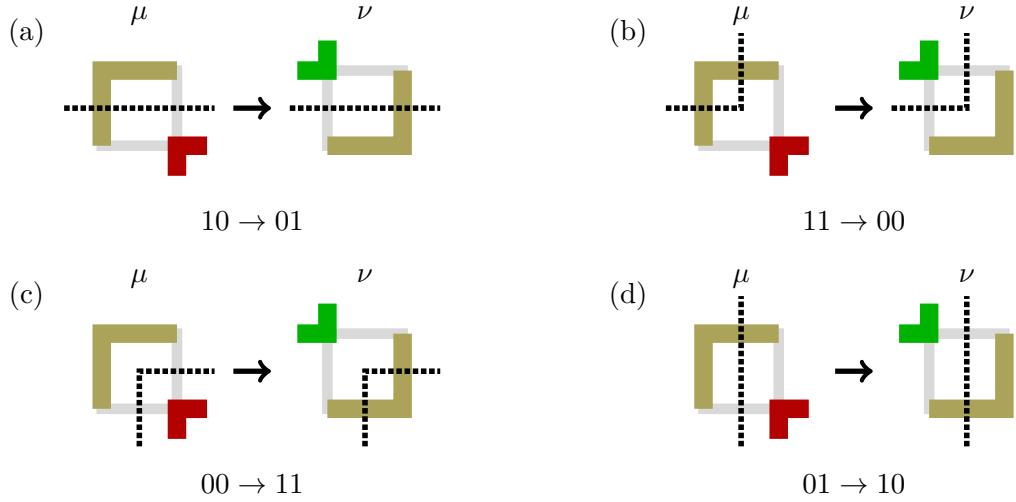

	Now let $\mu,\nu\subseteq \lambda$ are such that $\nu=\mu+\square$.
	In terms of Maya diagrams, this means that for some $k$, 
	\begin{equation}
		\label{eq:delta_symmetric_difference_proof}
		k\in X(\mu), \quad k+1\notin X(\mu), \qquad 
		X(\nu)=\left( X(\mu)\cup \left\{ k+1 \right\} \right)\setminus \left\{ k \right\}.
	\end{equation}
	There are four cases depending on whether $k$ and $k+1$ belong to $X(\lambda)$.
	They are	
	illustrated by local pictures in \Cref{fig:delta_symmetric_difference_proof}
	(an example of a 
	global rim-hook configuration is in \Cref{fig:rim_hook_decomposition}).
	The four cases correspond to four possible directions of the southeast border of $\lambda$
	through $k$ and $k+1$. Indeed, 
	in (a) we have $k,k+1\notin X(\lambda)$, and the boundary goes horizontally. 
	The other cases are (b) $k\notin X(\lambda)$, $k+1\in X(\lambda)$;
	(c) $k\in X(\lambda)$, $k+1\notin X(\lambda)$; and (d) $k,k+1\in X(\lambda)$.
	
	From the induction assumption, it follows 
	that the configuration of $\mu$-rim-hooks 
	around the part of the southeast border of $\lambda$
	through $k$ and $k+1$ is the same in all four cases.
	Adding a box to $\mu$ changes the $\mu$-rim-hook configuration
	to $\nu$-rim-hooks
	in the same way in all cases,
	which results in the corresponding change of the boundary binary string 
	$B(\lambda/\mu)\to B(\lambda/\nu)$.
	This completes the proof.
\end{proof}

\subsection{Vertical strip expansion of the five-vertex partition function}
\label{subsec:5V_expansion}

\begin{definition}
	\label{def:strand_transfer_matrix}
	Fix a Young diagram $\lambda$.
	Let us define a \defn{transfer matrix}
	$\mathscr{R}_\lambda^t$
	which depends on the 
	spectral parameter~$t$ (and also on $x_i,y_j$, but we suppress this in the notation),
	and has rows and columns indexed by Young diagrams $\mu,\nu\subseteq\lambda$.
	The value 
	$\mathscr{R}^t_\lambda(\mu,\nu)$ is
	a partition function of a single-row vertex model 
	whose 
	vertices are indexed by $I(\lambda)$
	\eqref{eq:def_I_lambda}.
	The vertex weight at each $k\in I(\lambda)$ has the form
	\begin{equation}
		\label{eq:vertex_weights_for_R_transfer_matrix}
		\begin{cases}
			r_{x_i-t}, & k= \lambda_i-i\in X(\lambda)
			;\\
			r_{y_j-t}, & k= -\lambda_j'+j-1\in X^c(\lambda).
		\end{cases}
	\end{equation}
	The boundary conditions on the left and right of the row are empty,
	and 
	boundary conditions on the top and bottom
	are given by $X(\mu)$ and $X(\nu)$ (viewed as subsets of $I(\lambda)$), respectively.

	\begin{remark}
		\label{rmk:how_to_assign_spectral_parameters}
		The choice of a spectral parameter $x_i-t$ or $y_j-t$ at a point in $k\in I(\lambda)$
		can be uniformly written as
		\begin{equation*}
			\mathrm{parameter}(k)\coloneqq 
			x_{|X(\lambda)\operatorname{\cap} \mathbb{Z}_{\ge k}|}\mathbf{1}_{k\in X(\lambda)}
			+
			y_{|X^c(\lambda)\operatorname{\cap} \mathbb{Z}_{\le k}|} \mathbf{1}_{k\in X^c(\lambda)}
			-t.
		\end{equation*}
	\end{remark}
	
	Clearly, for each $\mu,\nu$, there is at most one path configuration
	with these boundary conditions.
	If there are no path configurations,
	we set $\mathscr{R}^t_\lambda(\mu,\nu)=0$, and otherwise we let $\mathscr{R}^t_\lambda(\mu,\nu)$
	to be the product of the weights of all vertices along $I(\lambda)$.
	See \Cref{fig:5V_Pieri_transfer_matrix} for an illustration.
\end{definition}

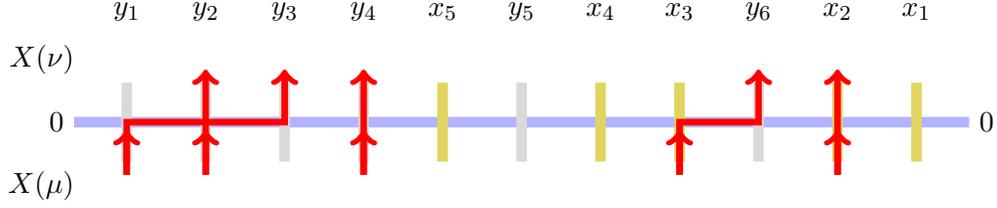
\begin{figure}[htpb]
	\centering
	\begin{tikzpicture}
		[scale=0.7]
		\draw[line width=4pt, blue!30] (-1,0)--(16,0);
		\foreach \ii in {0,1,2,3,5,8}
		{
			\draw[line width=4pt, gray!30] (\ii*1.5,-.75)--++(0,1.5);
		}
		\foreach \ii in {4,6,7,9,10}
		{
			\draw[line width=4pt, gray!30!yellow] (\ii*1.5,.75)--++(0,-1.5);
		}
		\node[left] at (-1,0) {0};
		\node[right] at (16,0) {0};
		\node at (0,2.1) {$y_1$};
		\node at (1.5,2.1) {$y_2$};
		\node at (3,2.1) {$y_3$};
		\node at (4.5,2.1) {$y_4$};
		\node at (6,2.1) {$x_5$};
		\node at (7.5,2.1) {$y_5$};
		\node at (9,2.1) {$x_4$};
		\node at (10.5,2.1) {$x_3$};
		\node at (12,2.1) {$y_6$};
		\node at (13.5,2.1) {$x_2$};
		\node at (15,2.1) {$x_1$};
		\node[left] at (-.8,-1.2) {$X(\mu)$}; 
		\node[left] at (-.8,1.2) {$X(\nu)$}; 
		\foreach \ii in {4,2,-2,-4,-5}
		{
			\draw[->,red,line width=2.5pt] (5*1.5+\ii*1.5,-1)--++(0,.85);
		}
		\draw[->,red,line width=2.5pt] (0*1.5,-1)--++(0,1)--++(1.5,0)--++(0,1);
		\draw[->,red,line width=2.5pt] (1*1.5,-1)--++(0,1)--++(1.5,0)--++(0,1);
		\draw[->,red,line width=2.5pt] (3*1.5,-1)--++(0,2);
		\draw[->,red,line width=2.5pt] (7*1.5,-1)--++(0,1)--++(1.5,0)--++(0,1);
		\draw[->,red,line width=2.5pt] (9*1.5,-1)--++(0,2);
	\end{tikzpicture}
	\caption{The one-row partition function for
	$\mathscr{R}^t_\lambda(\mu,\nu)$ with $\lambda=(6,6,5,5,4)$, $\mu=(5,4,1)$, and $\nu=(5,5,1,1,1)$.
	The different colors of the vertical edges correspond to the different spectral parameters $x_i-t$ 
	or $y_j-t$ in the vertex weights, see \eqref{eq:vertex_weights_for_R_transfer_matrix}.
	The sequence of spectral parameters depends only on $\lambda$.}
	\label{fig:5V_Pieri_transfer_matrix}
\end{figure}

Recall that a vertical strip is a skew Young diagram which has at most one
box in each row.
We have
the following expansion of the 
five-vertex model
partition functions 
$\mathsf{Z}_\mu(\lambda)$:
\begin{proposition}
	\label{prop:5V_Pieri_expansion}
	For any $\mu\subseteq\lambda$,
	we have
	\begin{equation}
		\label{eq:5V_Pieri_expansion}
		\mathsf{Z}_\mu(\lambda)=
		\frac{1}{(x_1-t)\ldots(x_{\ell(\lambda)}-t) }
		\sum_{\substack{\nu\subseteq\lambda\\ \nu=\mu+\textnormal{vertical strip}}}
		\mathscr{R}^t_\lambda(\mu,\nu)\ssp
		\mathsf{Z}_\nu(\lambda),
	\end{equation}
	where $\mathscr{R}^t_\lambda$ is the transfer matrix from \Cref{def:strand_transfer_matrix}.
	The vertical strip in \eqref{eq:5V_Pieri_expansion} 
	can be empty.
\end{proposition}
\begin{proof}[Proof of \Cref{prop:5V_Pieri_expansion}]
	We start from
	\Cref{prop:add_YB_vertices_omega_se}
	which states that 
	$\mathsf{Z}_\mu(\lambda)$ is the partition function
	of the vertex model in the domain $\Omega_\lambda^{\sesymbol}$
	(see \Cref{def:se_domain}),
	with the boundary conditions $B(\lambda/\mu)$ along the extra new strand
	of vertices carrying the weights 
	$r_{y_j-t}$ and $\check{w}_{x_i-t}$
	(see \Cref{fig:5V_Pieri_additional_strand}, right).
	Peeling off this extra strand
	and summing over the 
	binary strings between the strand and the southeast border of~$\lambda$,
	we immediately get the following expansion:
	\begin{equation}
		\label{eq:5V_Pieri_expansion_proof}
		\mathsf{Z}_\mu(\lambda)=
		\sum_{\nu\subseteq\lambda}
		\mathscr{T}_\lambda^t(\mu,\nu)
		\ssp
		\mathsf{Z}_\nu(\lambda).
	\end{equation}
	Indeed, $\mathsf{Z}_\nu(\lambda)$ is the partition function
	of the five-vertex model in $\Omega_\lambda$ with some boundary conditions.
	The coefficients 
	$\mathscr{T}_{\lambda}^t(\mu,\nu)$ are determined from one-row partition functions
	with the following data:
		\begin{enumerate}[$\bullet$]
		\item The vertices on the row are indexed by $I(\lambda)$.
		\item At each $\lambda_i-i\in X(\lambda)$, we put the \defn{reversed} weight $\check{w}_{x_i-t}$.
			Namely, paths at this vertex are oriented \defn{down and right}.
		\item At each $-\lambda_j'+j-1\in X^c(\lambda)$, we put the \defn{usual} weight $r_{y_j-t}$,
			with the \defn{up and right} path orientation.
		\item The boundary conditions on the left and right of the row are empty.
		\item The boundary conditions on the top and bottom of the row 
			are given by the binary strings $B(\lambda/\mu)$ and $B(\lambda/\nu)$, respectively.
	\end{enumerate}
	See \Cref{fig:5V_Pieri_transfer_matrix_T_matrix} for an illustration.
	
	\begin{figure}[htpb]
		\centering
		\begin{tikzpicture}
			[scale=0.7]
			\draw[line width=4pt, blue!30,->] (-1,0)--(16,0);
			\foreach \ii in {0,1,2,3,5,8}
			{
				\draw[line width=4pt, gray!30,->] (\ii*1.5,-.75)--++(0,1.5);
				\draw[fill] (\ii*1.5,0) circle [radius=2.5pt] {};
			}
			\foreach \ii in {4,6,7,9,10}
			{
				\draw[line width=4pt, gray!30,->] (\ii*1.5,.75)--++(0,-1.5);
				\draw[fill] (\ii*1.5,0) circle [radius=2.5pt] {};
			}
			\node[left] at (-1,0) {0};
			\node[right] at (16,0) {0};
			\node at (0,2.4) {$y_1$};
			\node at (1.5,2.4) {$y_2$};
			\node at (3,2.4) {$y_3$};
			\node at (4.5,2.4) {$y_4$};
			\node at (6,2.4) {$x_5$};
			\node at (7.5,2.4) {$y_5$};
			\node at (9,2.4) {$x_4$};
			\node at (10.5,2.4) {$x_3$};
			\node at (12,2.4) {$y_6$};
			\node at (13.5,2.4) {$x_2$};
			\node at (15,2.4) {$x_1$};
			\node at (0,-2.4) {$r$};
			\node at (1.5,-2.4) {$r$};
			\node at (3,-2.4) {$r$};
			\node at (4.5,-2.4) {$r$};
			\node at (6,-2.4) {$\check{w}$};
			\node at (7.5,-2.4) {$r$};
			\node at (9,-2.4) {$\check{w}$};
			\node at (10.5,-2.4) {$\check{w}$};
			\node at (12,-2.4) {$r$};
			\node at (13.5,-2.4) {$\check{w}$};
			\node at (15,-2.4) {$\check{w}$};
			\node[left] at (-.8,-1.2) {$B(\lambda/\mu)$}; 
			\node[left] at (-.8,1.2) {$B(\lambda/\nu)$}; 
			\foreach \ii in {0,1,3,4,6,10}
			{
				\node at (\ii*1.5,-1.1) {$1$};
			}
			\foreach \ii in {2,5,7,8,9}
			{
				\node at (\ii*1.5,-1.1) {$0$};
			}
			\foreach \ii in {1,2,3,4,6,7,8,10}
			{
				\node at (\ii*1.5,1.1) {$1$};
			}
			\foreach \ii in {0,5,9}
			{
				\node at (\ii*1.5,1.1) {$0$};
			}
		\end{tikzpicture}
		\caption{The one-row partition function for 
			the coefficients $\mathscr{T}_\lambda^t(\mu,\nu)$ 
			in \eqref{eq:5V_Pieri_expansion_proof}
			with the same $\lambda,\mu,\nu$ as in
			\Cref{fig:5V_Pieri_transfer_matrix}.
			The up and down arrows indicate the orientation of the vertical paths at the vertices.
			Note that
			the horizontal paths are always oriented to the right.
			Zeroes and ones indicate the boundary conditions $B(\lambda/\mu)$ and $B(\lambda/\nu)$.}
		\label{fig:5V_Pieri_transfer_matrix_T_matrix}
	\end{figure}
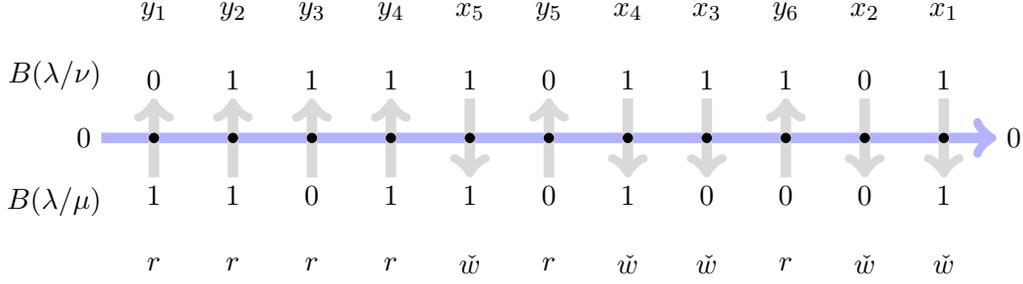

	In the partition function $\mathscr{T}_\lambda^t(\mu,\nu)$,
	we now \defn{reverse} the orientation of all vertical edges 
	carrying the weights $\check{w}_{x_i-t}$.
	We obtain new weights which have the form
	\begin{equation*}
		\check{w}_{x_i-t}
		\Bigl( 
				\begin{tikzpicture}[baseline=-3,scale=.7,very thick]
				\draw[line width=4pt, gray!30,->] (0,.75)--++(0,-1.5);
				\draw[line width=4pt, blue!30,->] (-1,0)--++(2,0);
				\draw[fill] (0,0) circle [radius=2.5pt] {};
				\node[left] at (0,.75) {\scriptsize$i_1$};
				\node[right] at (0,-.75) {\scriptsize$i_2$};
				\node[below] at (-.8,0) {\scriptsize$j_1$};
				\node[above] at (.8,.2) {\scriptsize$j_2$};
		\end{tikzpicture}
		\Bigr)
		=\frac{r_{x_i-t}(1-i_2,j_1;1-i_1,j_2)}{x_i-t},\qquad 
		i_1,j_1,i_2,j_2\in\left\{ 0,1 \right\},
		\quad i=1,2,\ldots,\ell(\lambda).
	\end{equation*}
	By 
	\Cref{prop:delta_symmetric_difference},
	this reversal modifies the bottom and top boundary
	conditions
	to 
	\begin{equation*}
		B(\lambda/\mu)\operatorname{\Delta} X(\lambda)=X(\mu),
		\qquad 
		B(\lambda/\nu)\operatorname{\Delta} X(\lambda)=X(\nu).
	\end{equation*}
	Thus, we conclude that 
	\begin{equation*}
		\mathscr{T}_\lambda^t(\mu,\nu)=
		\frac{\mathscr{R}^t_\lambda(\mu,\nu)}{(x_1-t)\ldots(x_{\ell(\lambda)}-t) }
		.
	\end{equation*}
	
	It remains to show that the sum over 
	$\nu\subseteq\lambda$ in \eqref{eq:5V_Pieri_expansion_proof}
	is restricted to $\nu$ obtained from $\mu$ by
	adding a vertical strip.
	This follows from the fact that 
	$r_z(0,1;0,1)=0$, 
	which implies that horizontal paths in \Cref{fig:5V_Pieri_transfer_matrix}
	cannot travel by more than one horizontal step. 
	This restriction implies that 
	$\mathscr{R}^t_\lambda(\mu,\nu)$
	vanishes unless $\nu=\mu+\text{vertical strip}$,
	and so we are done.
\end{proof}

\subsection{Pieri rule and proof \Cref{thm:general}}
\label{subsec:5V_Pieri}

We are now ready to establish the Pieri rule for the five-vertex partition functions
$\mathsf{Z}_\mu(\lambda)$ \eqref{eq:def_excited_sum}.
Together with vanishing (\Cref{rmk:excited_diagrams_vanishing}),
the 
general approach of \Cref{sub:general_formalism_general} then guarantees that 
$\mathsf{Z}_\mu(\lambda)$ is expressed as a 
sum over skew standard Young tableaux of shape $\lambda/\mu$.
This would complete the proof of the multivariate hook-length formula 
\eqref{eq:general2}.

\begin{definition}
	\label{def:projection_of_mu_onto_lambda}
	Let $\mu\subseteq\lambda$ be two Young diagrams.
	Define 
	\begin{equation}
		\label{eq:pieri_p_mu_def}
		\mathsf{p}_\mu(\lambda)\coloneqq
		\sum_{k\in X^c(\mu)\cap X(\lambda)} x_{|X(\lambda)\operatorname{\cap} \mathbb{Z}_{\ge k}|}
		-
		\sum_{k\in X(\mu)\cap X^c(\lambda)} y_{|X^c(\lambda)\operatorname{\cap} \mathbb{Z}_{\le k}|}
		.
	\end{equation}
\end{definition}
For example,
for $\lambda=(6,6,5,5,4)$ and $\mu=(5,4,1)$, we have
\begin{equation*}
	\mathsf{p}_\mu(\lambda)=
	(x_1+x_4+x_5)
	-
	(y_1+y_2+y_4)
	.
\end{equation*}

\begin{proposition}[Pieri rule for five-vertex partition functions]
	\label{prop:Pieri_rule}
	Let $\mu\subseteq\lambda$
	be two Young diagrams.
	Then we have
	\begin{equation*}
		\mathsf{p}_\mu(\lambda)\ssp
		\mathsf{Z}_\mu(\lambda)=
		\sum_{\substack{\nu\subseteq\lambda\\ \nu=\mu+\square}}
		\mathsf{Z}_\nu(\lambda).
	\end{equation*}
\end{proposition}
\begin{proof}
	We employ \Cref{prop:5V_Pieri_expansion}
	and consider the behavior
	of identity
	\eqref{eq:5V_Pieri_expansion}
	as $t\to\infty$.
	Since $\mathsf{Z}_\mu(\lambda)$ and 
	$\mathsf{Z}_\nu(\lambda)$ do not depend on $t$, 
	it suffices to look at the transfer matrix $\mathscr{R}^t_\lambda(\mu,\nu)$
	defined as the one-row partition function 
	(see \Cref{fig:5V_Pieri_transfer_matrix}).
	
	Recall (\Cref{rmk:X_mu_subset_I_lambda}) that the 
	number of paths in the one-row vertex model
	for $\mathscr{R}^t_\lambda(\mu,\nu)$ 
	is equal to $\ell(\lambda)$.
	First, observe that 
	\begin{equation}
		\label{eq:R_transfer_matrix_asymptotics_proof}
		\mathscr{R}^t_\lambda(\mu,\mu)=
		\prod_{k\in X(\mu)}
		\begin{cases}
			x_i-t, & k=\lambda_i-i\in X(\lambda);\\
			y_j-t, & k=-\lambda_j'+j-1\in X^c(\lambda),
		\end{cases}
	\end{equation}
	which behaves as 
	$t\to+\infty$ as follows:
	\begin{equation}
		\label{eq:R_transfer_matrix_asymptotics_proof3}		
		(-t)^{\ell(\lambda)}+(-t)^{\ell(\lambda)-1}
		\biggl( \ssp
		\sum_{k\in X(\mu)\cap X(\lambda)} x_{|X(\lambda)\operatorname{\cap} \mathbb{Z}_{\ge k}|}
		+
		\sum_{k\in X(\mu)\cap X^c(\lambda)} y_{|X^c(\lambda)\operatorname{\cap} \mathbb{Z}_{\le k}|}
		\biggr)
		+O(t^{\ell(\lambda)-2}).
	\end{equation}
	Indeed, the factors in \eqref{eq:R_transfer_matrix_asymptotics_proof}
	are in one-to-one correspondence with the summands 
	by $(-t)^{\ell(\lambda)-1}$ in \eqref{eq:R_transfer_matrix_asymptotics_proof3},
	cf. \Cref{rmk:how_to_assign_spectral_parameters}.

	Next, for any $\nu$ with $|\nu|>|\mu|$, we have
	\begin{equation}
		\label{eq:R_transfer_matrix_asymptotics_proof2}
		\mathscr{R}^t_\lambda(\mu,\nu)=(-t)^{\ell(\lambda)-|\nu|+|\mu|}+O(t^{\ell(\lambda)-|\nu|+|\mu|-1})
		, \qquad t\to\infty.
	\end{equation}
	Indeed, $|\nu|-|\mu|$ is the number of occupied horizontal 
	edges in the vertex model for $\mathscr{R}^t_\lambda(\mu,\nu)$.
	Placing each extra occupied horizontal edge
	exchanges one
	weight $r_{y_j-t}(1,0;1,0)=y_j-t$ or $r_{x_i-t}(1,0;1,0)=x_i-t$ (growing with $t$)
	by a product of other $r$ weights. All other $r$ weights are equal to $0$ or $1$ 
	(see \eqref{eq:r_weights}). This produces \eqref{eq:R_transfer_matrix_asymptotics_proof2}.

	Let us now combine the asymptotics \eqref{eq:R_transfer_matrix_asymptotics_proof3}, 
	\eqref{eq:R_transfer_matrix_asymptotics_proof2} with the prefactor 
	in \eqref{eq:5V_Pieri_expansion},
	\begin{equation*}
		\frac{1}{(x_1-t)\ldots(x_{\ell(\lambda)}-t) }=
		\left( -t \right)^{-\ell(\lambda)}\Bigl( 1+t^{-1}\sum_{i=1}^{\ell(\lambda)}x_i \Bigr)
		+O(t^{-\ell(\lambda)-2}), \qquad  t\to\infty.
	\end{equation*}
	We see that we can cancel out the overall multiplicative factor $(-t)^{\ell(\lambda)}$. After that, the constant terms in both sides are equal to $\mathsf{Z}_\mu(\lambda)$, which cancel out. Equating the terms of order $t^{-1}$, we obtain the desired Pieri rule.
\end{proof}

\begin{proof}[Proof of \Cref{thm:general}]
	The Pieri rule of \Cref{prop:Pieri_rule} 
	together with the vanishing (\Cref{rmk:excited_diagrams_vanishing})
	and the general result of \Cref{prop:general_skew_SYT_sum}
	imply that 
	\begin{equation*}
		\sum_{T\in \ssp\SYT(\lambda/\mu)}
		\prod_{m=1}^{|\lambda/\mu|}
		\frac{1}{\mathsf{p}_{T^{-1}[< m]}(\lambda)}
		=
		\sum_{D \in \mathcal{E}(\la/\mu)} \prod_{(i,j) \in \la \setminus D} \frac{1}{x_i-y_j}.
	\end{equation*}
	Note that the Pieri coefficients $\mathsf{C}_{\nu/\mu}$ are equal to $1$
	in our case.
	We also employed the definition of $\mathsf{Z}_\mu(\lambda)$
	as a sum over excited diagrams~\eqref{eq:F_mu_is_excited_RHS_of_Naruse}, and 
	cancelled out the factor
	$\mathsf{Z}_{\lambda}(\lambda)$~\eqref{eq:F_lambda_lambda_product}.

	For any $m$, let us denote $T^{-1}[< m]$ by $\nu$. 
	Starting from
	\eqref{eq:pieri_p_mu_def}, we can rewrite
	\begin{equation*}
		\begin{split}
			\mathsf{p}_{\nu}( \lambda )
			&=
			\sum_{k\in X^c(\nu)\cap X(\lambda)} x_{|X(\lambda)\operatorname{\cap} \mathbb{Z}_{\ge k}|}
			-
			\sum_{k\in X(\nu)\cap X^c(\lambda)} y_{|X^c(\lambda)\operatorname{\cap} \mathbb{Z}_{\le k}|}
			\\&=
			\sum_{i=1}^{\ell(\lambda)}
			x_i
			-
			\biggl(\ssp
				\sum_{k\in X(\nu)\cap X(\lambda)} x_{|X(\lambda)\operatorname{\cap} \mathbb{Z}_{\ge k}|}
				+
				\sum_{k\in X(\nu)\cap X^c(\lambda)} y_{|X^c(\lambda)\operatorname{\cap} \mathbb{Z}_{\le k}|}
			\biggr)\\&
			=
			\sum_{i=1}^{\ell(\lambda)} x_i-
			\sum_{j=1}^{\ell(\lambda)} b_{\nu_j-j}.
		\end{split}
	\end{equation*}
	Here
	\begin{equation*}
		b_j\coloneqq 
		\begin{cases}
			x_i,& j=\lambda_i-i;\\
			y_k,& j=n_k,
		\end{cases}
	\end{equation*}
	with the notation
	\begin{equation*}
		\{n_1<\ldots<n_{\lambda_1} \}=\{-\ell(\lambda),-\ell(\lambda)+1,\ldots,\lambda_1-2,\lambda_1-1\}
		\setminus\{\lambda_1-1,\ldots,\lambda_{\ell(\lambda)}-\ell(\lambda) \}.
	\end{equation*}
	We see that the expressions
	$\mathsf{p}_\nu(\lambda)
	=
	\sum_{i=1}^{\ell(\lambda)} x_i-
	\sum_{j=1}^{\ell(\lambda)} b_{\nu_j-j}$,
	where 
	$\mu\subseteq\nu\subseteq\lambda$,
	coincide with the factors
	in the denominator in the left-hand side of the 
	multivariate hook-length formula~\eqref{eq:general2}.
	This completes the proof of \Cref{thm:general}.
\end{proof}

\appendix

\section{A semistandard variant}
\label{app:SSYT_variant}

Let us modify the polynomials $f_j(u)$
\eqref{eq:f_factorial_schur} from \Cref{sec:contour_integrals},
and investigate the resulting contour integrals
defined in the same way as in \eqref{eq:integral_def}.
Denote the integrals by $J$ to avoid confusion.
Let $\beta$ be a parameter,
$a=(a_1,a_2,\ldots)$ be a sequence of parameters as before, and
$\mathsf{m}=(m_1,m_2,\ldots)$ be a sequence of nonnegative integers.
Set 
\begin{equation*}
	f^\beta_r(u \mid a) = \prod_{i=1}^r (u+a_i+\beta u a_i),
\end{equation*}
and
\begin{equation}
	\label{eq:groth_integral}
J_{\mu,\mathsf{m}}(x \mid a) \coloneqq
\frac{ (-1)^{\binom{n}{2}} }{(2\pi\sqrt{-1})^n} \oint_\gamma\ldots
\oint_\gamma
\prod_{i=1}^n \frac{ f^\beta_{\mu_i+m_i-i}(u_i \mid a)}{\prod_{j=1}^{m_i} (u_i-x_j)} \ssp\Delta(u)\ssp du_1 \cdots du_n.
\end{equation}
The contours $\gamma$ are the same as in \Cref{sec:contour_integrals},
they go around all the poles $x_i$ in the positive direction.
Set $m_i=n$ for all $i$, and omit $\mathsf{m}$ 
from the notation. 

Arguing as in \Cref{sub:pieri_rule_by_contour},
we obtain a Pieri-type rule:
\begin{align*}
    &\sum_{\epsilon \in \{0,1\}^n} \beta^{|\epsilon|} J_{\mu + \epsilon}( x \mid a) 
    = \frac{1}{(2\pi \sqrt{-1})^n}  
		\oint_\gamma\ldots\oint_\gamma  \prod_{i=1}^n \frac{ f^\beta_{\mu_i+n-i}(u_i \mid a) }{\prod_{j=1}^{n} (u_i-x_j) } \\
		&\hspace{80pt} \times \prod_{i=1}^n \left(1+ \beta(u_i+a_{\mu_i+1+n-i} + \beta u_i a_{\mu_i+1+n-i}) \right) \Delta(u) \, du_1 \cdots du_n \\
		&\hspace{20pt}= \frac{1}{(2\pi \sqrt{-1})^n}  
		\oint_\gamma\ldots\oint_\gamma \prod_{i=1}^n \frac{ f^\beta_{\mu_i+n-i}(u_i \mid a) }{\prod_{j=1}^{n} (u_i-x_j) } \prod_{i=1}^{n} (1+ \beta u_i)(1+\beta a_{\mu_i+1+n-i} ) \Delta(u) \, du_1 \cdots du_n \\
		&\hspace{20pt}= J_\mu(x \mid a) \prod_{i=1}^n (1+\beta x_i) (1+\beta a_{\mu_i+1+n-i}).
\end{align*}
Here $|\epsilon|=\sum_{i=1}^n \epsilon_i$.
Next, if $\hat{\mu}=\mu + \epsilon$ is not a partition, i.e.
$\mu_i +\epsilon_i < \mu_{i+1} + \epsilon_{i+1}$, we must
have $\mu_i=\mu_{i+1}$, $\epsilon_i=0$ and
$\epsilon_{i+1}=1$, and so
$f_{\hat{\mu}_i+n-i}(u \mid a) =
f_{\hat{\mu}_{i+1}+n-(i+1)}(u \mid a)$. This makes the
integral $0$ by skew symmetry.
Therefore, the Pieri-type rule takes the form:
\begin{equation}\label{eq:ssytJ_pieri}
	\sum_{\nu} J_{\nu} (x\mid a) 
	=
	J_\mu (x \mid a)  
	\prod_{i=1}^n
	(1+\beta x_i) (1+\beta a_{\mu_i+1+n-i}) ,
\end{equation}
where the sum is over all partitions $\nu$ obtained from $\mu$ by adding 
a (possibly employ) vertical strip. 

Similarly to \Cref{thm:integral}, integral
\eqref{eq:groth_integral} can be rewritten as
a determinant of the $f^\beta_i$'s: 
\begin{equation}
	\label{eq:groth_determinant_formula}
	\begin{split}
		J_{\mu}(x \mid a) & = \frac{ (-1)^{\binom{n}{2}} }{(2\pi\sqrt{-1})^n} \oint_\gamma\ldots\oint_\gamma \prod_{i=1}^n \frac{ f^\beta_{\mu_i+n-i}(u_i \mid a)}{\prod_{j=1}^{n} (u_i-x_j)} \ssp\Delta(u)\ssp du_1 \cdots du_n 
		 \\&= \frac{1}{\Delta(x)} \det \Biggl[ \prod_{r=1}^{\mu_j+n-j} (x_i +a_r+\beta x_ia_r)\Biggr]_{i,j=1}^n.
	 \end{split}
\end{equation}
In particular, when $\mu=\varnothing$ only the maximal degree terms in this determinant survive, so
\begin{align}\label{eq:empty_mu}
	J_{\varnothing}(x \mid a) = \prod_{i=1}^{n} (1+\beta a_i)^{n-i}.
\end{align}
\begin{remark}
	The determinantal formula~\eqref{eq:groth_determinant_formula} is similar to the one for factorial Grothendieck polynomials of 
	\cite{mcnamara2006factorial}
	or \cite{hwang2021refined}. However, in order to obtain the Grothendieck polynomials one needs to replace the polynomials $f_{\mu_i+n-i}^\beta (u \mid a)$ with $(1+\beta u)^{i-1}\prod_{j=1}^{\mu_i+n-i}(u+a_j+\beta u a_j)$. The approach outlined here would lead to the identities in~\cite{MPP4GrothExcited} after some tedious manipulations.  
\end{remark}

Consider now the vanishing of $J_\mu(x \mid a)$ for certain values of $x$. Let $x^\la_i\coloneqq  
- \frac{ a_{\la_i+n-i+1} }{1+\beta a_{\la_i+n-i+1}}$. Then $f^\beta_{\mu_i+n-i}(x^\la_j \mid a) =0$ if $\la_j+n-j+1+1 \leq \mu_i+n-i$. 

Let $\lambda$ be such that for some $i$, we have $\la_i+n-i < \mu_i+n-i$, i.e., $\la_i<\mu_i$. Then we have $f^\beta_{\mu_r+n-r}(x_j^\la\mid a)=0$ for $r \leq i$ and $j \geq i$, which
implies that 
$\det [f^\beta_{\mu_i+n-i}(x_j^\la \mid a)]_{i,j=1}^n=0$. 
On the other hand, if $\lambda=\mu$, then the matrix is lower triangular.
This implies
\begin{lemma}[Vanishing property]
     Let $x^\la_i\coloneqq -\frac{ a_{\la_i+n-i+1} }{1+\beta a_{\la_i+n-i+1}}$. Then
        $$J_\mu(x^\la \mid a) = 
				\begin{cases} 0, &\text{ if } \mu \not \subset \la; \\
					\displaystyle \prod_{i=1}^n \prod_{j=1}^{\mu_i+n-i} \frac{ a_j - a_{\la_i+n-i+1} }{(1+\beta a_{\la_i+n-i+1} )}, & \text{ if } \la=\mu .
        \end{cases}$$
 \end{lemma}
The Pieri-type rule~\eqref{eq:ssytJ_pieri} can be rewritten
as $(\cdots) J_\mu = \sum_{ \nu \supset \mu} J_\nu$, where
the sum is over all~$\nu$ such that $\nu /\mu$ is a nonempty
vertical strip. Iterating this identity as 
in \Cref{sub:general_formalism_general},
we get the following result.
 \begin{theorem}\label{thm:ssytJ}
		 Let $\mu \subset \la$ and set $x_i^\la = -\frac{
		 a_{\la_i+n-i+1} }{1+\beta a_{\la_i+n-i+1}}$. 
			For a Young diagram $\nu$, let
			$$Y(\nu) : = \prod_i (1+\beta x_i^\la) (1+\beta a_{\nu_i+n-i}) -1.$$     Then we have 
      \begin{equation}
				\label{eq:J_pieri}
				\frac{J_\mu(x^\la \mid a) }{  J_\la(x^\la \mid a)}= \sum_{T \in \ssp\SSYT(\la'/\mu')} \prod_{k=1}^{|\lambda/\mu|} \frac{1}{Y (T[< k])},
			\end{equation}
			where the sum is over all \textnormal{SSYT} of shape $\la'/\mu'$, and
			$T[<k]=\nu$ means that the shape 
			$\lambda/\nu$ is filled with entries $\ge k$. 
			By agreement, $T[<1]=\mu$.
 \end{theorem}
 When $\mu=\varnothing$, the RHS
 of \eqref{eq:J_pieri}
 is a sum over $\SSYT(\la')$, and the LHS is the product
$$  \frac{J_\varnothing (x^\la \mid a)}{J_\la(x^\la \mid a)} = \prod_{i=1}^n \frac{ (1+\beta a_{\la_i+n-i+1})^{n-i}(1+\beta a_i)^{n-i}}{\prod_{j=1}^{n-i}(a_j - a_{\la_i+n-i+1})}.$$
 
To see excited diagrams 
in the left-hand side of \eqref{eq:J_pieri},
let $z_i = -\frac{a_i}{1+\beta a_i}$. One can check that
\begin{equation*}
	J_\mu (x \mid a) = \prod_{i=1}^n
	\prod_{j=1}^{\mu_i+n-i}
	\frac{1}{1+\beta a_j} F_\mu(x \mid z),
\end{equation*}
where $F_\mu(x \mid z)$ is the factorial Schur function from
\Cref{sec:contour_integrals}.
Then $x^\la = z^\la$, and we 
can rewrite \eqref{eq:J_pieri} 
in terms of excited diagrams:

\begin{theorem}\label{thm:ssyt_excited}
	Let $x_1,x_2,\ldots, y_1,y_2,\ldots$ be two sets of
	indeterminates, and set 
	\begin{equation*}
		a_{\la_i+n-i+1} =
		-\frac{x_i}{1+\beta x_i},\qquad  a_{\ell_j} =
		-\frac{y_j}{1+\beta y_j},
	\end{equation*}
	where $\ell =
	[1,\ldots,n+\la_1-1] \setminus \{\la_j+n-j+1\colon 1\le j\le n\}$.
	With this notation, we have
    \begin{equation*}
        \sum_{D \in \mathcal{E}(\la/\mu)} \prod_{(i,j) \in \la \setminus D} (x_i-y_j) = 
				\prod_{i=1}^n 
				\prod_{j=1}^{\mu_i+n-i} \frac{ (a_j - a_{\la_i+n-i+1})(1+\beta a_j) }{(1+\beta a_{\la_i+n-i+1} )} \sum_{T \in \ssp\SSYT(\la'/\mu')} \prod_{k=1}^{|\lambda/\mu|} \frac{1}{Y (T[< k])}.
    \end{equation*}
\end{theorem}

\begin{remark}
	Observe that $Y(\nu) = \beta \sum_{i=1}^n (x_i^\la +
	a_{\nu_i+n-i}) + O(\beta^2) $. If we let $\beta \to 0$,
	and perform cancelations with the factors $a_j-a_r$ which
	are of the form $ \beta(y-x)$, the surviving terms above
	would be the ones where $T$ has a maximal number of
	different entries, so it is an SYT. This recovers the
	original formula of \Cref{thm:general}. 
	We do not observe any substitutions that directly connect
	the formula in Theorem~\ref{thm:ssyt_excited} to the
	expression in \cite[Theorem 9.3]{morales2023minimal}.
\end{remark}

\section{Skew hook-length formula from Macdonald polynomials}
\label{sub:Macdonald_hook_length}

Here we consider 
the example of interpolation Macdonald polynomials
\cite{knop1997,
knop1997symmetric,
sahi1996interpolation,
okounkov_newton_int,
okounkov1998shifted},
and apply the general formalism of \Cref{sec:general_formalism} 
to obtain a ``skew hook-length type'' formula involving summation 
over skew standard Young tableaux.
The discussion in the current \Cref{sub:Macdonald_hook_length} 
does not rely on contour integral or vertex model techniques
of \Cref{sec:contour_integrals,sec:YBE_proof_of_Naruse}, respectively.

We denote the interpolation Macdonald polynomials
by $I_\mu(x_1,\ldots,x_n;q,t )$. 
Note that we work only with symmetric polynomials
and not symmetric functions, so we drop the index $n$ 
(which is fixed)
from the notation $I_{\mu|n}$
used in 
\cite{olshanski2019interpolation}.
Throughout the current \Cref{sub:Macdonald_hook_length}, we assume that $n\ge \ell(\mu)$.

The polynomials $I_\mu$ are inhomogeneous symmetric
polynomials of degree $|\mu|$
whose top degree homogeneous part is the 
Macdonald symmetric polynomial $P_\mu(x_1,x_2,\ldots,x_n;q,t)$ \cite[Ch.~VI]{Macdonald1995}.
The substitution which ensures vanishing properties is
\begin{equation}
	\label{eq:a_lambda_Macdonald}
	\mathsf{x}^{(q,t)}(\lambda)
	=
	\bigl(x^{(q,t)}_1(\lambda),\ldots, x^{(q,t)}_n(\lambda)\bigr)
	\coloneqq
	\bigl( q^{-\lambda_1},q^{-\lambda_2}t,\ldots,q^{-\lambda_n}t^{n-1}  \bigr).
\end{equation}
Note that 
here we use the 
normalization 
from \cite{olshanski2019interpolation}, which means that the substitution
must be as in \eqref{eq:a_lambda_Macdonald} 
(there are other equivalent variants in the literature).

Let us recall the vanishing 
property and a tableau formula for $I_\mu$ \cite{okounkov_newton_int,okounkov1998shifted}.

\begin{proposition}
	\label{prop:vanishing_tableau_formula_Macdonald}
	\begin{enumerate}[\bf1.\/]
		\item 
			We have $I_\mu(\mathsf{x}^{(q,t)}(\lambda);q,t)=0$ unless $\mu\subseteq \lambda$.
		\item 
			The interpolation Macdonald polynomials $I_\mu$ admit the following tableau formula:
			\begin{equation}
				\label{eq:tableau_formula_Macdonald}
				I_\mu(x_1,\ldots,x_n;q,t)=
				\sum_{R\in {RTab}(\mu,n)}
				\psi_{R}(q; t) \prod_{(i,j) \in \mu}  \bigl( x_{R(i,j)} - q^{1-j}t^{R(i,j)+i-2} \bigr),
			\end{equation}
			where the sum is over all reverse semistandard tableaux of shape $\mu$ with values 
			in $\left\{ 1,\ldots,n  \right\}$
			(that is, the values in the tableau must 
			weakly decay along the rows and strictly decay down the columns).
			The coefficients $\psi_{R}(q; t)$
			(where we view $R$ as a sequence of horizontal strips)
			are rational functions in 
			$q,t$ given in \cite[Ch.~VI, (6.24)(ii) and (7.11')]{Macdonald1995}.
		\item
			We have
			\begin{equation}
				\label{eq:I_lambda_lambda_Macdonald}
				I_\lambda(\mathsf{x}^{(q,t)}(\lambda);q,t)=
				\prod_{(i,j)\in \lambda}
				\bigl(
					q^{-\lambda_i} t^{i-1} - q^{1-j} t^{\lambda_j'-1}
				\bigr)
				,
			\end{equation}
			where $\lambda'$ is the transposed Young diagram of $\lambda$.
	\end{enumerate}
\end{proposition}

Formula \eqref{eq:I_lambda_lambda_Macdonald} follows from
\eqref{eq:tableau_formula_Macdonald} since 
for $x=\mathsf{x}^{(q,t)}(\lambda)$,
there is a unique reverse tableau
$R(i,j)=\lambda_j'-i+1$
contributing a nonzero term to the sum, and for it we have $\psi_{R}(q; t)=1$.

\medskip

A (one-box) Pieri formula for 
the (non-specialized) interpolation polynomials
$I_\mu$ has 
the form
\begin{equation}
	\label{eq:Pieri_Macdonald}
	I_\mu
	(x_1,\ldots,x_n;q,t)
	\cdot
	\sum_{i=1}^{n}
	\bigl(x_i -q^{-\mu_i}t^{i-1} \bigr)
	=
	\sum_{\nu=\mu+\square}
	\varphi_{\nu/\mu}(q;t)	
	\ssp
	I_\nu(x_1,\ldots,x_n;q,t),
\end{equation}
where
$\varphi_{\nu/\mu}(q;t)$ are the rational functions 
in $q,t$
given in \cite[Ch.~VI, (6.24)(i)]{Macdonald1995}.
Identity \eqref{eq:Pieri_Macdonald} 
follows by comparing the degrees and top homogeneous
components in both sides, and using the uniqueness of interpolation.

\begin{remark}
	The Pieri rule can be generalized to a skew Cauchy type identity
	(also sometimes called Pieri rule)
	involving summation over horizontal strips
	\cite[Lemmas~5.5~and~5.9]{olshanski2019interpolation}:
	\begin{equation*}
			\begin{split}
			&
			I_\mu
			(x_1,\ldots,x_n;q,t)
			\cdot\prod_{i=1}^{n}
			\frac{(x_iyt;q)_\infty}{(x_iy;q)_\infty}
			\\&\hspace{40pt}=
			\sum_{\nu=\mu+\textnormal{horizontal strip}} I_\nu(x_1,\ldots,x_n;q,t)
			\cdot
			\varphi_{\nu/\mu}(q;t)\, y^{|\nu|-\mu}\prod_{i=1}^{n}
			\frac{(yq^{-\mu_i}t^i;q)_\infty}{(yq^{-\nu_i}t^{i-1};q)_\infty}.
		\end{split}
	\end{equation*}
\end{remark}

Applying \Cref{prop:general_skew_SYT_sum}
together with the properties of the interpolation Macdonald polynomials
in \Cref{prop:vanishing_tableau_formula_Macdonald},
we immediately obtain the following skew hook-length type formula:

\begin{proposition}[Skew hook-length type formula with Macdonald parameters]
	For any $\mu\subseteq\lambda$, we have
	\begin{equation*}
		\begin{split}
			&
			\sum_{T\in \ssp\SYT(\lambda/\mu)}
			\varphi_{T}(q;t)
			\prod_{k=1}^{|\lambda/\mu|}
			\biggl(\ssp\sum_{i=1}^{\ell(\lambda)}
				t^{i-1}\Bigl(
					q^{-\lambda_i}-q^{-T^{-1}[< k]_i}
				\Bigr)
			\biggr)^{-1}
			\\&\hspace{20pt}=
			\prod_{(i,j)\in \lambda}
			\bigl(
				q^{-\lambda_i} t^{i-1} - q^{1-j} t^{\lambda_j'-1}
			\bigr)^{-1}
			\sum_{R\in {RTab}(\mu,\ell(\lambda))}
			\psi_{R}(q; t) \prod_{(i,j) \in \mu}  t^{R_{i,j}-1}\bigl( q^{-\lambda_{R(i,j)}} - q^{1-j}t^{i-1} \bigr),
		\end{split}
	\end{equation*}
	where the left-hand sum is over skew standard Young tableaux $T$ of shape $\lambda/\mu$,
	and the right-hand side sum is over reverse semistandard tableaux $R$ of shape $\mu$
	with entries in $\left\{ 1,\ldots,\ell(\lambda)  \right\}$.
\end{proposition}

\bibliographystyle{alpha}
\bibliography{bib}

\medskip

\textsc{G. Panova, University of Southern California, Los Angeles, CA, USA}

E-mail: \texttt{gpanova@usc.edu}

\medskip

\textsc{L. Petrov, University of Virginia, Charlottesville, VA, USA}

E-mail: \texttt{lenia.petrov@gmail.com}

\end{document}